\documentclass[a4paper,oneside,onecolumn]{article}
\usepackage{cancel}
\usepackage{stmaryrd} 

\usepackage{amssymb} 
\usepackage{soul}
\usepackage[T1]{fontenc} 
\usepackage{textcomp} 
\usepackage{mathtools} 
\usepackage{bbm} 

\usepackage[left=2.0cm,top=2.5cm,right=2.0cm,bottom=2.5cm,bindingoffset=0.0cm]{geometry}

\usepackage{float} 
\usepackage{multirow}

\usepackage[utf8]{inputenc} 

\usepackage{dsfont}
\usepackage{graphics,graphicx,epstopdf} 

\usepackage{enumitem} 
\usepackage{mathrsfs}

\usepackage{subfig} 

\usepackage{setspace} 
\usepackage{amsmath,amssymb,amsthm,amsfonts} 
\usepackage{stmaryrd} 
\allowdisplaybreaks

\usepackage{lineno}                     
\usepackage{hyperref}                   

\usepackage{color, xcolor}

\usepackage{authblk} 

\usepackage{amsmath}              
\usepackage{amsfonts}             
\usepackage{amssymb}              
\usepackage{booktabs}             
\usepackage{caption}              

\theoremstyle{plain} 
\newtheorem{theorem}{Theorem}[section]
\newtheorem{lemma}{Lemma}[section]

\newtheorem{assumption}{Assumption}[section]
\newtheorem{definition}{Definition}[section]

\theoremstyle{remark} 
\newtheorem{remark}{Remark}[section]

\definecolor{c1}{rgb}{0,0,1} 
\definecolor{c2}{rgb}{0,0.5,1} 
\definecolor{c3}{rgb}{0.7,0,0.8} 

\hypersetup{
	pdfauthor={Lok Pati},
	pdfsubject={},
	pdftitle={},
	pdfkeywords={},
	breaklinks = true, 
	linktocpage=true, 
	colorlinks=true,       
	linkcolor={c1},          
	citecolor={c2},        
	filecolor=black,      
	urlcolor={c3},       
}
\numberwithin{equation}{section}

\allowdisplaybreaks
\makeatletter
\def\namedlabel#1#2{\begingroup
    #2%
    \def\@currentlabel{#2}%
    \phantomsection\label{#1}\endgroup
}
\makeatother

\begin{document}

\title{\textbf{Existence, Uniqueness, and Pathwise Regularity for Multidimensional Semilinear SPDEs with Locally Lipschitz Coefficients and Rough Initial Data}}

\author[1]{\textsc{Jitendra Nath Naik}\thanks{\href{mailto:jitendra20232201@iitgoa.ac.in}{jitendra20232201@iitgoa.ac.in}}}
\author[1]{\textsc{Lok Pati Tripathi}\thanks{\href{mailto:lokpati@iitgoa.ac.in}{lokpati@iitgoa.ac.in}}}
\affil[1]{School of Mathematics and Computer Science, Indian Institute of Technology Goa, Goa 403401, India.}

\date{} 

\maketitle 


\begin{abstract}
We study multidimensional semilinear stochastic evolution equations driven by multiplicative noise and subject to rough initial data. The drift and diffusion coefficients are assumed to take values in negative fractional order spaces and may exhibit temporal singularities at the initial time. Our main results establish well-posedness and pathwise spatio-temporal regularity of the solutions when these coefficients are globally Lipschitz, and we prove the existence, uniqueness, and pathwise regularity of maximal local solutions when the coefficients are locally Lipschitz. The local Lipschitz condition is formulated with respect to a specific time-weighted norm. This enables the application of our theoretical results to linear stochastic partial differential equations, as well as to a broad class of models with non-globally Lipschitz nonlinearities. These include the stochastic Burgers, Allen--Cahn, Fisher--KPP, and Burgers--Fisher equations, alongside the Ginzburg--Landau system and models featuring exponential growth, such as the stochastic solid fuel ignition model.  
\end{abstract}

\vspace{1.5em}


\vspace{0.8em}


\vspace{2em} 
\tableofcontents 
\vspace{2em} 

\section{Introduction}\label{sec:introduction}
Let $(\Omega, \mathcal{F}, \mathbb{P})$ be a probability space equipped with a normal filtration $\{\mathcal{F}_t\}_{t \in [0,T]}$, and let $\{W(t)\}_{t \in [0,T]}$ be a $Q$-Wiener process (possibly cylindrical) on a separable Hilbert space $(U, (\cdot, \cdot)_U)$ with respect to $\{\mathcal{F}_t\}_{t \in [0,T]}$, where the covariance operator $Q \colon U \to U$ is linear, bounded, self-adjoint, and positive semidefinite. We investigate the following semilinear stochastic evolution equation (SEE):
\begin{equation}  
\label{eq:SEE}
\begin{aligned}
    dX(t) + A X(t) \, dt &= F(t, X(t)) \, dt + G(t, X(t)) \, dW(t), \quad t \in (0, T], \\
    X(0) &= X_0,
\end{aligned}
\end{equation}
where $A \colon \mathcal{D}(A) \subset \mathbb{H} \to \mathbb{H}$ is a densely defined, self-adjoint, and positive definite linear operator with a compact inverse. Our framework allows for singular initial data $X_0$, i.e., $X_0\in L^0(\Omega, \mathcal{F}_0; \dot{\mathbb{H}}^{\mu})$ possibly with $\mu<0$, as well as nonlinear drift and diffusion coefficients $F$ and $G$ that are locally Lipschitz and may possess singularities at the initial time. The precise hypotheses are specified in Section~\ref{sec:hypothesis_and_main_results}.

The existence and uniqueness of mild solutions to SEEs with coefficients satisfying global Lipschitz and linear growth conditions are well understood, as detailed in standard monographs (e.g., \cite{MR1840644, MR2329435, MR2856611, MR3154916, MR3236753, MR3308418, MR3410409}). In the Banach space setting, the theory has been extensively developed using the factorization method \cite{MR920798, MR1154532, MR1488138} and stochastic integration in UMD spaces via $\gamma$-radonifying operators \cite{MR2051026, MR2330977, MR2433958, MR2952092, MR2982717, MR4952170}. In the Hilbert space setting, Jentzen and R\"{o}ckner~\cite{MR2852200}, Kruse and Larsson~\cite{MR2968672}, and Hong and Liu~\cite{MR3912731} established higher spatial and temporal regularity by exploiting the smoothing properties of the analytic $\mathcal{C}_0$-semigroup. To handle rough initial states, such as signed Borel measures, Chen and Dalang~\cite{MR3255231, MR3433576} established existence, uniqueness, and H\"{o}lder regularity for specific one-dimensional models. Recently, Andersson et al.~\cite{MR4172830} introduced an abstract framework with global Lipschitz coefficients taking values in negative fractional order spaces, which is capable of treating rough initial data by allowing for temporal singularities. However, a major limitation of frameworks restricted to globally Lipschitz coefficients is their inability to accommodate fundamental physical models featuring convective terms or polynomial and exponential growth nonlinearities.

To address standard models with non-globally Lipschitz nonlinearities, the existence and uniqueness of solutions for SEEs are traditionally established using the variational approach based on local monotonicity and generalized coercivity (e.g., \cite{MR2719279, MR2990049, MR3053478, MR3158475, MR3410409}), a framework recently extended to critical variational settings by Agresti and Veraar~\cite{MR4716343}. Alternatively, the semigroup approach yields maximal local solutions and sharp temporal H\"{o}lder continuity in fractional order spaces. Within the semigroup framework, the existence and uniqueness of maximal local solutions for locally Lipschitz coefficients were established by Seidler~\cite{MR1213834} in Hilbert spaces and by Brze\'{z}niak~\cite{MR1488138} in M-type 2 Banach spaces. This theory was subsequently extended to UMD Banach spaces by van Neerven et al.~\cite{MR2433958} and further advanced via stochastic maximal $L^p$-regularity by van Neerven et al.~\cite{MR2982717} and Agresti and Veraar~\cite{MR4459102} (see also the comprehensive survey~\cite{MR4952170}). Furthermore, for SEEs driven by time-dependent linear operators, Veraar~\cite{MR2602928} established the existence of local mild
solutions in UMD Banach spaces, complementing the earlier foundational Hilbert space results in \cite{MR1213834}. 

In a recent work, Foondun et al.~\cite{foondun2026local} established local well-posedness for one-dimensional SPDEs driven by multiplicative space-time white noise, featuring a bounded diffusion coefficient and a drift satisfying an $L \log L$ growth condition with $L^2$ initial data. However, their assumptions (see Assumption~1.1 in \cite{foondun2026local}) preclude models with unbounded diffusion coefficients, such as the equations examined in Subsections~\ref{subsec:PAM}--\ref{subsec:SBFE} of the present article, where we establish the existence and uniqueness of maximal local mild solutions. More recently, Agresti et al.~\cite{agresti2026optimallocaltheoryreactiondiffusion} established a comprehensive theory for the existence and uniqueness of maximal local solutions to SPDEs. Their framework accommodates rough initial data in Sobolev-Besov spaces and nonlinear drift and diffusion coefficients with polynomially growing Lipschitz constants. Although their assumptions cover a broad class of problems, they do not extend to models where the nonlinear drift exhibits an exponentially growing Lipschitz constant. In contrast, our approach and hypotheses differ from those in \cite{agresti2026optimallocaltheoryreactiondiffusion}, allowing us to encompass these additional problems involving exponential growth.

In this article, working within a Hilbert space setting, we establish the existence, uniqueness, and pathwise regularity of semilinear SPDEs with both globally and locally Lipschitz coefficients, subject to rough initial data. Our framework accommodates nonlinear drifts featuring polynomial growth (e.g., $f(u) = u - u^3$), exponential growth (e.g., $f(t,u) = e^{t u}$), and convective terms (e.g., $f(u) = \nabla \cdot (u \otimes u)$) alongside unbounded diffusion coefficients, driven by either space-time white noise or (weighted) trace-class noise. We achieve this by decoupling the initial data regularity parameter $\mu$ from the fractional space index $\nu$ characterizing the domains of the operators $F$ and $G$. This gives us the flexibility to formulate our local Lipschitz conditions with respect to a specific time-weighted norm $\|\cdot\|_{\lambda,t}$. Crucially, this approach provides the analytical tools required to incorporate rough initial data (i.e., data in negative fractional order spaces, $\mu < 0$) across a broad class of physical models, including those driven by space-time white noise in one spatial dimension. To justify the negative fractional order space mappings in our framework, we provide fractional convolution estimates (see Lemma~\ref{lem:det_smoothing}) that extend to the full real line ($\eta, \rho \in \mathbb{R}$) and recover the exact H\"{o}lder exponent at the parameter threshold $\eta = \rho$, eliminating the loss of regularity that occurred in previous literature (e.g., Hong and Liu~\cite[Proposition 4.1]{MR3912731}). Consequently, our approach recovers the main result (Theorem~2.9) of Andersson et al.~\cite{MR4172830} as a special case (see Remark~\ref{rem:global_case_comparison}(i)), while extending the theory to a broader class of initial conditions and pathwise regularity properties (see Theorem~\ref{thm:main_global_existence} and Theorems~\ref{thm:pathwise_regularity}--\ref{thm:general_initial_value}). Similarly, we recover Seidler's global existence result~\cite[Theorem~1.3]{MR1213834} as a special case and extend his maximal local existence result~\cite[Theorem~1.5]{MR1213834} to accommodate rough initial data (see Remarks~\ref{rem:global_case_comparison}(iii) and~\ref{rem:seidler_local}). Furthermore, these theoretical results provide the mathematical foundation necessary to support numerical approximations for this broader class of semilinear SEEs; see, for instance, \cite{naik2026optimalerrorestimatesfinite, naik2026sharperrorestimatesfully} for subsequent developments in numerical schemes. The flexibility of our framework is illustrated through applications to the following standard SPDEs:
\begin{itemize}
        \item \textbf{Linear models:} For linear equations in dimension $d=1$ driven by space-time white noise, our approach accommodates singular initial states. Specifically, we incorporate the Dirac measure for the parabolic Anderson model and even rougher distributions, such as the derivative of the Dirac measure, for the stochastic convection-diffusion equation with additive noise. Furthermore, our framework extends to spatial dimensions $d \in \{2, 3\}$, where the convection-diffusion equation driven by trace-class additive noise accommodates the Dirac measure (see Subsections~\ref{subsec:convection_diffusion} and \ref{subsec:PAM}).

         \item \textbf{Models with exponential growth:} For the stochastic solid fuel ignition model in dimension $d=1$, our framework accommodates nonsmooth initial data (i.e., $\mu >0$) under (weighted) trace-class multiplicative noise, and singular initial states, such as the Dirac measure, under (weighted) trace-class additive noise (see Subsection~\ref{subsec:SPDEs_with_exponential_growth}).

        \item \textbf{Models with polynomial and convective nonlinearities:} For the stochastic Allen--Cahn, Fisher--KPP, Burgers, and Burgers--Fisher models driven by space-time white noise, we establish the existence and uniqueness of maximal local mild solutions for rough initial data (i.e., $\mu <0$) in dimension $d=1$ (see Subsections~\ref{subsec:SACE}-\ref{subsec:SBFE}). Furthermore, in spatial dimensions $d \in \{2, 3\}$ under (weighted) trace-class noise, our framework accommodates rough initial data for the Fisher--KPP equation. Conversely, the parameter constraints for the Allen--Cahn, Burgers, and Burgers--Fisher models in dimensions $d \in \{2, 3\}$ dictate positive initial data regularity ($\mu > 0$) (see Subsections~\ref{subsec:SACE}--\ref{subsec:SBFE}).

        \item \textbf{The stochastic Ginzburg--Landau system:} Our framework accommodates rough initial data in dimension $d=1$ under (weighted) trace-class multiplicative noise, while the parameter constraints dictate that both space-time white noise in 1D and any noise in dimensions $d \in \{2, 3\}$ require positive initial data regularity ($\mu > 0$) (see Subsection~\ref{subsec:Ginzburg_Landau_system}).  
\end{itemize}

For simplicity of exposition, we consider a time-independent operator $A$. However, all results in this article can be extended to a time-dependent operator $A(t)$ by employing an evolution system under hypotheses \((E)\) and \((P)\), as stated in \cite{MR1213834}. The remainder of this article is organized as follows. Section~\ref{sec:hypothesis_and_main_results} introduces the preliminaries, states the core assumptions, and presents the main results. Section~\ref{sec:proof_of_main_results} contains the proofs of these theorems. Finally, Section~\ref{sec:examples} applies our theoretical results to several standard physical models.

\section{Mathematical framework and main results}\label{sec:hypothesis_and_main_results}
This section is devoted to the rigorous setup of our framework and the statement of our main results. We outline the required preliminaries in Subsection~\ref{subsec:Preliminaries}, establish the global well-posedness and pathwise regularity for globally Lipschitz coefficients in Subsection~\ref{subsec:Globally Lipschitz nonlinearities}, and present the existence of a unique maximal local solution together with a blow-up alternative result for locally Lipschitz nonlinearities in Subsection~\ref{subsec:Locally Lipschitz nonlinearities}.
\subsection{Preliminaries}\label{subsec:Preliminaries}

\noindent\textbf{Notation:} Throughout this article, the letter $C$ represents a generic positive constant whose exact value may change from line to line; its dependence on specific parameters will be explicitly indicated when essential. For $x \in \mathbb{R}$, we define the positive part as $x^+ \coloneqq \max\{0, x\}$. We denote by $\mathcal{P}_T$ the predictable $\sigma$-algebra on $\Omega_T \coloneqq [0,T] \times \Omega$. Finally, the symbol $\otimes$ is used in two contexts: it denotes the product $\sigma$-algebra when applied to collections of sets (e.g., $\mathcal{P}_T \otimes \mathcal{B}(\dot{\mathbb{H}}^{\mu})$), and it denotes the tensor product of vectors (e.g., $u \otimes v$) for convective nonlinearities.

Let $(H, (\cdot, \cdot)_H)$ be a separable Hilbert space. To model the underlying multi-dimensional systems, we introduce the product space $\mathbb{H} \coloneqq H^N$ for $N \ge 1$. The space $\mathbb{H}$ inherits a separable Hilbert space structure when equipped with the natural inner product $(u, v)_{\mathbb{H}} \coloneqq \sum_{i=1}^N (u_i, v_i)_H$, with the induced norm denoted by $\|\cdot\|_{\mathbb{H}}$. The space of bounded linear operators on $\mathbb{H}$ is denoted by $\mathcal{L}(\mathbb{H})$.

\begin{assumption}\label{ass:operator_A}
    As introduced in Section~\ref{sec:introduction}, we consider a densely defined, self-adjoint, and positive definite linear operator $A \colon \mathcal{D}(A) \subset \mathbb{H} \to \mathbb{H}$ with a compact inverse. 
\end{assumption}
 \noindent Under this assumption, $-A$ generates an analytic $\mathcal{C}_0$-semigroup of contractions $\{S(t)\}_{t\ge 0}$ on $\mathbb{H}$ given by $S(t) = e^{-tA}$. Applying the spectral theorem to $A^{-1}$ yields an orthonormal basis of eigenvectors for $A$, which allows us to define the fractional powers $A^\gamma$ for $\gamma \in \mathbb{R}$ (see, e.g., Kruse~\cite[Appendix~B.2]{MR3154916}). The associated fractional spaces are defined as $\dot{\mathbb{H}}^\gamma \coloneqq \mathcal{D}(A^{\gamma/2})$, equipped with the norm $\|u\|_{\dot{\mathbb{H}}^\gamma} \coloneqq \|A^{\gamma/2}u\|_{\mathbb{H}}$.

The following lemma collects several smoothing properties of the analytic $\mathcal{C}_0$-semigroup that will be essential throughout the article.

\begin{lemma}[Semigroup estimates]\label{lem:semigroup_estimates}
Let the operator $A$ and the corresponding analytic $\mathcal{C}_0$-semigroup $\{S(t)\}_{t\ge 0}$ be as defined above. Then the following estimates hold:
\begin{enumerate}[label={\upshape(\roman*)}]
    \item \label{est:semigroup_smoothing} For any $\gamma \in \mathbb{R}$, there exists a constant $C > 0$, depending on $\gamma$, such that
    \[
    \|A^{\gamma}S(t)\|_{\mathcal{L}(\mathbb{H})} \le C\, t^{-\gamma^+}, \quad t > 0.
    \]
    
    \item \label{est:semigroup_continuity_zero} For any $\rho \in [0,1]$, there exists a constant $C > 0$, depending on $\rho$, such that
    \[
    \|A^{-\rho}(I-S(t))\|_{\mathcal{L}(\mathbb{H})} \le C\, t^{\rho}, \quad t > 0.
    \]
    
    \item \label{est:semigroup_continuity_t} For any $\gamma \in \mathbb{R}$ and $\rho \in [0,1]$, there exists a constant $C > 0$, depending on $\gamma$ and $\rho$, such that
    \[
    \|A^{\gamma}(S(t)-S(s))\|_{\mathcal{L}(\mathbb{H})} \le C\, (t-s)^{\rho}\, s^{-(\gamma+\rho)^+}, \quad 0 < s < t.
    \]
\end{enumerate}
\end{lemma}

\begin{proof}
(i) Let $\gamma \ge 0$. The estimate $\|A^\gamma S(t)\|_{\mathcal{L}(\mathbb{H})} \le C t^{-\gamma}$ is a standard smoothing property of analytic semigroups (see, for instance, Kruse~\cite[Lemma~B.9]{MR3154916}). For $\gamma < 0$, the fractional power $A^\gamma$ is a bounded operator on $\mathbb{H}$ because $A$ is positive definite and possesses a bounded inverse. Utilizing the uniform boundedness of the semigroup $S(t)$, we obtain
\[
    \|A^\gamma S(t)\|_{\mathcal{L}(\mathbb{H})} \le \|A^\gamma\|_{\mathcal{L}(\mathbb{H})}\, \|S(t)\|_{\mathcal{L}(\mathbb{H})} \le C.
\]
Combining these two cases yields the bound $\|A^\gamma S(t)\|_{\mathcal{L}(\mathbb{H})} \le C\, t^{-\gamma^+}$.

\medskip
\noindent (ii) For the proof of this estimate, we refer the reader to Kruse~\cite[Lemma~B.9]{MR3154916}.

\medskip
\noindent (iii) By the semigroup property, we have $S(t)-S(s) = S(s)\big(S(t-s)-I\big)$. Applying  $A^\gamma$ and exploiting the fact that $A$ and its fractional powers commute with the semigroup, we deduce
\[
    A^\gamma (S(t)-S(s)) = A^{\gamma+\rho} S(s) A^{-\rho}(S(t-s)-I).
\]
Taking the operator norm in $\mathcal{L}(\mathbb{H})$ and applying the bounds established in (i) and (ii), we arrive at
\begin{align*}
    \|A^\gamma (S(t)-S(s))\|_{\mathcal{L}(\mathbb{H})} 
    &\le \|A^{\gamma+\rho} S(s)\|_{\mathcal{L}(\mathbb{H})}\, \|A^{-\rho}(S(t-s)-I)\|_{\mathcal{L}(\mathbb{H})} \\
    &\le C\, s^{-(\gamma+\rho)^+}\, (t-s)^\rho.
\end{align*}
This concludes the proof.
\end{proof}

For a given Banach space $(\mathcal{X}, \|\cdot\|_{\mathcal{X}})$, let $\mathcal{C}_b((a,b]; \mathcal{X})$ denote the Banach space of bounded continuous functions from $(a,b]$ to $\mathcal{X}$, endowed with the norm $\|\phi\|_{\mathcal{C}_b((a,b]; \mathcal{X})} \coloneqq \sup_{t\in(a,b]}\|\phi(t)\|_{\mathcal{X}}$. Similarly, $\mathcal{C}([a,b]; \mathcal{X})$ and $\mathcal{C}^{0,\delta}([a,b]; \mathcal{X})$ denote the Banach spaces of continuous and $\delta$-H\"{o}lder continuous functions from $[a,b]$ to $\mathcal{X}$, endowed with the standard norms $\|\phi\|_{\mathcal{C}([a,b]; \mathcal{X})} \coloneqq \sup_{t\in[a,b]}\|\phi(t)\|_{\mathcal{X}}$ and $\|\phi\|_{\mathcal{C}^{0,\delta}([a,b]; \mathcal{X})} \coloneqq\sup_{t\in[a,b]}\|\phi(t)\|_{\mathcal{X}} + \sup_{s\neq t} \frac{\|\phi(t) - \phi(s)\|_{\mathcal{X}}}{|t-s|^\delta}$, respectively. We also use the little H\"{o}lder space
\[
\mathring{\mathcal C}^{0,\delta}([a,b];\mathcal X) 
\coloneqq 
\left\{ 
\phi\in\mathcal C^{0,\delta}([a,b];\mathcal X): \lim_{h\searrow 0} \sup_{0<|t-s|<h} \frac{\|\phi(t)-\phi(s)\|_{\mathcal X}}{|t-s|^\delta} =0 
\right\},
\]
which is a closed subspace of $\mathcal C^{0,\delta}([a,b];\mathcal X)$. Crucially, $\mathring{\mathcal{C}}^{0, \delta}([a, b]; \mathcal{X})$ inherits separability whenever the underlying space $\mathcal{X}$ is separable.

Let $(\mathcal{S}, \mathcal{A}, \mathfrak{m})$ be a finite measure space. For $1\leq p<\infty$, we denote by $L^p(\mathcal{S};\mathcal{X})$ the Banach space of strongly measurable functions $f:\mathcal{S}\to\mathcal{X}$ such that $\|f\|_{L^p(\mathcal{S};\mathcal{X})}\coloneqq \big(\int_{\mathcal{S}}\|f\|_{\mathcal{X}}^p\,d\mathfrak{m}\big)^{1/p}<\infty$. Similarly, $L^0(\mathcal{S}; \mathcal{X})$ denotes the linear space of strongly measurable functions $f:\mathcal{S}\to\mathcal{X}$, where two functions are identified whenever they are equal $\mathfrak{m}$-almost everywhere. Whenever it is necessary to explicitly highlight the underlying $\sigma$-algebra $\mathcal{A}$, these spaces are written as $L^p(\mathcal{S},\mathcal{A}; \mathcal{X})$ and $L^0(\mathcal{S},\mathcal{A}; \mathcal{X})$.

 For a well-defined theory of stochastic integration in this infinite-dimensional setting, we follow the framework of Pr\'{e}v\^{o}t and R\"{o}ckner~\cite{MR2329435} and introduce the separable Hilbert space $U_0 \coloneqq Q^{1/2}(U)$, endowed with the inner product
$$
(u_0, v_0)_{U_0} \coloneqq (Q^{-1/2}u_0, Q^{-1/2}v_0)_U, \quad u_0, v_0 \in U_0,
$$
where $Q^{-1/2}$ denotes the pseudoinverse of $Q^{1/2}$ if $Q$ is not injective. Furthermore, we denote by $\mathrm{HS}(U_0, \mathbb{H})$ the space of Hilbert--Schmidt operators $\Phi: U_0 \to \mathbb{H}$, which forms a separable Hilbert space under the norm
$$
\|\Phi\|_{\mathrm{HS}(U_0, \mathbb{H})} \coloneqq \left( \sum_{m=1}^\infty \|\Phi \psi_m\|_{\mathbb{H}}^2 \right)^{1/2},
$$
where $\{\psi_m\}_{m \ge 1}$ is an arbitrary orthonormal basis of $U_0$.

The following lemma provides a Burkholder--Davis--Gundy type inequality for $\mathbb{H}$-valued stochastic integrals, which will be used repeatedly to estimate the moments of stochastic convolutions.

\begin{lemma}[{\cite[Proposition~2.12]{MR3154916}}]
\label{lem:BDG}
Let $p \ge 2$ and $0 \le t_1 < t_2 \le T$. Let $\Phi \colon [0,T]\times\Omega \to \mathrm{HS}(U_0, \mathbb{H})$ be a predictable process such that
\[
    \mathbb{E}\bigg[ \bigg( \int_{t_1}^{t_2} \|\Phi(\sigma)\|_{\mathrm{HS}(U_0, \mathbb{H})}^2\, d\sigma \bigg)^{p/2} \bigg] < \infty.
\]
Then the stochastic integral $\int_{t_1}^{t_2} \Phi(\sigma)\, dW(\sigma)$ is well-defined and satisfies
\[
    \mathbb{E}\bigg[ \bigg\| \int_{t_1}^{t_2} \Phi(\sigma)\, dW(\sigma) \bigg\|_{\mathbb{H}}^p \bigg]
    \le C_p\, \mathbb{E}\bigg[ \bigg( \int_{t_1}^{t_2} \|\Phi(\sigma)\|_{\mathrm{HS}(U_0, \mathbb{H})}^2\, d\sigma \bigg)^{p/2} \bigg],
\]
where the constant $C_p$ is given by
\[
    C_p = \bigg(\frac{p(p-1)}{2}\bigg)^{p/2} \bigg(\frac{p}{p-1}\bigg)^{p\big(\frac{p}{2}-1\big)}.
\] 
\end{lemma}

We impose the following assumptions on the parameters and the nonlinear coefficients of the problem.

\begin{assumption}[Conditions on parameters]\label{ass:parameters}
Throughout the article, the parameters \(\mu, \nu, \alpha, \beta, \hat{\alpha}, \hat{\beta} \in \mathbb{R}\) are assumed to satisfy
\begin{equation}\label{eq:parameter_constraints}
\begin{gathered}
0 \le \nu-\mu < 1, \qquad
\alpha < 2-\nu, \qquad
\beta < 1-\nu, \\[0.3em]
0 \le \hat{\alpha} < 1-\frac{1}{2}\max\{\nu-\mu,\alpha+\nu\},
\qquad
0 \le \hat{\beta} < \frac{1}{2}-\frac{1}{2}\max\{\nu-\mu,\beta+\nu\}.
\end{gathered}
\end{equation}
\end{assumption}

\begin{assumption}[Measurability of drift and diffusion terms] \label{ass:measurability_drift_diff}
The mapping $F \colon [0,T]\times\Omega \times \dot{\mathbb{H}}^{\mu} \to \dot{\mathbb{H}}^{-\alpha}$ is measurable from the measurable space $(\Omega_T \times \dot{\mathbb{H}}^{\mu}, \mathcal{P}_T \otimes \mathcal{B}(\dot{\mathbb{H}}^{\mu}))$ to $(\dot{\mathbb{H}}^{-\alpha}, \mathcal{B}(\dot{\mathbb{H}}^{-\alpha}))$. Similarly, the mapping $G \colon [0,T]\times\Omega \times \dot{\mathbb{H}}^{\mu} \to \mathrm{HS}(U_0,\dot{\mathbb{H}}^{-\beta})$ is measurable from $(\Omega_T \times \dot{\mathbb{H}}^{\mu}, \mathcal{P}_T \otimes \mathcal{B}(\dot{\mathbb{H}}^{\mu}))$ to $(\mathrm{HS}(U_0, \dot{\mathbb{H}}^{-\beta}), \mathcal{B}(\mathrm{HS}(U_0, \dot{\mathbb{H}}^{-\beta})))$. Here, $\mathcal{B}(\mathcal{Y})$ denotes the Borel $\sigma$-algebra on a topological space $\mathcal{Y}$.
\end{assumption}

\subsection{Globally Lipschitz nonlinearities}\label{subsec:Globally Lipschitz nonlinearities}
In this subsection, we provide the well-posedness and (pathwise) regularity results for the problem \eqref{eq:SEE} with globally Lipschitz continuous drift \(F\) and diffusion \(G\) coefficients. We impose the following assumptions on \(F\) and \(G\):

\begin{assumption}[Global Lipschitz conditions]\label{ass:global_lips_nonlin}
Under Assumption~\ref{ass:parameters}, there exist positive constants $L_j$, $j=1,2,3,4$, such that for all $t\in(0,T]$, \(F\) and \(G\) satisfy the following estimates $\mathbb{P}$-almost surely:
\begin{align*}
    &\|F(t,  u) - F(t,  v)\|_{\dot{\mathbb{H}}^{-\alpha}} \le L_1 \, t^{-\hat{\alpha}} \|u - v\|_{\dot{\mathbb{H}}^{\nu}}\;\forall\,u, v \in \dot{\mathbb{H}}^{\nu},\quad \|F(t, 0)\|_{\dot{\mathbb{H}}^{-\alpha}} \le L_2 \, t^{-\hat{\alpha}},  \\
    &\|G(t,  u) - G(t, v)\|_{\mathrm{HS}(U_0, \dot{\mathbb{H}}^{-\beta})} \le L_3 \, t^{-\hat{\beta}} \|u - v\|_{\dot{\mathbb{H}}^{\nu}}\;\forall u, v \in \dot{\mathbb{H}}^{\nu}, \quad \|G(t, 0)\|_{\mathrm{HS}(U_0, \dot{\mathbb{H}}^{-\beta})} \le L_4 \, t^{-\hat{\beta}}.
\end{align*}
\end{assumption}

\begin{remark}
We note that structural assumptions of a similar type on the nonlinear coefficients $F$ and $G$ have also been considered in the literature, notably by Andersson et al.~\cite{MR4172830}.
\end{remark}

\begin{definition}[Mild solution]\label{def:mild_solution}
Let \(\mu\in\mathbb R\) and \(X_0\in L^0(\Omega,\mathcal F_0;\dot{\mathbb H}^{\mu})\). Assume that the measurability condition in Assumption~\ref{ass:measurability_drift_diff} is satisfied by \(F\) and \(G\). An $\dot{\mathbb{H}}^{\mu}$-valued predictable stochastic process $\{X(t)\}_{t \in [0,T]}$ is called a mild solution to \eqref{eq:SEE} if, for every $t \in [0,T]$, the following conditions hold $\mathbb{P}$-almost surely:
\begin{enumerate}[label={\upshape(\roman*)}]
    \item \label{cond:mild_F} $\sigma \mapsto S(t-\sigma)F(\sigma, X(\sigma)) \in L^1(0,t; \dot{\mathbb{H}}^\mu)$,
    \item \label{cond:mild_G} $\sigma \mapsto S(t-\sigma)G(\sigma, X(\sigma)) \in L^2(0,t; \mathrm{HS}(U_0, \dot{\mathbb{H}}^{\mu}))$, and
    \item \label{cond:mild_eq} $X(t)$ satisfies the following equality in $\dot{\mathbb{H}}^{\mu}$:
\begin{equation}\label{eq:mild_solution_formula}
    X(t) = S(t)X_0 + \int_0^t S(t-\sigma)F(\sigma, X(\sigma)) \, d\sigma + \int_0^t S(t-\sigma)G(\sigma, X(\sigma)) \, dW(\sigma).
\end{equation}
\end{enumerate}
\end{definition}

Let us introduce the Banach space $\mathbb{V}_p$ of predictable processes, which will serve as the natural state space for our solution. For a fixed $T>0$, $\lambda \coloneqq (\nu-\mu)/2\geq 0 $, and $p \ge 2$, we define the linear space
\begin{align}
    \nonumber&\mathbb{V}_p \coloneqq
\Big\{
\varphi \in \mathcal{C}\big([0,T]; L^p(\Omega; \dot{\mathbb{H}}^{\mu})\big) :\,t^\lambda \varphi\in \mathcal{C}_b\big((0,T]; L^p(\Omega; \dot{\mathbb{H}}^{\nu})\big), \text{ and the process} \\
\nonumber&\hspace{4cm}
\{\varphi(t)\}_{t\in[0,T]} \text{ admits an }\dot{\mathbb{H}}^\mu\text{-valued predictable version}\Big\}
\end{align}
equipped with the norm
\[
\|\varphi\|_{\mathbb{V}_p} \coloneqq \sup_{t \in (0,T]} \Big( \|\varphi(t)\|_{L^p(\Omega; \dot{\mathbb{H}}^{\mu})} + t^{\lambda} \|\varphi(t)\|_{L^p(\Omega; \dot{\mathbb{H}}^{\nu})} \Big).
\]
It is straightforward to verify that $(\mathbb{V}_p,\|\cdot\|_{\mathbb{V}_p})$ is a Banach space. Now, we state our main well-posedness and regularity results. 

\begin{theorem}[Global existence, uniqueness, and regularity]
\label{thm:main_global_existence}
Suppose that Assumption~\ref{ass:operator_A} holds, and that Assumptions~\ref{ass:measurability_drift_diff} and~\ref{ass:global_lips_nonlin} hold for a set of parameters satisfying Assumption~\ref{ass:parameters}. Let $X_0\in L^p(\Omega, \mathcal{F}_0; \dot{\mathbb{H}}^{\mu})$ for some $p \ge 2$. Then the following hold:
\begin{enumerate}[label={\upshape(\roman*)}]
    \item \textbf{Existence and uniqueness:} There exists a unique mild solution $X \in \mathbb{V}_p$ to problem \eqref{eq:SEE}. Moreover, there exists a constant $C>0$, independent of $t$ and $X_0$, such that
    \begin{align}
        \label{eq:a_priori_estimate}
        \|X\|_{\mathbb{V}_p} \le C\big(1 + \|X_0\|_{L^p(\Omega;\dot{\mathbb{H}}^{\mu})}\big).
    \end{align}
    
    \item \textbf{Continuous dependence on the initial data:} For any two initial data $X_0, Y_0 \in L^p(\Omega,\mathcal{F}_0;\dot{\mathbb{H}}^{\mu})$, the corresponding mild solutions $X, Y \in \mathbb{V}_p$ satisfy
    \begin{align}
        \label{eq:cont_dep_ini_data}& \|X - Y\|_{\mathbb{V}_p} \le C\|X_0 - Y_0\|_{L^p(\Omega;\dot{\mathbb{H}}^{\mu})},
    \end{align}
    where $C>0$ is a constant independent of the initial data.
    
    \item \textbf{Temporal regularity:} For any arbitrary $a \in (0, T)$, the mild solution $X$ exhibits local Hölder continuity in time with respect to the $\dot{\mathbb{H}}^{\nu}$-norm. Specifically,
    \[
        X \in \mathcal{C}^{0,\delta}\big([a,T]; L^p(\Omega;\dot{\mathbb{H}}^{\nu})\big), \quad 
        0 < \delta < \min \Big\{\; 1 - \frac{1}{2}(\nu+\alpha)^+, \; \frac{1}{2} \Big( 1 - (\nu+\beta)^+ \Big) \Big\}.
    \]
    Furthermore, if the initial data possesses higher spatial regularity such that $X_0 \in L^p(\Omega; \dot{\mathbb{H}}^{\mu+\varepsilon})$ for some $0 < \varepsilon \le 1 $, then the solution is globally Hölder continuous in time with respect to the $\dot{\mathbb{H}}^{\mu}$-norm:
    \[
        X \in \mathcal{C}^{0,\delta_\varepsilon}\big([0,T]; L^p(\Omega;\dot{\mathbb{H}}^{\mu})\big), 
    \]
    where the Hölder exponent satisfies
    \[
        0 < \delta_\varepsilon < \min \Big\{ \frac{\varepsilon}{2}, \; 1 - \frac{1}{2}(\mu+\alpha)^+ - (\hat{\alpha} + \lambda), \; \frac{1}{2}  - \frac{1}{2}(\mu+\beta)^+ - (\hat{\beta} + \lambda)\Big\}.
    \]
\end{enumerate}
\end{theorem}

\begin{remark}\label{rem:global_case_comparison}
\begin{enumerate}[label=(\roman*)]
    \item \label{rem:andersson_article}
     If the diffusion coefficient $G(t,u)$ is constant with respect to the second variable $u$, a careful inspection of the proof reveals that the parameter restrictions $0\leq \nu-\mu<1$ and $0\leq \hat{\beta} < \frac12-\frac12\max\{\nu-\mu,\beta+\nu\}$ in Assumption~\ref{ass:parameters} can be relaxed to $0\leq \nu-\mu<2$ and $0\leq \hat{\beta} < \frac12-\frac12\max\!\left\{\frac{\nu-\mu}{2},\beta+\nu\right\}$. Consequently, the main result established by Andersson et al.~\cite[Theorem 2.9]{MR4172830} follows as a special case of Theorem \ref{thm:main_global_existence}. Specifically, they consider initial data $X_0\in L^p(\Omega, \mathcal{F}_0; H_{-\max\{\delta, 0\}})$ and globally Lipschitz coefficients on $H$. Under the identification $H_r \equiv \dot{\mathbb{H}}^{2r}$, this setting corresponds to the choices $\nu = 0$, $\mu = -2\max\{\delta, 0\}$, $\alpha = 2\tilde{\alpha}$, and $\beta = 2\tilde{\beta}$ (where $\tilde{\alpha}$ and $\tilde{\beta}$ correspond to the parameters $\alpha$ and $\beta$ therein). By keeping the temporal exponents $\hat{\alpha}$ and $\hat{\beta}$ identical, we obtain $\lambda = \max\{\delta, 0\}$. Under these values, the conditions in Assumption \ref{ass:parameters} simplify directly to the parameter bounds and integrability conditions required in \cite[Theorem 2.9]{MR4172830}.
     \item \label{rem:jentzen_article}
     While the condition $\nu \ge \mu$ prevents recovering the higher spatial and temporal regularity results of Jentzen and R\"{o}ckner~\cite[Theorem 1]{MR2852200}, their existence and uniqueness result for $H$-valued solutions is a special case of Theorem \ref{thm:main_global_existence}. They consider initial data in $H$ and globally Lipschitz coefficients. We recover this setting by choosing $\nu = \mu = \alpha = \beta = \hat{\alpha} = \hat{\beta} = 0$. This yields $\lambda = 0$ and satisfies the parameter conditions in Assumption \ref{ass:parameters}, reducing our solution space $\mathbb{V}_p$ to $\mathcal{C}([0,T]; L^p(\Omega; H))$.
     
     \item \label{rem:seidler_global}
     The global well-posedness result established by Seidler~\cite[Theorem 1.3]{MR1213834} for equations with time-independent linear operators is a special case of Theorem \ref{thm:main_global_existence}. Specifically, Seidler considers globally Lipschitz coefficients in $H_\delta$ (setting $\nu = 2\delta$) that map into $H$ without temporal singularities ($\alpha = \beta = \hat{\alpha} = \hat{\beta} = 0$). For initial data in $H_{\tilde{\alpha}}$ (setting $\mu = 2\tilde{\alpha}$), the required condition is $0 \le \tilde{\alpha} \le \delta$. This yields the time-weighted bound $t^{\delta - \tilde{\alpha}}$, exactly matching our parameter $\lambda = (\nu - \mu)/2$.
\end{enumerate}
\end{remark}

The following result establishes the pathwise spatio-temporal regularity of the mild solution.

\begin{theorem}[Pathwise regularity]
\label{thm:pathwise_regularity}
Suppose that Assumption~\ref{ass:operator_A} holds, and that Assumptions~\ref{ass:measurability_drift_diff} and~\ref{ass:global_lips_nonlin} hold for a set of parameters satisfying Assumption~\ref{ass:parameters}. Let $X_0 \in L^p(\Omega,\mathcal{F}_0; \dot{\mathbb{H}}^\mu)$ with $p > 2$. Then there exist positive constants $C$ and $C_a$ (where $C_a$ may diverge as $a \searrow 0$) such that the mild solution \(X\) to \eqref{eq:SEE} exhibits the following pathwise regularity properties:

  \begin{enumerate}[label={\upshape(\roman*)}]
  \item \textit{\textbf{Continuity:}} For $(\mu+\alpha)^+ < 2 - \frac{2}{p}$ and $(\mu+\beta)^+ < 1 - \frac{2}{p}$, we have $X \in L^p(\Omega; \mathcal{C}([0,T]; \dot{\mathbb{H}}^\mu))$, satisfying the bound: $$\|X\|_{L^p(\Omega;\mathcal{C}([0,T];\dot{\mathbb{H}}^{\mu}))}\;\leq\;C(1+\|X_0\|_{L^p(\Omega;\dot{\mathbb{H}}^{\mu})}).$$
  \item \textit{\textbf{Local H\"{o}lder regularity:}} For any $a > 0$, provided $(\nu+\alpha)^+ < 2 - \frac{2}{p}$, $(\nu+\beta)^+ < 1 - \frac{2}{p}$, and $0 < \delta < \min\big\{1 - \frac{1}{p} - \frac{1}{2}(\nu+\alpha)^+, \frac{1}{2} - \frac{1}{p} - \frac{1}{2}(\nu+\beta)^+\big\},$ we have $X \in L^p(\Omega; \mathring{\mathcal{C}}^{0, \delta}([a, T]; \dot{\mathbb{H}}^\nu))$ with the local estimate: $$\|X\|_{L^p(\Omega;\mathring{\mathcal{C}}^{0, \delta}([a,T];\dot{\mathbb{H}}^{\nu}))}\;\leq\;C_a(1+\|X_0\|_{L^p(\Omega;\dot{\mathbb{H}}^{\mu})}).$$
  \item \textit{\textbf{Global H\"{o}lder regularity:}} If the initial data possesses enhanced regularity $X_0 \in L^p(\Omega; \dot{\mathbb{H}}^{\mu+\epsilon})$ for some $0<\epsilon \le 1$, then for $(\mu+\alpha)^+ < 2 - \frac{2}{p}$, $(\mu+\beta)^+ < 1 - \frac{2}{p}$, and $0 < \delta < \min\big\{\frac{\epsilon}{2}, 1 - \hat{\alpha} - \lambda - \frac{1}{2}(\mu+\alpha)^+, 1 - \frac{1}{p} - \frac{1}{2}(\mu+\alpha)^+, \frac{1}{2} - \hat{\beta} - \lambda - \frac{1}{2}(\mu+\beta)^+, \frac{1}{2} - \frac{1}{p} - \frac{1}{2}(\mu+\beta)^+\big\},$ it holds that $X \in L^p\left(\Omega; \mathring{\mathcal{C}}^{0, \delta}([0, T]; \dot{\mathbb{H}}^\mu)\right)$. Moreover, the corresponding global norm is bounded by $C(1+\|X_0\|_{L^p(\Omega;\dot{\mathbb{H}}^{\mu+\epsilon})}).$
  \end{enumerate}
\end{theorem}

Next, we extend our results to accommodate measurable initial data $X_0\in L^0(\Omega, \mathcal{F}_0; \dot{\mathbb{H}}^{\mu})$.

\begin{theorem}
\label{thm:general_initial_value}
  Suppose that Assumption~\ref{ass:operator_A} holds, and that Assumptions~\ref{ass:measurability_drift_diff} and~\ref{ass:global_lips_nonlin} hold for a set of parameters satisfying Assumption~\ref{ass:parameters}. Let $X_0\in L^0(\Omega, \mathcal{F}_0; \dot{\mathbb{H}}^{\mu})$. Then \eqref{eq:SEE} admits a unique mild solution $X$ whose sample paths satisfy the following regularity results \(\mathbb{P}\)-almost surely: For any \(0<a<T\) and \(0<\delta<\min\big\{1-\frac{1}{2}(\nu+\alpha)^+, \frac{1}{2}-\frac{1}{2}(\nu+\beta)^+\big\}\),
\begin{align}
    \label{eq:general_initial_value_1}&X \in  \mathcal{C}\big([0,T]; \dot{\mathbb{H}}^{\mu}\big) \cap \mathring{\mathcal{C}}^{0,\delta}\big([a,T]; \dot{\mathbb{H}}^{\nu}\big) \;\text{ with  }\;X(0)=X_0,\;\text{ and } \\
    \label{eq:general_initial_value_2}&\lim\limits_{t\searrow 0}t^\lambda\|X(t)\|_{\dot{\mathbb{H}}^\nu}=0\;\text{ for } \lambda \coloneqq (\nu-\mu)/2>0.
\end{align}
\end{theorem}

\subsection{Locally Lipschitz nonlinearities}\label{subsec:Locally Lipschitz nonlinearities}

Assumption~\ref{ass:global_lips_nonlin} is too restrictive for equations with polynomial or convective nonlinearities (e.g., the Allen--Cahn and Burgers equations). To accommodate a broader class of physical models, we extend our analysis to locally Lipschitz coefficients. Specifically, we formulate our locally Lipschitz assumptions on the drift and diffusion coefficients with respect to a time-weighted norm on \(\dot{\mathbb{H}}^{\nu}\), defined as
\begin{equation}
    \|u\|_{\lambda,t} \coloneqq \begin{cases}
        \|u\|_{\dot{\mathbb{H}}^{\mu}} & \text{if } t=0 \text{ or } \lambda=0, \\[4pt]
        \|u\|_{\dot{\mathbb{H}}^{\mu}} + t^{\lambda} \|u\|_{\dot{\mathbb{H}}^{\nu}} & \text{if } 0 < t \leq T \text{ and } \lambda > 0,
    \end{cases}
\label{eq:time_weighted_norm}
\end{equation}
where \(\lambda \coloneqq (\nu-\mu)/2\).

\begin{assumption}[Local Lipschitz  conditions]
\label{ass:local_lip_stoch}
Under Assumption~\ref{ass:parameters}, for each $R > 0$, there exist positive constants $L_j(R),\;j=1,2,3,4,$ such that for $t \in (0,T]$, \(F\) and \(G\) satisfy the following estimates \(\mathbb{P}\)-almost surely:
\begin{align*}
    &\|F(t,u) - F(t,v)\|_{\dot{\mathbb{H}}^{-\alpha}} \le L_1(R)\, t^{-\hat{\alpha}} \|u - v\|_{\lambda,t}, \quad \|F(t,0)\|_{\dot{\mathbb{H}}^{-\alpha}} \le L_2(R)\, t^{-\hat{\alpha}}, \\
    &\|G(t,u) - G(t,v)\|_{\mathrm{HS}(U_0,\dot{\mathbb{H}}^{-\beta})} \le L_3(R)\, t^{-\hat{\beta}} \|u - v\|_{\lambda,t}, \quad \|G(t,0)\|_{\mathrm{HS}(U_0,\dot{\mathbb{H}}^{-\beta})} \le L_4(R)\, t^{-\hat{\beta}},
\end{align*}
for all \(u,v \in\{\phi\in\dot{\mathbb{H}}^{\nu}:\,\|\phi\|_{\lambda, t}\leq R\}\).
\end{assumption}

\begin{definition}[Local mild solution]\label{def:local_mild_solution} Assume that \(X_0\in L^0(\Omega,\mathcal F_0;\dot{\mathbb H}^{\mu})\) and that \(F\) and \(G\) satisfy the measurability conditions of Assumption~\ref{ass:measurability_drift_diff}. A pair $(X, \tau)$ is defined as a local mild solution to \eqref{eq:SEE} if $\tau: \Omega \to [0, T]$ is a stopping time satisfying $\tau>0\;\mathbb{P}$-almost surely, and there exists a localizing sequence of stopping times \(\tau_n\uparrow \tau\;\mathbb{P}\)-almost surely such that for each $n \in \mathbb{N}$, the stopped process $\{X(t \wedge \tau_n)\}_{t \in [0, T]}$ is an $\dot{\mathbb{H}}^{\mu}$-valued predictable stochastic process, and for $t \in [0, T]$, the following hold \(\mathbb{P}\)-almost surely: 
\begin{enumerate}[label={\upshape(\roman*)}]
    \item \(\sigma\mapsto S(t-\sigma)F(\sigma,X(\sigma))\mathbbm{1}_{\llbracket 0,\tau_{n}\rrbracket}(\sigma)\in L^1(0,t;\dot{\mathbb{H}}^\mu)\),
    \item \(\sigma\mapsto S(t-\sigma)G(\sigma,X(\sigma))\mathbbm{1}_{\llbracket 0,\tau_{n}\rrbracket}(\sigma)\in L^2(0,t;\mathrm{HS}(U_0, \dot{\mathbb{H}}^{\mu}))\),
    \item on \(\{t\leq\tau_n\}\), \(X\) satisfies the following equality:
\begin{equation}\label{eq:local_mild_solution_formula}
X(t) = S(t)X_0 + \int_0^{t} S(t -\sigma)F(\sigma, X(\sigma)) \, d\sigma + \int_0^{t} S(t -\sigma)G(\sigma, X(\sigma)) \, dW(\sigma)\;\text{in }\;\dot{\mathbb{H}}^{\mu}.
\end{equation}
\end{enumerate}
A local solution \((X,\tau)\) is called maximal if for any other local solution \((\widetilde{X},\widetilde{\tau})\), it holds that \(\widetilde{\tau}\leq\tau\;\mathbb{P}\)-almost surely.
\end{definition}
\begin{remark}
\begin{enumerate} [label={\upshape(\roman*)}]
    \item For the stochastic interval \(\llbracket 0,\tau_{n}\rrbracket\coloneqq \{(t,\omega)\in[0,T]\times\Omega:\,0\leq t \leq \tau_n(\omega)\}\), the process \(\big(\mathbbm{1}_{\llbracket 0,\tau_{n}\rrbracket}(t)\big)_{t\in[0,T]}\) defined by \[\mathbbm{1}_{\llbracket 0,\tau_{n}\rrbracket}(t,\omega)\coloneqq \begin{cases}
        1&:\,0\leq t\leq \tau_n(\omega),\\
        0&:\,\text{otherwise},
    \end{cases}\] is predictable. 
    \item One may be tempted to replace condition (iii) in Definition~\ref{def:local_mild_solution} of local mild solution with a more compact expression evaluated at the stopped time, a formulation occasionally encountered in the literature (\cite[Definition~1.2]{MR1213834}, \cite[Definition~4.7]{MR1488138}):
    \[
    X(t\wedge\tau_n) = S(t\wedge\tau_n)X_0 + \int_0^{t\wedge\tau_n} S(t\wedge\tau_n -\sigma)F(\sigma, X(\sigma)) \, d\sigma + \int_0^{t\wedge\tau_n} S(t\wedge\tau_n -\sigma)G(\sigma, X(\sigma)) \, dW(\sigma).
    \]
    However, as pointed out in \cite[Appendix~A]{MR2136869} and \cite[Definition~7.1]{MR2602928}, this formulation is technically flawed. The integrand of the stochastic convolution involves the term $S(t\wedge\tau_n - \sigma)$. Because the realization of the stopping time $\tau_n$ is generally not determined by the filtration up to time $\sigma$, the mapping $\sigma \mapsto S(t\wedge\tau_n - \sigma)x$ (for $x \neq 0$) is not $\mathcal{F}_\sigma$-adapted. Consequently, the stochastic integral anticipates the future and is ill-defined in the It\^o sense.
    
    To rigorously circumvent this lack of predictability, one must decouple the stopping time from the deterministic semigroup. An equivalent and mathematically well-posed formulation, which holds $\mathbb{P}$-almost surely, relies on indicator functions over the stochastic interval $\llbracket 0,\tau_{n}\rrbracket$:
    \begin{align}
        \nonumber \mathbbm{1}_{\llbracket 0,\tau_{n}\rrbracket}(t) X(t) = &\mathbbm{1}_{\llbracket 0,\tau_{n}\rrbracket}(t) \Bigg( S(t)X_0 + \int_0^{t} \mathbbm{1}_{\llbracket 0,\tau_{n}\rrbracket}(\sigma)S(t -\sigma)F(\sigma, X(\sigma)) \, d\sigma\\
        &+ \int_0^{t} \mathbbm{1}_{\llbracket 0,\tau_{n}\rrbracket}(\sigma)S(t -\sigma)G(\sigma, X(\sigma)) \, dW(\sigma) \Bigg).
    \end{align}
    Depending on algebraic convenience, condition (iii) in Definition~\ref{def:local_mild_solution} may be freely replaced with this global indicator formulation without any loss of generality.
\end{enumerate}
\end{remark}

\begin{theorem}[Local existence, uniqueness, and blow-up alternative]
\label{thm:local_existence}
Suppose that Assumption~\ref{ass:operator_A} holds, and that Assumptions~\ref{ass:measurability_drift_diff} and~\ref{ass:local_lip_stoch} hold for a set of parameters satisfying Assumption~\ref{ass:parameters}. Let $X_0\in L^0(\Omega, \mathcal{F}_0; \dot{\mathbb{H}}^{\mu})$. Then, \eqref{eq:SEE} admits a unique maximal local mild solution $(X, \tau_{\max})$ with explosion time \(\tau_{\max}\) and a localizing sequence of stopping times \(\tau_n\uparrow \tau_{\max}\;\mathbb{P}\)-almost surely such that for each $n \in \mathbb{N}$, the stopped process $\{X(t \wedge \tau_n)\}_{t \in [0, T]}$ is an $\dot{\mathbb{H}}^{\mu}$-valued predictable stochastic process satisfying the following pathwise regularity $\mathbb{P}$-almost surely: for any \(0<a<T\) and \(0<\delta<\min\big\{1-\frac{1}{2}(\nu+\alpha)^+, \frac{1}{2}-\frac{1}{2}(\nu+\beta)^+\big\}\),
\begin{align}
\label{eq:local_existence_1}&X(\cdot\wedge\tau_n) \in \mathcal{C}\big([0,T]; \dot{\mathbb{H}}^{\mu}\big) \cap \mathring{\mathcal{C}}^{0,\delta}\big([a,T]; \dot{\mathbb{H}}^{\nu}\big)\text{ with }X(0)=X_0,\text{ and } \\
\label{eq:local_existence_2}& \lim\limits_{t\searrow 0}\|X(t)\|_{\lambda,t}=\|X_0\|_{\dot{\mathbb{H}}^\mu} \text{ on }\{\tau_n>0\},
\end{align}
where the time-weighted norm \(\|\cdot\|_{\lambda,t}\) is given in (\ref{eq:time_weighted_norm}). The solution $(X, \tau_{\max})$ is the unique maximal local mild solution to \eqref{eq:SEE} in the sense that for any other local mild solution $(\widetilde{X}, \widetilde{\tau})$ with identical pathwise regularity, one has $\widetilde{\tau} \leq \tau_{\max}$ $\mathbb{P}$-a.s., and $\|X(t) - \widetilde{X}(t)\|_{\lambda,t} = 0$ $\mathbb{P}$-a.s. on $\{t < \tau_{\max} \wedge \widetilde{\tau}\}$. Furthermore, the explosion time $\tau_{\max}$ satisfies $\limsup_{t \uparrow \tau_{\max}} \|X(t)\|_{\lambda,t} = +\infty$ $\mathbb{P}$-a.s. on $\{\tau_{\max} < T\}$.
\end{theorem}

\begin{remark}\label{rem:seidler_local}
Alongside the global well-posedness result discussed in Remark~\ref{rem:global_case_comparison}(iii), our framework also generalizes the maximal local existence results of Seidler~\cite[Theorem 1.5]{MR1213834}. Seidler's theorem restricts initial data to the fractional space $H_\zeta$ for $\zeta \ge \delta$ (corresponding to the restriction $\mu \ge \nu$). We overcome this limitation by formulating the locally Lipschitz conditions and stopping times with respect to the time-weighted norm $\|\cdot\|_{\lambda,t}$. This decoupling permits initial data with regularity $\mu < \nu$, which is essential for accommodating equations with locally Lipschitz coefficients subject to rough initial data (see Section~\ref{sec:examples}).
\end{remark}

\section{Proof of main results}\label{sec:proof_of_main_results}
To streamline the exposition and avoid repetitive qualifications regarding uniqueness, we adopt the following convention: two $\mathcal{X}$-valued stochastic processes $\{X(t)\}_{t\in[0,T]}$ and $\{Y(t)\}_{t\in[0,T]}$ are identified whenever they are versions (modifications) of each other. That is,
\[
\|X(t)-Y(t)\|_{\mathcal{X}}=0 \quad \mathbb{P}\text{-a.s.} \quad \text{for all } t\in[0,T].
\]

Before proceeding to the proofs of our main theorems, we establish two essential auxiliary lemmas. The first lemma provides bounds for the nonlinear operators $F$ and $G$.

\begin{lemma}\label{lem:Bounds_on_F_and_G}
Suppose that Assumptions~\ref{ass:measurability_drift_diff} and~\ref{ass:global_lips_nonlin} hold for a set of parameters satisfying Assumptions~\ref{ass:parameters}. Let $p \ge 2$ and $\lambda = (\nu-\mu)/2$. Then, for all $\varphi \in \mathbb{V}_p$ and $\sigma \in (0,T]$, the nonlinear operators $F$ and $G$ satisfy the following bounds:
\begin{align*}
    \|F(\sigma,\varphi(\sigma))\|_{L^p(\Omega;\dot{\mathbb{H}}^{-\alpha})}
    &\le \max\{L_1, L_2 T^\lambda\} \sigma^{-(\hat{\alpha}+\lambda)} \big(1+\|\varphi\|_{\mathbb{V}_p}\big), \\
    \|G(\sigma,\varphi(\sigma))\|_{L^p(\Omega;\mathrm{HS}(U_0,\dot{\mathbb{H}}^{-\beta}))}
    &\le \max\{L_3, L_4 T^\lambda\} \sigma^{-(\hat{\beta}+\lambda)} \big(1+\|\varphi\|_{\mathbb{V}_p}\big).
\end{align*}
\end{lemma}

\begin{proof}
From the definition of the space $\mathbb{V}_p$, we possess the pointwise estimate
\[
\|\varphi(\sigma)\|_{L^p(\Omega;\dot{\mathbb{H}}^{\nu})} \le \sigma^{-\lambda}\|\varphi\|_{\mathbb{V}_p}, \quad \sigma \in (0,T]. 
\]
Utilizing Assumptions~\ref{ass:measurability_drift_diff} and \ref{ass:global_lips_nonlin}, we obtain
\[
\begin{aligned}
\|F(\sigma,\varphi(\sigma))\|_{L^p(\Omega;\dot{\mathbb{H}}^{-\alpha})} 
&\le L_1 \sigma^{-\hat{\alpha}} \|\varphi(\sigma)\|_{L^p(\Omega;\dot{\mathbb{H}}^{\nu})} + L_2 \sigma^{-\hat{\alpha}} \\
&\le L_1 \|\varphi\|_{\mathbb{V}_p} \sigma^{-(\hat{\alpha}+\lambda)} + L_2 \sigma^{-\hat{\alpha}}.
\end{aligned}
\]
Since $\sigma \le T$, we can bound the second term by writing $\sigma^{-\hat{\alpha}} = \sigma^{-(\hat{\alpha}+\lambda)} \sigma^\lambda \le T^\lambda \sigma^{-(\hat{\alpha}+\lambda)}$. Substituting this yields
\[
\begin{aligned}
\|F(\sigma,\varphi(\sigma))\|_{L^p(\Omega;\dot{\mathbb{H}}^{-\alpha})} 
&\le L_1 \|\varphi\|_{\mathbb{V}_p} \sigma^{-(\hat{\alpha}+\lambda)} + L_2 T^{\lambda} \sigma^{-(\hat{\alpha}+\lambda)} \\
&\le \max\{L_1, L_2 T^{\lambda}\} \sigma^{-(\hat{\alpha}+\lambda)} \left(1+\|\varphi\|_{\mathbb{V}_p}\right).
\end{aligned} 
\]
The estimate for $G$ follows by an identical argument.
\end{proof}

Next, we establish the well-definedness and Hölder regularity of the deterministic and stochastic convolutions. The shifted convolution operators introduced in this lemma will play a crucial role in the proofs of Theorem~\ref{thm:main_global_existence} and Theorem~\ref{thm:pathwise_regularity}.

\begin{lemma}[Convolution regularity]\label{lem:convolutions_regularity}
    Suppose that Assumptions~\ref{ass:measurability_drift_diff} and~\ref{ass:global_lips_nonlin} hold for a set of parameters satisfying satisfying Assumption~\ref{ass:parameters}. Then, for $0 \leq t_0 < T$, the shifted convolution operators \(\mathcal{J}_F^{t_0}, \mathcal{J}_G^{t_0} \colon \mathbb{V}_p \to \mathbb{V}_p\) for \(p \ge 2\), defined by
\begin{align*}
    (\mathcal{J}_F^{t_0}\varphi)(t) &\coloneqq \int_{0}^t S(t-\sigma)\,\mathbbm{1}_{[t_0,T]}(\sigma)\, F(\sigma,\varphi(\sigma))\,d\sigma, \\
    (\mathcal{J}_G^{t_0}\varphi)(t) &\coloneqq \int_{0}^t S(t-\sigma)\,\mathbbm{1}_{[t_0,T]}(\sigma)\,G(\sigma,\varphi(\sigma))\,dW(\sigma),
\end{align*}
with $ \mathcal{J}_F \coloneqq \mathcal{J}_F^0$ and $\mathcal{J}_G \coloneqq \mathcal{J}_G^0$, are well-defined. Moreover, there exist positive constants $C$ and $C_a$ (where $C_a$ may diverge as $a \searrow 0$) such that these operators satisfy the following regularity bounds:
\begin{enumerate}[label={\upshape(\roman*)}]
    \item \textbf{Deterministic convolution:} Let \(\eta \le \nu\). Then, the following estimate holds:
    \begin{equation}\label{eq:det_w_norm_bound}
        \|(\mathcal{J}_F^{t_0}\varphi)(t)\|_{L^p(\Omega; \dot{\mathbb{H}}^\eta)} \leq C\big(1+\|\varphi\|_{\mathbb{V}_p}\big)\,t^{1-\hat{\alpha}-\lambda-\frac{1}{2}(\eta+\alpha)^+},\quad t\in(0,T].
    \end{equation}
    Moreover, for any $a > 0$ and $0<\delta < 1 - \frac{1}{2}(\eta+\alpha)^+$, $\mathcal{J}_F^{t_0}\varphi \in \mathcal{C}^{0, \delta}([a, T]; L^p(\Omega; \dot{\mathbb{H}}^\eta))$ and the following local estimate holds:
    \begin{equation}\label{eq:det_con_local}
        \|\mathcal{J}_F^{t_0} \varphi\|_{\mathcal{C}^{0, \delta}([a, T]; L^p(\Omega; \dot{\mathbb{H}}^\eta))} \leq C_a\big(1+\|\varphi\|_{\mathbb{V}_p}\big).
    \end{equation}
    In particular, for \(\eta=\mu\) and $0<\delta < 1 - \hat{\alpha} - \lambda - \frac{1}{2}(\mu+\alpha)^+$, we have $\mathcal{J}_F^{t_0}\varphi \in \mathcal{C}^{0, \delta}([0, T]; L^p(\Omega; \dot{\mathbb{H}}^\mu))$, and the following global estimate holds:
    \begin{equation}\label{eq:det_con_global}
        \|\mathcal{J}_F^{t_0}\varphi\|_{\mathcal{C}^{0, \delta}([0, T]; L^p(\Omega; \dot{\mathbb{H}}^\mu))} \leq C\big(1+\|\varphi\|_{\mathbb{V}_p}\big).
    \end{equation}
    
    \item \textbf{Stochastic convolution:} Let \(\vartheta \le \nu\). Then, the following estimate holds: 
    \begin{equation}\label{eq:stoc_w_norm_bound}
        \|(\mathcal{J}_G^{t_0}\varphi)(t)\|_{L^p(\Omega; \dot{\mathbb{H}}^\vartheta)} \leq C\big(1+\|\varphi\|_{\mathbb{V}_p}\big)\,t^{\frac{1}{2}-\hat{\beta}-\lambda-\frac{1}{2}(\vartheta+\beta)^+},\quad t\in(0,T].
    \end{equation}
    Moreover, for any \(a>0\) and $0<\delta < \frac{1}{2}  - \frac{1}{2}(\vartheta+\beta)^+$, $\mathcal{J}_G^{t_0}\varphi \in \mathcal{C}^{0, \delta}([a, T]; L^p(\Omega; \dot{\mathbb{H}}^\vartheta))$ and the following local estimate holds:
    \begin{equation}\label{eq:sto_con_local}
        \|\mathcal{J}_G^{t_0} \varphi\|_{\mathcal{C}^{0, \delta}([a, T]; L^p(\Omega; \dot{\mathbb{H}}^\vartheta))} \leq C_a\big(1+\|\varphi\|_{\mathbb{V}_p}\big).
    \end{equation}
    In particular, for \(\vartheta=\mu\) and $0<\delta < \frac{1}{2} - \hat{\beta} - \lambda - \frac{1}{2}(\mu+\beta)^+$, we obtain $\mathcal{J}_G^{t_0} \varphi \in \mathcal{C}^{0, \delta}([0, T]; L^p(\Omega; \dot{\mathbb{H}}^\mu))$ satisfying the following global estimate:
    \begin{equation}\label{eq:sto_con_global}
        \|\mathcal{J}_G^{t_0} \varphi\|_{\mathcal{C}^{0, \delta}([0, T]; L^p(\Omega; \dot{\mathbb{H}}^\mu))} \leq C\big(1+\|\varphi\|_{\mathbb{V}_p}\big).
    \end{equation}
\end{enumerate}
\end{lemma}

\begin{proof}
Let \(\varphi\in\mathbb{V}_p\) for \(p\geq 2\). Then, by the definition of the space $\mathbb{V}_p$, $\varphi$ admits a predictable version, with which we identify $\varphi$. Since the composition of a measurable mapping with a predictable process yields a predictable process, Assumption \ref{ass:measurability_drift_diff} ensures that $\{F(\sigma, \varphi(\sigma))\}_{\sigma\in[0,T]}$ and $\{G(\sigma, \varphi(\sigma))\}_{\sigma\in[0,T]}$ are predictable processes taking values in $\dot{\mathbb{H}}^{-\alpha}$ and $\mathrm{HS}(U_0, \dot{\mathbb{H}}^{-\beta})$, respectively. 

Since, for each fixed \(t\in[0,T]\), \(\{\mathbbm{1}_{[t_0,T]}(\sigma)\}_{\sigma\in[0,T]} \) is a deterministic process and \(\{S(t)\}_{t\geq 0}\) is an analytic $\mathcal{C}_0$-semigroup, the processes \(\{\mathbf{\Phi}_t(\sigma)\}_{\sigma\in[0,T]}\) and \(\{\mathbf{\Psi}_t(\sigma)\}_{\sigma\in[0,T]}\) are predictable, where
\begin{align*}
    \mathbf{\Phi}_t(\sigma) &\coloneqq \mathbbm{1}_{[t_0,T]}(\sigma)\,S(t-\sigma)\,F(\sigma, \varphi(\sigma)), \\
    \mathbf{\Psi}_t(\sigma) &\coloneqq \mathbbm{1}_{[t_0,T]}(\sigma)\,S(t-\sigma)\,G(\sigma, \varphi(\sigma)).
\end{align*}
Moreover, for any \(\eta,\vartheta\in\mathbb{R}\), \(\rho\in(0,1]\), and \(s,t\in[0,T]\), by utilizing Lemma~\ref{lem:semigroup_estimates} and Lemma~\ref{lem:Bounds_on_F_and_G}, we obtain the following estimates:
\begin{align}
    \nonumber \|\mathbf{\Phi}_t(\sigma)\|_{L^p(\Omega;\dot{\mathbb{H}}^{\eta})} 
    &\leq \|A^{\frac{\eta+\alpha}{2}} S(t-\sigma) \|_{\mathcal{L}(\mathbb{H})} \|F(\sigma, \varphi(\sigma))\|_{L^p(\Omega; \dot{\mathbb{H}}^{-\alpha})}\\
    \label{eq:Phi_t_estimate} 
    &\leq C \max\{L_1, L_2 T^\lambda\} \big(1+\|\varphi\|_{\mathbb{V}_p}\big) (t-\sigma)^{-\frac{1}{2}(\eta+\alpha)^+} \sigma^{-(\hat{\alpha}+\lambda)} \quad \forall \,0< \sigma < t,\\
    \nonumber \|\mathbf{\Psi}_t(\sigma)\|_{L^p(\Omega;\mathrm{HS}(U_0,\dot{\mathbb{H}}^{\vartheta}))} 
    &\leq \|A^{\frac{\vartheta+\beta}{2}} S(t-\sigma) \|_{\mathcal{L}(\mathbb{H})} \|G(\sigma,\varphi(\sigma))\|_{L^p(\Omega;\mathrm{HS}(U_0,\dot{\mathbb{H}}^{-\beta}))} \\
    \label{eq:Psi_t_estimate} 
    &\leq C \max\{L_3, L_4 T^\lambda\} \big(1+\|\varphi\|_{\mathbb{V}_p}\big) (t-\sigma)^{-\frac{1}{2}(\vartheta+\beta)^+} \sigma^{-(\hat{\beta}+\lambda)} \quad \forall \,0< \sigma < t. 
\end{align}
Furthermore, for \(0<\sigma<s\wedge t\), we have
\begin{align}
   \nonumber \|\mathbf{\Phi}_t(\sigma)-\mathbf{\Phi}_s(\sigma)\|_{L^p(\Omega;\dot{\mathbb{H}}^{\eta})} 
   &\leq \| A^{\frac{\eta+\alpha}{2}}\big( S(t-\sigma) - S(s-\sigma)\big)\|_{\mathcal{L}(\mathbb{H})} \|F(\sigma, \varphi(\sigma))\|_{L^p(\Omega; \dot{\mathbb{H}}^{-\alpha})} \\
   \label{eq:Phi_t_diff_estimate}
   &\leq C \max\{L_1, L_2 T^\lambda\} \big(1+\|\varphi\|_{\mathbb{V}_p}\big) |t-s|^\rho(s\wedge t-\sigma)^{-\frac{1}{2}(\eta+\alpha +2\rho)^+} \sigma^{-(\hat{\alpha}+\lambda)},\\
   \nonumber \|\mathbf{\Psi}_t(\sigma)-\mathbf{\Psi}_s(\sigma)\|_{L^p(\Omega;\mathrm{HS}(U_0,\dot{\mathbb{H}}^{\vartheta}))} 
   &\leq \|A^{\frac{\vartheta+\beta}{2}} \big(S(t-\sigma) - S(s-\sigma)\big) \|_{\mathcal{L}(\mathbb{H})} \|G(\sigma,\varphi(\sigma))\|_{L^p(\Omega;\mathrm{HS}(U_0,\dot{\mathbb{H}}^{-\beta}))}\\
   \label{eq:Psi_t_diff_estimate}
   &\leq C \max\{L_3, L_4 T^\lambda\} \big(1+\|\varphi\|_{\mathbb{V}_p}\big) |t-s|^\rho(s\wedge t-\sigma)^{-\frac{1}{2}(\vartheta+\beta+2\rho)^+} \sigma^{-(\hat{\beta}+\lambda)}.
\end{align}
Let \(\eta\leq \nu\). Applying Minkowski's integral inequality and the estimate \eqref{eq:Phi_t_estimate} yields, for \(t>0\):
\begin{align}
\nonumber \|(\mathcal{J}_F^{t_0}\varphi)(t)\|_{L^p(\Omega;\dot{\mathbb{H}}^{\eta})} 
&\le \int_{0}^t \|\mathbf{\Phi}_t(\sigma)\|_{L^p(\Omega;\dot{\mathbb{H}}^{\eta})} \, d\sigma \\
\nonumber&\le C \max\{L_1, L_2 T^\lambda\} \big(1+\|\varphi\|_{\mathbb{V}_p}\big) \int_{0}^t (t-\sigma)^{-\frac{1}{2}(\eta+\alpha)^+} \sigma^{-(\hat{\alpha}+\lambda)} \, d\sigma \\
\label{eq:F_bound}&\leq C \max\{L_1, L_2 T^\lambda\} \big(1+\|\varphi\|_{\mathbb{V}_p}\big) B\Big(1-\frac{1}{2}(\eta+\alpha)^+, 1-\hat{\alpha}-\lambda\Big)\, t^{1-\hat{\alpha} - \lambda -\frac{1}{2}(\eta+\alpha)^+},
\end{align}
where \(B\) is the Euler Beta function. The estimate \eqref{eq:F_bound} establishes the result \eqref{eq:det_w_norm_bound}, and yields the following:
\begin{align}
\label{eq:F_Vp-norm_bound}& \sup_{t \in (0,T]} \Big( \|(\mathcal{J}_F^{t_0}\varphi)(t)\|_{L^p(\Omega; \dot{\mathbb{H}}^{\mu})} + t^{\lambda} \|(\mathcal{J}_F^{t_0}\varphi)(t)\|_{L^p(\Omega; \dot{\mathbb{H}}^{\nu})} \Big)\,<\,\infty.
\end{align}
To establish the H\"older continuity of \(\mathcal{J}_F^{t_0}\varphi\), we estimate the difference 
\begin{align}
    \label{eq:Holder_F_difference_eq}& (\mathcal{J}_F^{t_0}\varphi)(t) - (\mathcal{J}_F^{t_0}\varphi)(s) = \int_s^t \mathbf{\Phi}_t(\sigma)\,d\sigma + \int_{0}^s\Big(\mathbf{\Phi}_t(\sigma) - \mathbf{\Phi}_s(\sigma)\Big)\,d\sigma 
\end{align}
for \(0\leq s <t\leq T\). Choose \((s_0,\rho)\in[0,T]\times(0,1]\) such that
\begin{align}
    \nonumber&  \int_{0}^s(s-\sigma)^{-\frac{1}{2}(\eta+\alpha+2\rho)^+} \sigma^{-\hat{\alpha} - \lambda}\,d\sigma \,<\,\infty\quad \forall s\in[s_0,T],
\end{align}
and then from \eqref{eq:Holder_F_difference_eq}, by applying Minkowski's integral inequality along with the estimate \eqref{eq:Phi_t_diff_estimate}, obtain:
 \begin{align}
    \nonumber& \|(\mathcal{J}_F^{t_0}\varphi)(t) - (\mathcal{J}_F^{t_0}\varphi)(s)\|_{L^p(\Omega;\dot{\mathbb{H}}^{\eta})}  \leq C\max\{L_1,L_2 T^\lambda\}(1+\|\varphi\|_{\mathbb{V}_p}) \Big(\int_s^t (t-\sigma)^{-\frac{1}{2}(\eta+\alpha)^+}\sigma^{-\hat{\alpha}-\lambda}\,d\sigma \\
    \label{eq:Holder_F_difference_eq_estimate}&\hspace{5cm}+ |t-s|^\rho\int_{0}^s(s-\sigma)^{-\left(\frac{\eta+\alpha}{2}+\rho\right)^+} \sigma^{-\hat{\alpha} - \lambda}\,d\sigma \Big)\quad\text{ for }\;0\leq s_0 \leq s <t\leq T.
 \end{align}
 Setting \(s_0=a>0\) and \(0<\rho=\delta<1- \frac{1}{2}(\eta+\alpha)^+\) in \eqref{eq:Holder_F_difference_eq_estimate}, we obtain, for \(a\leq s < t\leq T\): 
\begin{align}
    \nonumber& \|(\mathcal{J}_F^{t_0}\varphi)(t) - (\mathcal{J}_F^{t_0}\varphi)(s)\|_{L^p(\Omega;\dot{\mathbb{H}}^{\eta})}\,\leq\, C\max\{L_1,L_2 T^\lambda\}(1+\|\varphi\|_{\mathbb{V}_p})|t-s|^\delta \Bigg( \frac{a^{-\lambda-\hat{\alpha}}}{1-\frac{1}{2}(\eta+\alpha)^+}\\
    \label{eq:Holder_F_difference_eq_eta_estimate}&\hspace{4cm} +D_a B\Big(1-\frac{1}{2}(\eta+\alpha+2\delta)^+,\,1-\hat{\alpha}-\lambda \Big)  \Bigg), 
 \end{align}
where 
\begin{align}
    \nonumber&D_a\coloneqq \begin{cases}
        a^{1-\hat{\alpha}-\lambda -\frac{1}{2}(\eta+\alpha+2\delta)^+} & :\,\text{if }1-\hat{\alpha}-\lambda -\frac{1}{2}(\eta+\alpha+2\delta)^+ \,<\,0,\\
        T^{1-\hat{\alpha}-\lambda -\frac{1}{2}(\eta+\alpha+2\delta)^+} &:\,\text{ otherwise.}
    \end{cases}
\end{align}
This establishes estimate \eqref{eq:det_con_local}. Setting \(\eta=\mu, s_0=0,\) and \(0<\rho=\delta<1-\hat{\alpha} - \lambda - \frac{1}{2}(\mu+\alpha)^+\) in \eqref{eq:Holder_F_difference_eq_estimate} and using \(\sigma^{-(\hat{\alpha}+\lambda)}\leq (\sigma-s)^{-(\hat{\alpha}+\lambda)}\) for all \(\sigma>s\), yields for \(0\leq  s < t\leq T\):
 \begin{align}
    \nonumber \|(\mathcal{J}_F^{t_0}\varphi)(t) - (\mathcal{J}_F^{t_0}\varphi)(s)\|_{L^p(\Omega;\dot{\mathbb{H}}^{\mu})}  
    &\leq C\max\{L_1,L_2 T^\lambda\}\big(1+\|\varphi\|_{\mathbb{V}_p}\big) \Bigg(\int_s^t (t-\sigma)^{-\frac{1}{2}(\mu+\alpha)^+}(\sigma-s)^{-(\hat{\alpha}+\lambda)}\,d\sigma \\
    \nonumber&\hspace{1cm}+ |t-s|^\delta\int_{0}^s(s-\sigma)^{-\left(\frac{\mu+\alpha}{2}+\delta\right)^+} \sigma^{-(\hat{\alpha} + \lambda)}\,d\sigma \Bigg)\\
    \nonumber&\leq C\max\{L_1,L_2 T^\lambda\} \big(1+\|\varphi\|_{\mathbb{V}_p}\big)\,|t-s|^\delta\,\Bigg(B\Big(1-\frac{1}{2}(\mu+\alpha)^+, 1-\hat{\alpha}-\lambda\Big)\\
    \label{eq:Holder_F_difference_eq_mu_estimate}&\hspace{1cm}+ T^{1-\hat{\alpha}-\lambda - \frac{1}{2}(2\delta+\mu+\alpha)^+} B\Big(1-\frac{1}{2}(2\delta+\mu+\alpha)^+, 1-\hat{\alpha}-\lambda\Big)\Bigg). 
\end{align}
The estimate \eqref{eq:Holder_F_difference_eq_mu_estimate} establishes \eqref{eq:det_con_global}. Note that estimates \eqref{eq:Holder_F_difference_eq_mu_estimate}, \eqref{eq:Holder_F_difference_eq_eta_estimate}, and \eqref{eq:F_Vp-norm_bound} imply that \(\mathcal{J}_F^{t_0}\varphi\in \mathcal{C}\big([0,T]; L^p(\Omega; \dot{\mathbb{H}}^{\mu})\big) \) and \(t^\lambda \mathcal{J}_F^{t_0}\varphi\in \mathcal{C}_b\big((0,T]; L^p(\Omega; \dot{\mathbb{H}}^{\nu})\big)\). Moreover, for each fixed \(t\in[0,T]\), \(\{\mathbf{\Phi}_t(\sigma)\}_{\sigma\in[0,T]}\) is an \(\dot{\mathbb H}^{\mu}\)-valued predictable process and hence progressively measurable, i.e., \(\mathbf{\Phi}_t(\cdot)\) restricted to \([0,t]\times\Omega\) is \(\mathcal{B}([0,t])\otimes\mathcal{F}_t\big/
\mathcal B(\dot{\mathbb H}^{\mu})\)-measurable. Also, estimate \eqref{eq:F_bound} implies that \(\mathbb{E}\int_0^t\|\mathbf{\Phi}_t(\sigma)\|_{\dot{\mathbb H}^{\mu}}\,d\sigma \,<\,\infty\), and hence an appeal to Fubini's theorem yields that \(\mathcal{J}_F^{t_0}\varphi(t)\) is \(\mathcal{F}_t\)-measurable. Estimate \eqref{eq:Holder_F_difference_eq_mu_estimate} implies that the process \(\{\mathcal{J}_F^{t_0}\varphi(t)\}_{t\in[0,T]}\) is stochastically continuous. By invoking a standard result due to Da~Prato and Zabczyk
(see \cite[Proposition~3.7(ii)]{MR3236753}), the
adapted and stochastically continuous process \(\{\mathcal{J}_F^{t_0}\varphi(t)\}_{t\in[0,T]}\) admits a predictable version. Hence, \(\mathcal{J}_F^{t_0}\varphi\in\mathbb{V}_p\).

To estimate the stochastic convolution, we apply Lemma~\ref{lem:BDG} (Burkholder--Davis--Gundy inequality) together with Minkowski's integral inequality and obtain, for \(\vartheta\leq\nu\) and \(t>0\):
\begin{align}
  \nonumber  \|(\mathcal{J}_G^{t_0}\varphi)(t)\|_{L^p(\Omega;\dot{\mathbb{H}}^{\vartheta})}
&\,\le\,
C
\left\| \left( \int_{0}^t \|\mathbf{\Psi}_t(\sigma)\|_{\mathrm{HS}(U_0,\dot{\mathbb{H}}^{\vartheta})}^2 \, d\sigma \right)^{1/2} \right\|_{L^p(\Omega;\mathbb{R})} \,\le\,
C
 \left( \int_{0}^t \|\mathbf{\Psi}_t(\sigma)\|_{L^p(\Omega;\mathrm{HS}(U_0,\dot{\mathbb{H}}^{\vartheta}))}^2 \, d\sigma \right)^{1/2}.
\end{align}
Now, an appeal to estimate \eqref{eq:Psi_t_estimate} yields, for \(t>0\):
\begin{align}
   \nonumber & \|(\mathcal{J}_G^{t_0}\varphi)(t)\|_{L^p(\Omega;\dot{\mathbb{H}}^{\vartheta})} \le
C \max\{L_3,\,L_4 T^\lambda\} \left(1+\|\varphi\|_{\mathbb{V}_p}\right)
\left( \int_0^t (t-\sigma)^{-(\vartheta+\beta)^+} \sigma^{-2(\hat{\beta}+\lambda)} \, d\sigma \right)^{1/2}\\
\label{eq:G_bound}&\hspace{2cm}\le
C \max\{L_3,\,L_4 T^\lambda\} \left(1+\|\varphi\|_{\mathbb{V}_p}\right)\left(B(1-(\vartheta+\beta)^+, 1-2\hat{\beta} - 2\lambda)\right)^{1/2}\,t^{\frac{1}{2}-\hat{\beta}-\lambda - \frac{1}{2}(\vartheta+\beta)^+}.
\end{align}
This establishes the estimate \eqref{eq:stoc_w_norm_bound}, and yields the following:
\begin{align}
\label{eq:G_Vp-norm_bound}& \sup_{t \in (0,T]} \Big( \|(\mathcal{J}_G^{t_0}\varphi)(t)\|_{L^p(\Omega; \dot{\mathbb{H}}^{\mu})} + t^{\lambda} \|(\mathcal{J}_G^{t_0}\varphi)(t)\|_{L^p(\Omega; \dot{\mathbb{H}}^{\nu})} \Big)\,<\,\infty.
\end{align}
To establish the H\"older continuity of the stochastic convolution \(\mathcal{J}_G^{t_0}\varphi\), we estimate the difference 
\begin{align}
   \label{eq:Holder_G_difference_eq}  (\mathcal{J}_G^{t_0}\varphi)(t) - (\mathcal{J}_G^{t_0}\varphi)(s) = \int_s^t \mathbf{\Psi}_t (\sigma)\,dW(\sigma) + \int_{0}^s\Big(\mathbf{\Psi}_t (\sigma) - \mathbf{\Psi}_s (\sigma)\Big)\,dW(\sigma) 
\end{align}
for \(0\leq s <t\leq T\). Choose \((s_0,\rho)\in[0,T]\times(0,1]\) such that
\begin{align}
    \nonumber&  \int_{0}^s (s-\sigma)^{-(\vartheta+\beta+2\rho)^+} \sigma^{-2(\hat{\beta}+\lambda)} \, d\sigma \,<\,\infty\quad \forall s\in[s_0,T],
\end{align}
and then apply Lemma~\ref{lem:BDG} (Burkholder--Davis--Gundy inequality), Minkowski's integral inequality, and the estimate \eqref{eq:Psi_t_diff_estimate}, to obtain:
 \begin{align}
    \nonumber& \|(\mathcal{J}_G^{t_0}\varphi)(t) - (\mathcal{J}_G^{t_0}\varphi)(s)\|_{L^p(\Omega;\dot{\mathbb{H}}^{\vartheta})}  \leq C\max\{L_3,L_4 T^\lambda\}(1+\|\varphi\|_{\mathbb{V}_p}) \Bigg(\Big( \int_s^t (t-\sigma)^{-(\vartheta+\beta)^+} \sigma^{-2(\hat{\beta}+\lambda)} \, d\sigma \Big)^{1/2} \\
    \label{eq:Holder_G_difference_eq_estimate}&\hspace{3cm}+ |t-s|^\rho\Big( \int_{0}^s (s-\sigma)^{-(\vartheta+\beta+2\rho)^+} \sigma^{-2(\hat{\beta}+\lambda)} \, d\sigma \Big)^{1/2}\Bigg)\quad\text{ for }\;0\leq s_0\leq s<t\leq T.
 \end{align}
 Setting \(s_0=a>0\) and \(0<\rho=\delta<\frac{1}{2}- \frac{1}{2}(\vartheta+\beta)^+\) in \eqref{eq:Holder_G_difference_eq_estimate}, we obtain, for \(a\leq s < t\leq T\): 
\begin{align}
    \nonumber& \|(\mathcal{J}_G^{t_0}\varphi)(t) - (\mathcal{J}_G^{t_0}\varphi)(s)\|_{L^p(\Omega;\dot{\mathbb{H}}^{\vartheta})}\,\leq\, C\max\{L_3,L_4 T^\lambda\}(1+\|\varphi\|_{\mathbb{V}_p})|t-s|^\delta \Bigg( \Big(\frac{a^{-2\lambda-2\hat{\beta}}}{1-(\vartheta+\beta)^+}\Big)^{1/2}\\
    \label{eq:Holder_G_difference_eq_eta_estimate}&\hspace{3cm}+ \left(\widetilde{D}_a\,B\big(1-(\vartheta+\beta+2\delta)^+,\,1-2\hat{\beta} -2\lambda\big)\right)^{1/2} \Bigg),
 \end{align}
where 
\begin{align}
    \nonumber&\widetilde{D}_a\coloneqq \begin{cases}
        a^{1-2\hat{\beta}-2\lambda -(\vartheta+\beta+2\delta)^+} & :\,\text{if } 1-2\hat{\beta}-2\lambda -(\vartheta+\beta+2\delta)^+ \,<\,0,\\
        T^{1-2\hat{\beta}-2\lambda -(\vartheta+\beta+2\delta)^+} &:\,\text{ otherwise.}
    \end{cases}
\end{align}
This establishes the estimate \eqref{eq:sto_con_local}. Now, by setting \(\vartheta=\mu\), \(t_0=0\), and \(0<\rho=\delta<\frac{1}{2}-\hat{\beta} - \lambda - \frac{1}{2}(\mu+\beta)^+\) in \eqref{eq:Holder_G_difference_eq_estimate} and using \(\sigma^{-2\hat{\beta}-2\lambda}\leq (\sigma-s)^{-2\hat{\beta}-2\lambda} \quad  \forall \, \sigma>s\), we obtain, for \(0\leq s < t\leq T\): 
\begin{align}
    \nonumber \|(\mathcal{J}_G\varphi)(t) - (\mathcal{J}_G\varphi)(s)\|_{L^p(\Omega;\dot{\mathbb{H}}^{\mu})}  &\leq C\max\{L_3,L_4 T^\lambda\}(1+\|\varphi\|_{\mathbb{V}_p}) \Bigg(\Big( \int_s^t (t-\sigma)^{-(\mu+\beta)^+} (\sigma-s)^{-2(\hat{\beta}+\lambda)} \, d\sigma \Big)^{1/2} \\
    \nonumber&\hspace{3cm}+ |t-s|^\delta\Big( \int_{0}^s (s-\sigma)^{-(\mu+\beta+2\delta)^+} \sigma^{-2(\hat{\beta}+\lambda)} \, d\sigma \Big)^{1/2}\Bigg)\\
    \nonumber&\leq C\max\{L_3,L_4 T^\lambda\}(1+\|\varphi\|_{\mathbb{V}_p})|t-s|^\delta \Bigg(\Big(B(1-(\mu+\beta)^+,1-2\hat{\beta}-2\lambda)\Big)^{1/2}\\
    \label{eq:Holder_G_difference_eq_mu_estimate}&\hspace{0.5cm}+ \Big(T^{1-2\hat{\beta}-2\lambda -(\mu+\beta+2\delta)^+}\,B(1-(\mu+\beta+2\delta)^+,1-2\hat{\beta}-2\lambda)\Big)^{1/2}\Bigg). 
 \end{align}
The above estimate \eqref{eq:Holder_G_difference_eq_mu_estimate} establishes \eqref{eq:sto_con_global} and, as a consequence, \(\mathcal{J}_G^{t_0}\varphi\in \mathcal{C}\big([0,T]; L^p(\Omega; \dot{\mathbb{H}}^{\mu})\big) \) which implies that the process \(\{(\mathcal{J}_G^{t_0}\varphi)(t)\}_{t\in[0,T]}\) is stochastically continuous. The estimates \eqref{eq:Holder_G_difference_eq_eta_estimate} and \eqref{eq:G_Vp-norm_bound} imply that \(t^\lambda \mathcal{J}_G^{t_0}\varphi\in \mathcal{C}_b\big((0,T]; L^p(\Omega; \dot{\mathbb{H}}^{\nu})\big)\). For each fixed \(t\in[0,T]\), the integrand \(\{\mathbf{\Psi}_t(\sigma)\}_{\sigma\in[0,t]}\) is an \(\mathrm{HS}(U_0,\dot{\mathbb{H}}^{\mu})\)-valued predictable process satisfying, by \eqref{eq:Psi_t_estimate}, \(\int_{0}^t \|\mathbf{\Psi}_t(\sigma)\|_{L^p(\Omega;\,\mathrm{HS}(U_0,\dot{\mathbb{H}}^{\mu}))}^2\, d\sigma\,<\,\infty\). Therefore, from the standard construction of the It\^{o} integral (see \cite[Chapter~6]{MR3410409}, \cite[Section~4.2]{MR3236753}), the stochastic convolution \((\mathcal{J}_G^{t_0}\varphi)(t)\) evaluates to an \(\mathcal{F}_t\)-measurable random variable, rendering the process \(\{(\mathcal{J}_G^{t_0}\varphi)(t)\}_{t\in[0,T]}\) adapted. Because the process \(\{(\mathcal{J}_G^{t_0}\varphi)(t)\}_{t\in[0,T]}\) is both adapted and stochastically continuous, it admits an \(\dot{\mathbb{H}}^{\mu}\)-valued predictable version (see \cite[Proposition~3.7(ii)]{MR3236753}). Hence, we have \(\mathcal{J}_G^{t_0}\varphi\in \mathbb{V}_p\). This concludes the proof.
\end{proof}

\subsection{Proof of Theorem~\ref{thm:main_global_existence}}

\begin{proof}[\textbf{Proof of Theorem~\ref{thm:main_global_existence}\textup{(i)} -- Existence and uniqueness.}]We define the operator $\mathcal{J} : \mathbb{V}_p \to \mathbb{V}_p$ by
\begin{align} \label{eq:operator_J_def}
(\mathcal{J}\varphi)(t)
&\coloneqq
S(t)X_0
+
\int_{0}^{t} S(t-\sigma)\, F(\sigma,\varphi(\sigma))\, d\sigma
+
\int_{0}^{t} S(t-\sigma)\, G(\sigma,\varphi(\sigma))\, dW(\sigma) \nonumber\\
&\eqqcolon \mathcal{J}_0(t) + (\mathcal{J}_F\varphi)(t) + (\mathcal{J}_G\varphi)(t),
\qquad t \in [0,T].
\end{align}
The operator $\mathcal{J}$ is a well-defined mapping from $\mathbb{V}_p$ to $\mathbb{V}_p$. To see this, note that the \(\mathcal{F}_0\)-measurability of \(\dot{\mathbb{H}}^\mu\)-valued \(X_0\), and the strong continuity and the smoothing property of the semigroup $\{S(t)\}_{t\geq 0}$ (Lemma~\ref{lem:semigroup_estimates}(i)) guarantee that the initial data term $\mathcal{J}_0(\cdot) \coloneqq S(\cdot)X_0$ belongs to $\mathbb{V}_p$. Furthermore, by Lemma~\ref{lem:convolutions_regularity}, the integral operators $\mathcal{J}_F$ and $\mathcal{J}_G$ are well-defined mappings from $\mathbb{V}_p$ to $\mathbb{V}_p$. Consequently, their sum $\mathcal{J}$ is well-defined. Next, we show that the operator $\mathcal{J}:\mathbb{V}_p\to \mathbb{V}_p$ is a contraction map. Introduce an equivalent exponentially weighted norm \(\|\cdot\|_{\mathbb{V}_p,\gamma}\) on \(\mathbb{V}_p\) such that
\[
\|\varphi\|_{\mathbb{V}_p,\gamma}
\coloneqq
\sup_{t\in(0,T]} e^{-\gamma t}
\Big(
\|\varphi(t)\|_{L^p(\Omega;\dot{\mathbb{H}}^{\mu})}
+
t^{\lambda}\|\varphi(t)\|_{L^p(\Omega;\dot{\mathbb{H}}^{\nu})}
\Big),
\qquad
\lambda=(\nu-\mu)/2, \quad \gamma>0.
\]
Since $e^{-\gamma T}\le e^{-\gamma t}\le 1$ for all $t\in[0,T]$, the norms
$\|\cdot\|_{\mathbb{V}_p}$ and $\|\cdot\|_{\mathbb{V}_p,\gamma}$ are equivalent, and hence
$(\mathbb{V}_p,\|\cdot\|_{\mathbb{V}_p,\gamma})$ is a Banach space. Let $\varphi,\psi\in \mathbb{V}_p$. Then, for \((\eta+\alpha)^+<2\) and \(0<t\leq T\), an appeal to Lemma~\ref{lem:semigroup_estimates} and Assumptions~\ref{ass:global_lips_nonlin} yields:
\begin{align}
    \nonumber \|(\mathcal{J}_F\varphi)(t)-(\mathcal{J}_F\psi)(t)\|_{L^p(\Omega;\dot{\mathbb{H}}^\eta)} &\leq C\int_0^t (t-\sigma)^{-\frac{1}{2}(\eta+\alpha)^+}\|F(\sigma,\varphi(\sigma)) - F(\sigma,\psi(\sigma))\|_{L^p(\Omega;\dot{\mathbb{H}}^{-\alpha})}\,d\sigma\\
    \nonumber&\leq CL_1\int_0^t (t-\sigma)^{-\frac{1}{2}(\eta+\alpha)^+}\sigma^{-\hat{\alpha}}\|\varphi(\sigma) - \psi(\sigma)\|_{L^p(\Omega;\dot{\mathbb{H}}^{\nu})}\,d\sigma\\
    \label{eq:contr_F_est_1}&\leq CL_1 \|\varphi - \psi\|_{\mathbb{V}_p,\gamma}\int_0^t (t-\sigma)^{-\frac{1}{2}(\eta+\alpha)^+}\sigma^{-\hat{\alpha}-\lambda}\,e^{\gamma\sigma} \,d\sigma.
\end{align}
From the above estimate \eqref{eq:contr_F_est_1} with particular values of $\eta=\mu$ and $\eta=\nu$, we obtain
\begin{align}
    \nonumber&e^{-\gamma t}\big(\|(\mathcal{J}_F\varphi)(t)-(\mathcal{J}_F\psi)(t)\|_{L^p(\Omega;\dot{\mathbb{H}}^\mu)} + t^\lambda\|(\mathcal{J}_F\varphi)(t)-(\mathcal{J}_F\psi)(t)\|_{L^p(\Omega;\dot{\mathbb{H}}^\nu)}\big)\\
    \nonumber&\leq CL_1 \|\varphi - \psi\|_{\mathbb{V}_p,\gamma}\Big(\int_0^t (t-\sigma)^{-\frac{1}{2}(\mu+\alpha)^+}\sigma^{-\hat{\alpha}-\lambda}\,e^{-\gamma(t-\sigma)}\,d\sigma + t^\lambda \int_0^t (t-\sigma)^{-\frac{1}{2}(\nu+\alpha)^+}\sigma^{-\hat{\alpha}-\lambda}\,e^{-\gamma(t-\sigma)}\,d\sigma\Big), 
\end{align}
and as a consequence
\begin{align}
    \label{eq:contr_operator_J_F}&\|\mathcal{J}_F \varphi - \mathcal{J}_F \psi\|_{\mathbb{V}_p,\gamma}\leq K_F(\gamma) \|\varphi - \psi\|_{\mathbb{V}_p,\gamma},
\end{align}
where 
\[
K_F(\gamma)\coloneqq CL_1\sup_{t\in(0,T]}\Big(\int_0^t (t-\sigma)^{-\frac{1}{2}(\mu+\alpha)^+}\sigma^{-\hat{\alpha}-\lambda}\,e^{-\gamma(t-\sigma)}\,d\sigma + t^\lambda \int_0^t (t-\sigma)^{-\frac{1}{2}(\nu+\alpha)^+}\sigma^{-\hat{\alpha}-\lambda}\,e^{-\gamma(t-\sigma)}\,d\sigma\Big).
\] 
Since \(1-\hat{\alpha} -\lambda-\frac{1}{2}(\mu+\alpha)^+>0\) and \(1-\hat{\alpha}-\frac{1}{2}(\nu+\alpha)^+>0\), there exists a \(0<p<\infty\) such that \(\frac{1}{p}-\hat{\alpha} -\lambda-\frac{1}{2}(\mu+\alpha)^+>0\) and \(\frac{1}{p}-\hat{\alpha}-\frac{1}{2}(\nu+\alpha)^+>0\), and hence an appeal to H\"older's inequality yields the following estimate for \(K_F(\gamma)\):
\begin{align}
    \nonumber K_F(\gamma) &\leq \frac{CL_1}{(q\gamma)^{1/q}}\sup_{t\in(0,T]}\left(\Big(\int_0^t (t-\sigma)^{-\frac{p}{2}(\mu+\alpha)^+}\sigma^{-p\hat{\alpha}-p\lambda}\,d\sigma\Big)^{1/p} + \Big(t^{p\lambda} \int_0^t (t-\sigma)^{-\frac{p}{2}(\nu+\alpha)^+}\sigma^{-p\hat{\alpha}-p\lambda}\,d\sigma\Big)^{1/p}\right)\\
    \nonumber&\leq\frac{CL_1}{(q\gamma)^{1/q}}\Big(T^{\frac{1}{p}-\frac{1}{2}(\mu+\alpha)^+ - \hat{\alpha} - \lambda}\,\Big(B\big(1-\frac{p}{2}(\mu+\alpha)^+,1-p\hat{\alpha} - p\lambda\big)\Big)^{1/p}\\
    \nonumber&\hspace{5cm}+ T^{\frac{1}{p}-\frac{1}{2}(\nu+\alpha)^+ - \hat{\alpha} }\,\Big(B\big(1-\frac{p}{2}(\nu+\alpha)^+,1-p\hat{\alpha} - p\lambda\big)\Big)^{1/p}\Big).
\end{align}
Thus there exists a \(\gamma_1^\ast>0\) such that \(K_F(\gamma)\leq1/4\quad\forall\,\gamma\geq\gamma_1^\ast\). Applying this estimate in \eqref{eq:contr_operator_J_F} yields:
\begin{align}
    \label{eq:contr_operator_J_F_final}&\|\mathcal{J}_F \varphi - \mathcal{J}_F \psi\|_{\mathbb{V}_p,\gamma}\leq \frac{1}{4} \|\varphi - \psi\|_{\mathbb{V}_p,\gamma}\quad \forall\,\gamma\geq\gamma_1^\ast.
\end{align}
Now, for \((\vartheta+\beta)^+<1\) and \(0<t\leq T\), using Lemma~\ref{lem:BDG} (Burkholder--Davis--Gundy inequality) together with Minkowski's integral inequality, Lemma~\ref{lem:semigroup_estimates} and Assumptions~\ref{ass:global_lips_nonlin}, we obtain:
\begin{align}
    \nonumber \|(\mathcal{J}_G\varphi)(t)-(\mathcal{J}_G\psi)(t)\|_{L^p(\Omega;\dot{\mathbb{H}}^\vartheta)}^2 &\leq C^2\int_0^t (t-\sigma)^{-(\vartheta+\beta)^+}\|G(\sigma,\varphi(\sigma)) - G(\sigma,\psi(\sigma))\|_{L^p(\Omega;\mathrm{HS}(U_0,\dot{\mathbb{H}}^{-\beta}))}^2\,d\sigma\\
    \nonumber&\leq C^2L_3^2\int_0^t (t-\sigma)^{-(\vartheta+\beta)^+}\sigma^{-2\hat{\beta}}\|\varphi(\sigma) - \psi(\sigma)\|_{L^p(\Omega;\dot{\mathbb{H}}^{\nu})}^2\,d\sigma\\
    \label{eq:contr_G_est_1}&\leq C^2L_3^2 \|\varphi - \psi\|_{\mathbb{V}_p,\gamma}^2\int_0^t (t-\sigma)^{-(\vartheta+\beta)^+}\sigma^{-2\hat{\beta}-2\lambda}\,e^{2\gamma\sigma} \,d\sigma.
\end{align}
Setting particular values of $\vartheta=\mu$ and $\vartheta=\nu$ in \eqref{eq:contr_G_est_1} yields:
\begin{align}
    \nonumber&e^{-\gamma t}\big(\|(\mathcal{J}_G\varphi)(t)-(\mathcal{J}_G\psi)(t)\|_{L^p(\Omega;\dot{\mathbb{H}}^\mu)} + t^\lambda\|(\mathcal{J}_G\varphi)(t)-(\mathcal{J}_G\psi)(t)\|_{L^p(\Omega;\dot{\mathbb{H}}^\nu)}\big)\leq  K_G(\gamma)\,\|\varphi - \psi\|_{\mathbb{V}_p,\gamma},
\end{align}
where
\begin{align*}
    K_G(\gamma)\coloneqq&CL_3\sup_{t\in(0,T]}\Big(\Big(\int_0^t (t-\sigma)^{-(\mu+\beta)^+}\sigma^{-2\hat{\beta}-2\lambda}\,e^{-2\gamma(t-\sigma)}\,d\sigma \Big)^{1/2}\\
    &\hspace{2cm}+t^\lambda \Big(\int_0^t (t-\sigma)^{-(\nu+\beta)^+}\sigma^{-2\hat{\beta}-2\lambda}\,e^{-2\gamma(t-\sigma)}\,d\sigma\Big)^{1/2}\Big).
\end{align*}
Under Assumption~\ref{ass:parameters}, following the approach used to estimate \(K_F(\gamma)\), we can find a sufficiently large number \(\gamma_2^\ast>0\) such that
\[K_G(\gamma)\leq\frac{1}{4}\quad\forall\,\gamma\geq\gamma_2^\ast.\]
Thus,
\begin{align}
    \label{eq:contr_operator_J_G_final}&\|\mathcal{J}_G \varphi - \mathcal{J}_G \psi\|_{\mathbb{V}_p,\gamma}\leq \frac{1}{4} \|\varphi - \psi\|_{\mathbb{V}_p,\gamma}\quad\forall \gamma\geq\gamma_2^{\ast}.
\end{align}
\noindent
Finally, the integral equation \eqref{eq:operator_J_def} combined with the estimates \eqref{eq:contr_operator_J_F_final} and \eqref{eq:contr_operator_J_G_final} yields:
\begin{align}
    \label{eq:contr_operator_J_estimate}&\|\mathcal{J}\varphi-\mathcal{J}\psi\|_{\mathbb{V}_p,\gamma}
\le
\frac{1}{2}\,
\|\varphi-\psi\|_{\mathbb{V}_p,\gamma}\quad\forall\, \varphi,\psi\in \mathbb{V}_p,\;\gamma\geq \gamma^\ast\coloneqq\max\{\gamma_1^\ast, \gamma_2^\ast\},
\end{align}
that is, the mapping $\mathcal{J}:\mathbb{V}_p\to \mathbb{V}_p$ is a strict contraction with respect to
the norm $\|\cdot\|_{\mathbb{V}_p,\gamma^\ast}$.
By the Banach fixed point theorem, $\mathcal{J}$ admits a unique fixed point
$X\in \mathbb{V}_p$, which is the unique mild solution of the problem \eqref{eq:SEE}.

\noindent
\paragraph{A priori bound.}
Let $X = \mathcal{J}X \in \mathbb{V}_p$ be the unique mild solution of the problem \eqref{eq:SEE} obtained above. Then the estimate \eqref{eq:contr_operator_J_estimate} yields:
\begin{align}
    \label{eq:aprior_bound1}& \|X\|_{\mathbb{V}_p,\gamma^\ast} = \|\mathcal{J}X\|_{\mathbb{V}_p,\gamma^\ast} \leq \|\mathcal{J}X - \mathcal{J}\boldsymbol{0}\|_{\mathbb{V}_p,\gamma^\ast}  + \|\mathcal{J}\boldsymbol{0}\|_{\mathbb{V}_p,\gamma^\ast} \leq \frac{1}{2}\|X\|_{\mathbb{V}_p,\gamma^\ast} + \|\mathcal{J}\boldsymbol{0}\|_{\mathbb{V}_p}. 
\end{align}
Now, applying Lemma~\ref{lem:semigroup_estimates}(i) and the estimates \eqref{eq:det_w_norm_bound} and \eqref{eq:stoc_w_norm_bound} of Lemma~\ref{lem:convolutions_regularity}, we obtain:
\begin{align}
    \nonumber \|\mathcal{J}\boldsymbol{0}\|_{\mathbb{V}_p} &\leq \|S(\cdot)X_0\|_{\mathbb{V}_p} + \|\mathcal{J}_F\boldsymbol{0}\|_{\mathbb{V}_p} + \|\mathcal{J}_G\boldsymbol{0}\|_{\mathbb{V}_p}\\
    \label{eq:aprior_bound2}&\leq C(1+\|X_0\|_{\dot{\mathbb{H}}^\mu}).
\end{align}
Finally, by utilizing the equivalence of \(\|\cdot\|_{\mathbb{V}_p,\gamma^\ast}\) and \(\|\cdot\|_{\mathbb{V}_p}\) norms and combining the above estimates \eqref{eq:aprior_bound1} and \eqref{eq:aprior_bound2}, we obtain the desired a priori estimate \eqref{eq:a_priori_estimate}. 


\noindent \textit{\textbf{Proof of Theorem~\ref{thm:main_global_existence}\textup{(ii)} -- Continuous dependence on the initial data.}}
Let $X, Y \in \mathbb{V}_p$ be the unique mild solutions corresponding to the initial data $X_0$ and $Y_0$, respectively. Then
\begin{equation*}
X(t) - Y(t) = S(t)(X_0 - Y_0) + (\mathcal{J}_F X)(t) - (\mathcal{J}_F Y)(t) + (\mathcal{J}_G X)(t) - (\mathcal{J}_G Y)(t).
\end{equation*}
An appeal to Lemma~\ref{lem:semigroup_estimates} and the estimates \eqref{eq:contr_operator_J_F_final} and \eqref{eq:contr_operator_J_G_final} yields:
\begin{align}
    & \|X - Y\|_{\mathbb{V}_p,\gamma^\ast} \leq C\|X_0 - Y_0\|_{L^p(\Omega;\dot{\mathbb{H}}^{\mu})} +\frac{1}{2}\|X - Y\|_{\mathbb{V}_p,\gamma^\ast}.
\end{align}
Finally, by applying the equivalence of the norms \(\|\cdot\|_{\mathbb{V}_p,\gamma^\ast}\) and \(\|\cdot\|_{\mathbb{V}_p}\), we obtain the desired estimate \eqref{eq:cont_dep_ini_data} which completes the proof of Theorem~\ref{thm:main_global_existence}(ii).

\noindent
\textit{\textbf{Proof of Theorem~\ref{thm:main_global_existence}\textup{(iii)} -- Temporal regularity.}}
To establish Hölder temporal regularity results, consider the following shifted integral equation satisfied by the mild solution \(X\) of \eqref{eq:SEE}: for \(0\leq t_0<t\leq T\),
\begin{align}\label{eq:shifted_int_eq}
 &X(t) = S(t-t_0)X(t_0) + (\mathcal{J}_F^{t_0} X)(t) + (\mathcal{J}_G^{t_0} X)(t)\quad\text{in }\dot{\mathbb{H}}^\mu\;\mathbb{P}\text{-a.s.}  
\end{align}
For any \(0<a<T\) and \(0<\epsilon\leq 1\), an appeal to Lemma~\ref{lem:semigroup_estimates}(iii) and the solution bound \eqref{eq:a_priori_estimate} in Theorem~\ref{thm:main_global_existence}(i) yields:
\begin{align}
    \nonumber\|S(t-a/2)X(a/2) - S(s-a/2)X(a/2)\|_{L^p(\Omega;\dot{\mathbb{H}}^\nu)}&\leq C(a/2)^{-(\delta+\lambda)}|t-s|^\delta (1+\|X_0\|_{L^p(\Omega;\dot{\mathbb{H}}^\mu)})\;\forall\,s,t\in[a,T],\;\text{and}\\
   \nonumber  \|S(t)X_0 - S(s)X_0\|_{L^p(\Omega;\dot{\mathbb{H}}^\mu)} &\leq C |t-s|^{\epsilon/2} (1+\|X_0\|_{L^p(\Omega;\dot{\mathbb{H}}^{\mu+\epsilon})})\;\forall\,s,t\in[0,T]
\end{align}
for all \(0<\delta<\min\big\{1-\frac{1}{2}(\nu+\alpha)^+, \frac{1}{2}-\frac{1}{2}(\nu+\beta)^+\big\}\), where \(\lambda=(\nu-\mu)/2\). Consequently,
\begin{align}
  \label{eq:initial_local_temp_holder}\|S(\cdot - a/2)X(a/2)\|_{\mathcal{C}^{0,\delta}([a,T];L^p(\Omega;\dot{\mathbb{H}}^\nu))}  &\leq C(a/2)^{-(\delta+\lambda)}(1+\|X_0\|_{L^p(\Omega;\dot{\mathbb{H}}^\mu)}),\quad\text{and}\\
  \label{eq:initial_global_temp_holder} \|S(\cdot )X_0\|_{\mathcal{C}^{0,\epsilon/2}([0,T];L^p(\Omega;\dot{\mathbb{H}}^\mu))}  &\leq C(1+\|X_0\|_{L^p(\Omega;\dot{\mathbb{H}}^{\mu+\epsilon})}).
\end{align}
Setting \(t_0=a/2\) in \eqref{eq:shifted_int_eq} and applying \eqref{eq:initial_local_temp_holder}, estimates \eqref{eq:det_con_local} and \eqref{eq:sto_con_local} of Lemma~\ref{lem:convolutions_regularity}, and solution bound \eqref{eq:a_priori_estimate}, yields: for any \(0<\delta<\min\big\{1-\frac{1}{2}(\nu+\alpha)^+,\frac{1}{2}-\frac{1}{2}(\nu+\beta)^+\big\}\),
\begin{align*}
    &X\in \mathcal{C}^{0,\delta}([a,T];L^p(\Omega;\dot{\mathbb{H}}^\nu)) \quad\text{with }\; \|X\|_{\mathcal{C}^{0,\delta}([a,T];L^p(\Omega;\dot{\mathbb{H}}^\nu))}\leq C_a(1+\|X_0\|_{L^p(\Omega;\dot{\mathbb{H}}^\mu)}).
\end{align*}
By setting \(t_0=0\) in \eqref{eq:shifted_int_eq}, and applying estimate \eqref{eq:initial_global_temp_holder} along with the estimates \eqref{eq:det_con_global} and \eqref{eq:sto_con_global} of Lemma~\ref{lem:convolutions_regularity}, and solution bound \eqref{eq:a_priori_estimate}, we obtain, for any \(0<\delta_\epsilon < \min\big\{\frac{\epsilon}{2}, 1-\hat{\alpha} - \lambda -\frac{1}{2}(\mu+\alpha)^+,\frac{1}{2}-\hat{\beta} - \lambda -\frac{1}{2}(\mu+\beta)^+\big\}\),
\begin{align*}
    &X\in \mathcal{C}^{0,\delta_\epsilon}([0,T];L^p(\Omega;\dot{\mathbb{H}}^\mu)) \quad\text{with } \|X\|_{\mathcal{C}^{0,\delta_\epsilon}([0,T];L^p(\Omega;\dot{\mathbb{H}}^\mu))}\leq C(1+\|X_0\|_{L^p(\Omega;\dot{\mathbb{H}}^{\mu+\epsilon})}).
\end{align*}
This completes the proof.
\end{proof}

\subsection{Proof of Theorem~\ref{thm:pathwise_regularity}}
We employ the factorization formula introduced by Da Prato, Kwapie\'n, and Zabczyk \cite{MR920798} to establish the pathwise regularity of the deterministic and the stochastic convolutions. Before starting the proof of Theorem~\ref{thm:pathwise_regularity}, we introduce some auxiliary results. First, we introduce the following lemma that generalizes Proposition~4.1 in Hong and Liu~\cite{MR3912731} by unifying their three parameter regimes into a single cohesive condition via $(\eta-\rho)^+$. Moreover, it recovers the exact Hölder exponent at the parameter threshold $\eta = \rho$, eliminating the loss of regularity that occurred in \cite[Proposition~4.1]{MR3912731}.

\begin{lemma}\label{lem:det_smoothing}
Let $\{S(t)\}_{t \ge 0}$ be the analytic $\mathcal{C}_0$-semigroup generated by $-A$. For any $0 \leq t_0 < T$, $0 < \theta \leq 1$, $1 < q \leq \infty$, and $\eta, \rho \in \mathbb{R}$ satisfying $0<\theta - \frac{1}{q} - \frac{1}{2}(\eta-\rho)^+ <1$, the fractional convolution operator
\[
(R_\theta \psi)(t) \coloneqq \int_{t_0}^t (t-\sigma)^{\theta-1} S(t-\sigma) \psi(\sigma) \, d\sigma, \quad t \in [t_0, T],
\]
defines a bounded linear mapping $R_\theta \in \mathcal{L}(L^q(t_0,T; \dot{\mathbb{H}}^\rho), \mathcal{C}^{0,\delta}([t_0,T]; \dot{\mathbb{H}}^\eta))$, where $\delta \coloneqq \theta - \frac{1}{q} - \frac{1}{2}(\eta-\rho)^+$.
\end{lemma}

\begin{proof}
Let $\psi \in L^q(t_0, T; \dot{\mathbb{H}}^\rho)$.  
We first establish spatial boundedness. Applying Lemma~\ref{lem:semigroup_estimates}(i) and Hölder's inequality with conjugate exponent $q' = q/(q-1)$, we obtain
\begin{align*}
    \|(R_\theta \psi)(t)\|_{\dot{\mathbb{H}}^\eta} 
    &\le \int_{t_0}^t (t-\sigma)^{\theta-1} \Big\| A^{\frac{1}{2}(\eta-\rho)} S(t-\sigma) \Big\|_{\mathcal{L}(\mathbb{H})} \|\psi(\sigma)\|_{\dot{\mathbb{H}}^\rho} \, d\sigma \\
    &\le C \int_{t_0}^t (t-\sigma)^{\theta-1 - \frac{1}{2}(\eta-\rho)^+} \|\psi(\sigma)\|_{\dot{\mathbb{H}}^\rho} \, d\sigma \\
    &\le C \left( \int_{t_0}^t (t-\sigma)^{\left(\theta-1-\frac{1}{2}(\eta-\rho)^+\right)q'} \, d\sigma \right)^{\frac{1}{q'}} \|\psi\|_{L^q(t_0,T; \dot{\mathbb{H}}^\rho)},
\end{align*}
where, under the hypothesis $\theta > \frac{1}{q} + \frac{1}{2}(\eta-\rho)^+$, the time integral converges for all $t \in [t_0,T]$ with $t_0 \ge 0$. Thus $\sup_{t \in [t_0,T]} \|(R_\theta \psi)(t)\|_{\dot{\mathbb{H}}^\eta} \le C \|\psi\|_{L^q(t_0,T; \dot{\mathbb{H}}^\rho)}$.

To establish Hölder continuity, let $t_0 \le t_1 < t_2 \le T$ and set $h \coloneqq t_2 - t_1$. We decompose the increment of the operator into three distinct parts:
\begin{align*}
    (R_\theta \psi)(t_2) - (R_\theta \psi)(t_1) 
    &= \int_{t_1}^{t_2} (t_2-\sigma)^{\theta-1} S(t_2-\sigma)\psi(\sigma) \, d\sigma \\
    &\quad + \int_{t_0}^{t_1} \big[(t_2-\sigma)^{\theta-1} - (t_1-\sigma)^{\theta-1}\big] S(t_2-\sigma) \psi(\sigma) \, d\sigma \\
    &\quad + \int_{t_0}^{t_1} (t_1-\sigma)^{\theta-1} \big[S(t_2-\sigma) - S(t_1-\sigma)\big] \psi(\sigma) \, d\sigma \\
    &\eqqcolon J_1 + J_2 + J_3.
\end{align*}

\noindent\textbf{Estimate of $J_1$:}
Proceeding exactly as in the spatial bound, direct integration gives
\[
    \|J_1\|_{\dot{\mathbb{H}}^\eta} \le C \left( \int_{t_1}^{t_2} (t_2-\sigma)^{\left(\theta-1-\frac{1}{2}(\eta-\rho)^+\right)q'} d\sigma \right)^{\frac{1}{q'}} \|\psi\|_{L^q(t_0,T; \dot{\mathbb{H}}^\rho)} \le C h^\delta \|\psi\|_{L^q(t_0,T; \dot{\mathbb{H}}^\rho)}.
\]

\noindent\textbf{Estimate of $J_2$:}
If $\theta = 1$, $J_2$ vanishes identically. For $0 < \theta < 1$, applying H\"older's inequality alongside the monotonicity $(t_2-\sigma)^{-\frac{1}{2}(\eta-\rho)^+ q'} \le (t_1-\sigma)^{-\frac{1}{2}(\eta-\rho)^+ q'}$ (for $\sigma< t_1 < t_2$) and the elementary algebraic inequality $(b-a)^p \le b^p - a^p$ (valid for $p \ge 1$ and $b \ge a \ge 0$), we obtain:
\begin{align*}
    \|J_2\|_{\dot{\mathbb{H}}^\eta}^{q'} 
    &\le C \|\psi\|_{L^q(t_0,T; \dot{\mathbb{H}}^\rho)}^{q'} \int_{t_0}^{t_1} \Big[ (t_1-\sigma)^{(\theta-1)q'} - (t_2-\sigma)^{(\theta-1)q'} \Big] (t_2-\sigma)^{-\frac{1}{2}(\eta-\rho)^+ q'} \, d\sigma \\
    &\le C \|\psi\|_{L^q(t_0,T; \dot{\mathbb{H}}^\rho)}^{q'} \int_{t_0}^{t_1} \Big[ (t_1-\sigma)^{\kappa} - (t_2-\sigma)^{\kappa} \Big] \, d\sigma,
\end{align*}
where $\kappa \coloneqq (\theta-1-\frac{1}{2}(\eta-\rho)^+)q' > -1$. Evaluating the integral exactly yields $\frac{1}{\kappa+1} \big[ (t_1-t_0)^{\kappa+1} - (t_2-t_0)^{\kappa+1} + (t_2-t_1)^{\kappa+1} \big]$. Since $\kappa + 1 > 0$ and $t_1 < t_2$, it follows that $(t_1-t_0)^{\kappa+1} - (t_2-t_0)^{\kappa+1} < 0$. Bounding the bracketed term by $(t_2-t_1)^{\kappa+1} = h^{\kappa+1}$, taking the $q'$-th root, and noting that $\frac{\kappa+1}{q'} = \delta$, we conclude $\|J_2\|_{\dot{\mathbb{H}}^\eta} \le C h^\delta \|\psi\|_{L^q(t_0,T; \dot{\mathbb{H}}^\rho)}$.

\noindent\textbf{Estimate of $J_3$:}
Using the identity $S(t_2-\sigma) - S(t_1-\sigma) = -\int_{t_1-\sigma}^{t_2-\sigma} A S(u) \, du$, along with Lemma~\ref{lem:semigroup_estimates}(i) and H\"older's inequality, we bound $\|J_3\|_{\dot{\mathbb{H}}^\eta}$ as follows:
\begin{align*}
    \|J_3\|_{\dot{\mathbb{H}}^\eta} 
    &\leq \int_{t_0}^{t_1} (t_1-\sigma)^{\theta-1} \int_{t_1-\sigma}^{t_2-\sigma} \big\| A^{1+\frac{1}{2}(\eta-\rho)} S(u) \big\|_{\mathcal{L}(\mathbb{H})} \|\psi(\sigma)\|_{\dot{\mathbb{H}}^\rho} \, du \, d\sigma \\
    &\leq C \int_{t_0}^{t_1} (t_1-\sigma)^{\theta-1} \left( \int_{t_1-\sigma}^{t_2-\sigma} u^{-1-\frac{1}{2}(\eta-\rho)^+} \, du \right) \|\psi(\sigma)\|_{\dot{\mathbb{H}}^\rho} \, d\sigma \\
    &\leq C \|\psi\|_{L^q(t_0,T;\dot{\mathbb{H}}^\rho)} \left( \int_{t_0}^{t_1} \left( (t_1-\sigma)^{\theta-1} \int_{t_1-\sigma}^{t_2-\sigma} u^{-1-\frac{1}{2}(\eta-\rho)^+} \, du \right)^{q^\prime} d\sigma \right)^{\frac{1}{q^\prime}}.
\end{align*}
Letting $h = t_2 - t_1$ and applying the substitutions $t_1 - \sigma = s h$ and $u = \tau h$, we deduce:
\begin{align*}
    \|J_3\|_{\dot{\mathbb{H}}^\eta} &\leq C h^{\theta - \frac{1}{q}-\frac{1}{2}(\eta-\rho)^+} \|\psi\|_{L^q(t_0,T;\dot{\mathbb{H}}^\rho)} \left( \int_{0}^{\frac{t_1-t_0}{h}} \left( s^{\theta-1} \int_{s}^{s+1} \tau^{-1-\frac{1}{2}(\eta-\rho)^+} \, d\tau \right)^{q^\prime} ds \right)^{\frac{1}{q^\prime}}\\
    &\leq C h^\delta \|\psi\|_{L^q(t_0,T;\dot{\mathbb{H}}^\rho)} \left( \int_0^\infty \left( s^{\theta-1} \int_{s}^{s+1} \tau^{-1-\frac{1}{2}(\eta-\rho)^+} \, d\tau \right)^{q^\prime} ds \right)^{\frac{1}{q^\prime}}.
\end{align*}
Under the hypothesis $0<\theta - \frac{1}{q} - \frac{1}{2}(\eta-\rho)^+ <1$, the infinite integral on the right-hand side converges. Finally, combining the estimates for $J_1$, $J_2$, and $J_3$ concludes the proof.
\end{proof}

We establish the pathwise H\"older regularity of the deterministic and the stochastic convolutions in the following lemma.

\begin{lemma} \label{lem:convolutions_pathwise_regularity}
    Let $X_0 \in L^p(\Omega,\mathcal{F}_0; \dot{\mathbb{H}}^\mu)$ with $p > 2$, and let $X$ be the mild solution to \eqref{eq:SEE}. Suppose that Assumptions~\ref{ass:measurability_drift_diff} and~\ref{ass:global_lips_nonlin} hold for a set of parameters satisfying Assumption~\ref{ass:parameters}. For $0 \leq t_0 \leq t \leq T$, we define the shifted convolutions:
\begin{equation*}
    (\mathcal{J}_F^{t_0}X)(t) \coloneqq \int_{t_0}^t S(t-\sigma)F(\sigma,X(\sigma))\,d\sigma \quad \text{and} \quad (\mathcal{J}_G^{t_0}X)(t) \coloneqq \int_{t_0}^t S(t-\sigma)G(\sigma,X(\sigma))\,dW(\sigma),
\end{equation*}
with $\mathcal{J}_F \coloneqq \mathcal{J}_F^0$ and $\mathcal{J}_G \coloneqq \mathcal{J}_G^0$. Then, there exist positive constants $C$ and $C_a$ (where $C_a$ may diverge as $a \searrow 0$) such that these convolutions admit versions (which we still denote using the same notation) satisfying the following regularity assertions:
\begin{enumerate}[label={\upshape(\roman*)}]
    \item \textbf{\textit{Deterministic convolution:}} If $(\eta+\alpha)^+ < \min\left\{2 - \frac{2}{p}, 2 - 2\hat{\alpha} - 2\lambda\right\}$ and $0 < \delta < \min\big\{1 - \hat{\alpha} - \lambda - \frac{1}{2}(\eta+\alpha)^+, 1 - \frac{1}{p} - \frac{1}{2}(\eta+\alpha)^+\big\}$, then $\mathcal{J}_F(X) \in L^p\left(\Omega; \mathring{\mathcal{C}}^{0, \delta}([0, T]; \dot{\mathbb{H}}^\eta)\right)$, and the following global estimate holds:
    \begin{equation}\label{eq:det_con_global_pathwise}
        \|\mathcal{J}_FX\|_{L^p(\Omega;\mathring{\mathcal{C}}^{0, \delta}([0,T];\dot{\mathbb{H}}^{\eta}))} \leq C\left(1+\|X_0\|_{L^p(\Omega;\dot{\mathbb{H}}^{\mu})}\right).
    \end{equation}
    Furthermore, for any $a > 0$, under the relaxed constraints $(\eta+\alpha)^+ < 2 - \frac{2}{p}$ and $0 < \delta < 1 - \frac{1}{p} - \frac{1}{2}(\eta+\alpha)^+$, we obtain the local bound:
    \begin{equation}\label{eq:det_con_local_pathwise}
        \|\mathcal{J}_F^aX\|_{L^p(\Omega;\mathring{\mathcal{C}}^{0, \delta}([a,T];\dot{\mathbb{H}}^{\eta}))} \leq C_a\left(1+\|X_0\|_{L^p(\Omega;\dot{\mathbb{H}}^{\mu})}\right).
    \end{equation}

    \item \textbf{\textit{Stochastic convolution:}} If $(\eta+\beta)^+ < \min\left\{1 - \frac{2}{p}, 1 - 2\hat{\beta} - 2\lambda\right\}$ and $0 < \delta < \min\big\{\frac{1}{2} - \hat{\beta} - \lambda - \frac{1}{2}(\eta+\beta)^+, \frac{1}{2} - \frac{1}{p} - \frac{1}{2}(\eta+\beta)^+\big\}$, then $\mathcal{J}_G(X) \in L^p\left(\Omega; \mathring{\mathcal{C}}^{0, \delta}([0, T]; \dot{\mathbb{H}}^\eta)\right)$, satisfying the global estimate:
    \begin{equation}\label{eq:sto_con_global_pathwise}
        \|\mathcal{J}_GX\|_{L^p(\Omega;\mathring{\mathcal{C}}^{0, \delta}([0,T];\dot{\mathbb{H}}^{\eta}))} \leq C\left(1+\|X_0\|_{L^p(\Omega;\dot{\mathbb{H}}^{\mu})}\right).
    \end{equation}
    Similarly, for any $a > 0$, under the relaxed constraints $(\eta+\beta)^+ < 1 - \frac{2}{p}$ and $0 < \delta < \frac{1}{2} - \frac{1}{p} - \frac{1}{2}(\eta+\beta)^+$, we obtain the local bound:
    \begin{equation}\label{eq:sto_con_local_pathwise}
        \|\mathcal{J}_G^aX\|_{L^p(\Omega;\mathring{\mathcal{C}}^{0, \delta}([a,T];\dot{\mathbb{H}}^{\eta}))} \leq C_a\left(1+\|X_0\|_{L^p(\Omega;\dot{\mathbb{H}}^{\mu})}\right).
    \end{equation}
\end{enumerate}
\end{lemma}

\begin{proof}
An appeal to the (stochastic) Fubini theorem yields the following factorization formulae: for all \(0\leq t_0 \leq t\leq T\),
\begin{align}
    \label{eq:factorization_formula_F}(\mathcal{J}_F^{t_0}X)(t) &= \frac{\sin(\pi\theta)}{\pi}(R_\theta F_\theta)(t) \quad\text{ in }\;\dot{\mathbb{H}}^{-\alpha}\quad\mathbb{P}\text{-a.s.}\quad \forall\; 0<\theta<1,\\
    \label{eq:factorization_formula_G}(\mathcal{J}_G^{t_0}X)(t) &= \frac{\sin(\pi\theta)}{\pi}{\pi}(R_\theta G_\theta)(t) \quad\text{ in }\;\dot{\mathbb{H}}^{-\beta}\quad\mathbb{P}\text{-a.s.}\quad \forall\;0<\theta<1/2,
\end{align}
where the operator \(R_\theta\) is defined in Lemma~\ref{lem:det_smoothing} and 
\begin{align*}
    F_\theta(t) \coloneqq \int_{t_0}^t(t-r)^{-\theta}S(t-r)F(r,X(r))\,dr, \quad\text{and}\quad     G_\theta(t) \coloneqq \int_{t_0}^t(t-r)^{-\theta}S(t-r)G(r,X(r))\,dW(r).
\end{align*}
For \(p>2\), suppose that \((t_0,\theta)\in[0,T]\times(0,1)\) and \((t_0,\theta)\in[0,T]\times(0,\tfrac12)\) are chosen in such a way that
\[
\int_{t_0}^T
\left(
\int_{t_0}^t
(t-\sigma)^{-\theta}\sigma^{-\hat{\alpha}-\lambda}
\,d\sigma
\right)^p
dt
<\infty\quad\text{ and }\quad\int_{t_0}^T
\left(
\int_{t_0}^t
(t-\sigma)^{-2\theta}\sigma^{-2\hat{\beta}-2\lambda}
\,d\sigma
\right)^{p/2}
dt
<\infty,
\]
respectively.

Now, by applying Tonelli's theorem and Minkowski’s integral inequality, we estimate \(F_\theta\):
\begin{align}
\nonumber\|F_\theta\|_{L^p(\Omega;L^p(t_0,T;\dot{\mathbb{H}}^{-\alpha}))}^p &= \int_{t_0}^T\mathbb{E}\left[\|F_\theta(t)\|_{\dot{\mathbb{H}}^{-\alpha}}^p\right]\,dt\\
\label{eq:ftheta_smoothing_1}&\leq  \int_{t_0}^T\left(\int_{t_0}^t(t-\sigma)^{-\theta}\|S(t-\sigma)F(\sigma,X(\sigma))\|_{L^p(\Omega; \dot{\mathbb{H}}^{-\alpha})}\,d\sigma\right)^{p}\,dt.
\end{align}
Similarly, employing Tonelli's theorem, the Burkholder-Davis-Gundy  inequality (Lemma~\ref{lem:BDG}), and  Minkowski’s integral inequality, we bound \(G_\theta\):
\begin{align}
\nonumber\|G_\theta\|_{L^p(\Omega;L^p(t_0,T;\dot{\mathbb{H}}^{-\beta}))}^p &= \int_{t_0}^T\mathbb{E}\left[\|G_\theta(t)\|_{\dot{\mathbb{H}}^{-\beta}}^p\right]\,dt\\
\nonumber&\leq C \int_{t_0}^T\mathbb{E}\left[\left(\int_{t_0}^t(t-\sigma)^{-2\theta}\|S(t-\sigma)G(\sigma,X(\sigma))\|_{\mathrm{HS}(U_0,\dot{\mathbb{H}}^{-\beta})}^2\,d\sigma\right)^{p/2}\right]\,dt\\
\label{eq:gtheta_smoothing_1}&\leq C \int_{t_0}^T\left(\int_{t_0}^t(t-\sigma)^{-2\theta}\|S(t-\sigma)G(\sigma,X(\sigma))\|_{L^p(\Omega; \mathrm{HS}(U_0,\dot{\mathbb{H}}^{-\beta}))}^2\,d\sigma\right)^{p/2}\,dt.
\end{align}
Invoking the boundedness of the analytic \(\mathcal{C}_0\)-semigroup \(\{S(t)\}_{t\geq 0}\) (Lemma~\ref{lem:semigroup_estimates}~(i)), the nonlinearity bounds from  Lemma~\ref{lem:Bounds_on_F_and_G}, and the a priori mild solution estimate (\ref{eq:a_priori_estimate}) from Theorem~\ref{thm:main_global_existence}, the inequalities \eqref{eq:ftheta_smoothing_1} and \eqref{eq:gtheta_smoothing_1} yield:
\begin{align}
\label{eq:ftheta_smoothing_2}\|F_\theta\|_{L^p(\Omega;L^p(t_0,T;\dot{\mathbb{H}}^{-\alpha}))} 
&\leq C (1+\|X_0\|_{L^p(\Omega;\dot{\mathbb{H}}^{\mu})})\left[\int_{t_0}^T\left(\int_{t_0}^t(t-\sigma)^{-\theta}\sigma^{-\hat{\alpha}-\lambda}\,d\sigma\right)^{p}\,dt\right]^{1/p},\quad\text{and}\\
\label{eq:gtheta_smoothing_2}\|G_\theta\|_{L^p(\Omega;L^p(t_0,T;\dot{\mathbb{H}}^{-\beta}))} 
&\leq C (1+\|X_0\|_{L^p(\Omega;\dot{\mathbb{H}}^{\mu})})\left[\int_{t_0}^T\left(\int_{t_0}^t(t-\sigma)^{-2\theta}\sigma^{-2\hat{\beta}-2\lambda}\,d\sigma\right)^{p/2}\,dt\right]^{1/p}.
\end{align}
Notice that the integrals in \eqref{eq:ftheta_smoothing_2} and \eqref{eq:gtheta_smoothing_2} are finite for \((t_0,\theta)\in (\{a\}\times(0,1))\cup (\{0\}\times(0,\theta_F))\) and \((t_0,\theta)\in (\{a\}\times(0,1/2))\cup (\{0\}\times(0,\theta_G))\),  respectively, where \(\theta_F \coloneqq \min\{1,1 +1/p-\hat{\alpha} - \lambda\}\) and \(\theta_G \coloneqq \min\{1/2,1/2 +1/p-\hat{\beta} - \lambda\}\). For such parameter pairs \((t_0,\theta)\), applying Lemma~\ref{lem:det_smoothing} to \eqref{eq:ftheta_smoothing_2} and \eqref{eq:gtheta_smoothing_2} results in:
\begin{align*}
\|\mathcal{J}_F^{t_0}X\|_{L^p(\Omega;\mathring{\mathcal{C}}^{0, \delta}([t_0,T];\dot{\mathbb{H}}^{\eta}))} &\leq C \|F_\theta\|_{L^p(\Omega;L^p(t_0,T;\dot{\mathbb{H}}^{-\alpha}))}\\
&\leq  C (1+\|X_0\|_{L^p(\Omega;\dot{\mathbb{H}}^{\mu})})\left[\int_{t_0}^T\left(\int_{t_0}^t(t-\sigma)^{-\theta}\sigma^{-\hat{\alpha}-\lambda}\,d\sigma\right)^{p}\,dt\right]^{1/p}\quad\forall\,0<\delta<\delta_F,
\end{align*}
and
\begin{align*}
\|\mathcal{J}_G^{t_0}X\|_{L^p(\Omega;\mathring{\mathcal{C}}^{0, \delta}([t_0,T];\dot{\mathbb{H}}^{\eta}))} &\leq C \|G_\theta\|_{L^p(\Omega;L^p(t_0,T;\dot{\mathbb{H}}^{-\beta}))}\\
&\leq  C (1+\|X_0\|_{L^p(\Omega;\dot{\mathbb{H}}^{\mu})})\left[\int_{t_0}^T\left(\int_{t_0}^t(t-\sigma)^{-2\theta}\sigma^{-2\hat{\beta}-2\lambda}\,d\sigma\right)^{p/2}\,dt\right]^{1/p}\quad\forall\,0<\delta<\delta_G,
\end{align*}
respectively, where
\begin{align*}
    \delta_F& \coloneqq \begin{cases}
       \min\{1-1/p - (\eta+\alpha)^+/2,1-\hat{\alpha} - \lambda -(\eta+\alpha)^+/2\} &: \, (t_0,\theta)\in\{0\}\times (0,\theta_F)\\[4pt]
       1-1/p - (\eta+\alpha)^+/2 &: \, (t_0,\theta)\in\{a\}\times (0,1),
    \end{cases}
\end{align*}
and
\begin{align*}
    \delta_G& \coloneqq \begin{cases}
       \min\{1/2-1/p - (\eta+\beta)^+/2,1/2-\hat{\beta} - \lambda -(\eta+\beta)^+/2\} &: \, (t_0,\theta)\in\{0\}\times (0,\theta_G)\\[4pt]
       1/2-1/p - (\eta+\beta)^+/2 &: \, (t_0,\theta)\in\{a\}\times (0,1/2).
    \end{cases}
\end{align*}
This establishes the required global and local regularity estimates, which concludes the proof.
\end{proof}

\begin{proof}[\textbf{Proof of Theorem~\ref{thm:pathwise_regularity}}]
The hypothesis $X_0 \in L^p(\Omega,\mathcal{F}_0; \dot{\mathbb{H}}^\mu)$ with $p > 2$ ensures that $X_0 \in \dot{\mathbb{H}}^\mu$ $\mathbb{P}$-a.s. The strong continuity of the analytic $\mathcal{C}_0$-semigroup $\{S(t)\}_{t \geq 0}$ ensures that $S(\cdot)X_0 \in \mathcal{C}([0,T]; \dot{\mathbb{H}}^\mu)$ $\mathbb{P}$-a.s. Applying Lemma~\ref{lem:semigroup_estimates}(i), we obtain the uniform bound $\sup_{t \in [0,T]} \|S(t)X_0\|_{\dot{\mathbb{H}}^\mu} \leq C \|X_0\|_{\dot{\mathbb{H}}^\mu}$ $\mathbb{P}$-a.s., which yields the global continuity estimate:
\begin{equation} \label{eq:proof_initial_global_cont}
    \|S(\cdot)X_0\|_{L^p(\Omega; \mathcal{C}([0,T]; \dot{\mathbb{H}}^\mu))} \leq C \|X_0\|_{L^p(\Omega; \dot{\mathbb{H}}^\mu)}.
\end{equation}
To establish local Hölder regularity, we invoke Theorem~\ref{thm:main_global_existence}, which guarantees $X(a/2) \in \dot{\mathbb{H}}^\mu$ $\mathbb{P}$-a.s.\ and $\|X(a/2)\|_{L^p(\Omega;\dot{\mathbb{H}}^\mu)} \leq C (1+\|X_0\|_{L^p(\Omega;\dot{\mathbb{H}}^\mu)})$ for any $a \in (0, T]$. Lemma~\ref{lem:semigroup_estimates}(iii) implies that for any $\nu \geq \mu$ and $s, t \in [a, T]$, the increment satisfies $\|S(t-a/2)X(a/2) - S(s-a/2)X(a/2)\|_{\dot{\mathbb{H}}^\nu} \leq C |t-s|^\delta (a/2)^{-\frac{1}{2}(2\delta+\nu-\mu)^+}\|X(a/2)\|_{\dot{\mathbb{H}}^\mu}$ $\mathbb{P}$-a.s. For any $0< \delta<\min\big\{1 - \frac{1}{p} - \frac{1}{2}(\nu+\alpha)^+, \frac{1}{2} - \frac{1}{p} - \frac{1}{2}(\nu+\beta)^+\big\}$, this integrates to the bound:
\begin{equation} \label{eq:proof_initial_local_holder}
    \|S(\cdot-a/2)X(a/2)\|_{L^p(\Omega; \mathring{\mathcal{C}}^{0, \delta}([a,T]; \dot{\mathbb{H}}^\nu))} \leq C (a/2)^{-\frac{1}{2}(2\delta+\nu-\mu)^+} \left(1+\|X_0\|_{L^p(\Omega; \dot{\mathbb{H}}^\mu)}\right).
\end{equation}
Under the enhanced spatial regularity assumption $X_0 \in L^p(\Omega; \dot{\mathbb{H}}^{\mu+\epsilon})$ for some $\epsilon \in (0, 1]$, Lemma~\ref{lem:semigroup_estimates}(ii) provides the uniform bound $\|S(t)X_0 - S(s)X_0\|_{\dot{\mathbb{H}}^\mu} \leq C |t-s|^{\epsilon/2} \|X_0\|_{\dot{\mathbb{H}}^{\mu+\epsilon}}$ $\mathbb{P}$-a.s. Taking the $L^p(\Omega)$-norm directly yields global H\"older continuity for any $0< \delta < \epsilon/2$:
\begin{equation} \label{eq:proof_initial_global_holder}
    \|S(\cdot)X_0\|_{L^p(\Omega; \mathring{\mathcal{C}}^{0, \delta}([0,T]; \dot{\mathbb{H}}^\mu))} \leq C \|X_0\|_{L^p(\Omega; \dot{\mathbb{H}}^{\mu+\epsilon})}.
\end{equation}
Finally, the mild solution $X$ satisfies the shifted integral equation, for all \(0\leq t_0<t\leq T\):
\begin{equation}\label{eq:shifted_mild_int_eqn}
    X(t) = S(t-t_0)X(t_0) + (\mathcal{J}_F^{t_0}X)(t) + (\mathcal{J}_G^{t_0}X)(t), \quad \mathbb{P}\text{-a.s. in } \dot{\mathbb{H}}^\mu.
\end{equation}
The assertions of Theorem~\ref{thm:pathwise_regularity} follow by combining \eqref{eq:shifted_mild_int_eqn} and the initial condition estimates \eqref{eq:proof_initial_global_cont}--\eqref{eq:proof_initial_global_holder} with the pathwise convolution bounds of Lemma~\ref{lem:convolutions_pathwise_regularity}. Specifically, we set $\eta=\mu$ and $t_0=0$ to utilize the global estimates \eqref{eq:det_con_global_pathwise} and \eqref{eq:sto_con_global_pathwise}, and $\eta=\nu$ and $t_0=a/2$ for the local estimates \eqref{eq:det_con_local_pathwise} and \eqref{eq:sto_con_local_pathwise}.
\end{proof}

\subsection{Proof of Theorem~\ref{thm:general_initial_value}}
To prove Theorem~\ref{thm:general_initial_value}, we first require several auxiliary results. We begin by stating a lemma that will be used to establish uniqueness throughout this work; its proof relies on a generalized Gr\"onwall inequality (see \cite[Lemma~2.3]{MR1213834}).

\begin{lemma}\label{lem:gtype}
Let $\psi \in L^\infty(0,T)$ satisfy
\begin{equation*}
0 \leq \psi(t) \leq \sum_{j=1}^n C_j t^{\alpha_j} \int_0^t (t-\sigma)^{\beta_j - 1} \sigma^{\gamma_j - 1} \psi(\sigma) \, d\sigma \quad \text{for all } t \in [0,T],
\end{equation*}
where $\alpha_j \ge 0$, $C_j, \beta_j, \gamma_j \in (0,\infty)$ and $\alpha_j + \beta_j + \gamma_j > 1$. Then $\psi = 0$ a.e. on $[0,T]$.
\end{lemma}

Next, we establish a local uniqueness result.

\begin{lemma}\label{lem:local_uniqueness_gen_ini_data}
Suppose Assumptions~\ref{ass:measurability_drift_diff} and~\ref{ass:global_lips_nonlin} hold for a set of parameters satisfying Assumption~\ref{ass:parameters}. For $j \in \{1,2\}$, let $X_j$ be a mild solution of \eqref{eq:SEE} with initial value $X_{j,0} \in L^0(\Omega, \mathcal{F}_0; \dot{\mathbb{H}}^\mu)$. Furthermore, suppose $X_j$ satisfies the following pathwise regularity $\mathbb{P}$-almost surely: for any $a \in (0,T)$,
\begin{align}
& X_j \in \mathcal{C}\big([0,T]; \dot{\mathbb{H}}^{\mu}\big) \cap \mathcal{C}\big([a,T]; \dot{\mathbb{H}}^{\nu}\big) \quad \text{with } X_j(0) = X_{j,0}, \label{eq:global_existence_1} \\
&\lim_{t \searrow 0} \|X_j(t)\|_{\lambda,t} = \|X_{j,0}\|_{\dot{\mathbb{H}}^\mu}. \label{eq:global_existence_2}
\end{align}
Then, for all $t \in [0,T]$,
\begin{equation*}
\|X_2(t) - X_1(t)\|_{\lambda,t} = 0 \quad \mathbb{P}\text{-a.s. on } \Gamma,
\end{equation*}
where $\Gamma \coloneqq \{\omega \in \Omega : \|X_{2,0}(\omega) - X_{1,0}(\omega)\|_{\dot{\mathbb{H}}^\mu} = 0\}$ is the event of coincident initial data, and the time-weighted norm $\|\cdot\|_{\lambda,t}$ with $\lambda \coloneqq \frac{\nu-\mu}{2}$ is defined in \eqref{eq:time_weighted_norm}.
\end{lemma}

\begin{proof}
We present the proof for $\mu < \nu$ (i.e., $\lambda \coloneqq (\nu-\mu)/2 > 0$). The case $\mu = \nu$ follows an identical argument by setting $\lambda = 0$ and replacing $\nu$ with $\mu$. 

By definition, $\|X_2(0) - X_1(0)\|_{\lambda,0} = \|X_{2,0} - X_{1,0}\|_{\dot{\mathbb{H}}^\mu} = 0$ $\mathbb{P}$-almost surely on $\Gamma$. Since $\dot{\mathbb{H}}^\nu$ is continuously embedded in $\dot{\mathbb{H}}^\mu$, it suffices to show that for $t \in (0,T]$,
\begin{equation*}
\|X_2(t) - X_1(t)\|_{\dot{\mathbb{H}}^\nu} = 0 \quad \mathbb{P}\text{-a.s. on } \Gamma.
\end{equation*}
The adaptedness of the processes $\{X_j\}_{t \in [0,T]}$ and the regularity properties \eqref{eq:global_existence_1} and \eqref{eq:global_existence_2} imply that $\{\|X_j(t)\|_{\lambda,t}\}_{t \in [0,T]}$ is an adapted and pathwise continuous process. Consequently, for each $R > 0$, 
\begin{equation*}
\tau_{R} \coloneqq T \wedge \inf \Bigl\{ t \in [0,T] : \|X_1(t)\|_{\lambda,t} + \|X_2(t)\|_{\lambda,t} \geq R \Bigr\}\quad \Big(\inf \emptyset = \infty\Big)
\end{equation*}
is a well-defined stopping time. Defining $\xi_R(t) \coloneqq \mathbbm{1}_{\llbracket 0, \tau_{R} \rrbracket}(t) \mathbbm{1}_{\Gamma}$ and $X_{j,R} \coloneqq \xi_R X_j$ for $j \in \{1,2\}$, we deduce from \eqref{eq:mild_solution_formula} that for all $t \in [0,T]$,
\begin{align*}
X_{2,R}(t) - X_{1,R}(t) &= \xi_R(t) \int_0^t \xi_R(\sigma) S(t-\sigma) \big(F(\sigma, X_{2,R}(\sigma)) - F(\sigma, X_{1,R}(\sigma))\big) \, d\sigma \\
&\quad + \xi_R(t) \int_0^t \xi_R(\sigma) S(t-\sigma) \big(G(\sigma, X_{2,R}(\sigma)) - G(\sigma, X_{1,R}(\sigma))\big) \, dW(\sigma) \quad \mathbb{P}\text{-a.s.}
\end{align*}
Applying Lemma~\ref{lem:BDG} yields the following estimate for $t \in [0,T]$:
\begin{align*}
\mathbb{E}\left[t^{2\lambda} \|X_{2,R}(t) - X_{1,R}(t)\|_{\dot{\mathbb{H}}^\nu}^2\right] &\leq C t^{2\lambda} \mathbb{E}\left[\left(\int_0^t \xi_R(\sigma) \|S(t-\sigma) \big(F(\sigma, X_{2,R}(\sigma)) - F(\sigma, X_{1,R}(\sigma))\big)\|_{\dot{\mathbb{H}}^\nu} \, d\sigma \right)^2\right] \\
&\quad + C t^{2\lambda} \mathbb{E}\left[\int_0^t \xi_R(\sigma) \|S(t-\sigma) \big(G(\sigma, X_{2,R}(\sigma)) - G(\sigma, X_{1,R}(\sigma))\big)\|_{\mathrm{HS}(U_0,\dot{\mathbb{H}}^{\nu})}^2 \, d\sigma\right].
\end{align*}
Under Assumption~\ref{ass:global_lips_nonlin}, invoking Lemma~\ref{lem:semigroup_estimates} provides
\begin{align*}
\mathbb{E}\left[t^{2\lambda} \|X_{2,R}(t) - X_{1,R}(t)\|_{\dot{\mathbb{H}}^\nu}^2\right] &\leq C t^{2\lambda} \mathbb{E}\left[\left(\int_0^t (t-\sigma)^{-\frac{1}{2}(\alpha+\nu)^+} \sigma^{-\hat{\alpha}} \|X_{2,R}(\sigma) - X_{1,R}(\sigma)\|_{\dot{\mathbb{H}}^\nu} \, d\sigma \right)^2\right] \\
&\quad + C t^{2\lambda} \int_0^t (t-\sigma)^{-(\beta+\nu)^+} \sigma^{-2\hat{\beta}} \mathbb{E}\left[\|X_{2,R}(\sigma) - X_{1,R}(\sigma)\|_{\dot{\mathbb{H}}^\nu}^2\right] \, d\sigma.
\end{align*}
Imposing the parameter constraints \eqref{eq:parameter_constraints}, we apply H\"older's inequality and Fubini's theorem to obtain
\begin{equation}\label{eq:R_local_uniqueness_2}
\begin{aligned}
\varphi_R(t) &\leq C t^{\lambda+1-\hat{\alpha}-\frac{1}{2}(\nu+\alpha)^+} \int_0^t (t-\sigma)^{-\frac{1}{2}(\alpha+\nu)^+} \sigma^{-\hat{\alpha}-\lambda} \varphi_R(\sigma) \, d\sigma \\
&\quad + C t^{2\lambda} \int_0^t (t-\sigma)^{-(\beta+\nu)^+} \sigma^{-2\hat{\beta}-2\lambda} \varphi_R(\sigma) \, d\sigma,
\end{aligned}
\end{equation}
for all $t \in [0,T]$, where $\varphi_R(t) \coloneqq \mathbb{E}\big[t^{2\lambda} \|X_{2,R}(t) - X_{1,R}(t)\|_{\dot{\mathbb{H}}^\nu}^2\big]$. Invoking Lemma~\ref{lem:gtype} in conjunction with the pathwise regularity conditions \eqref{eq:global_existence_1} and \eqref{eq:global_existence_2} ensures that $\varphi_R(t) = 0$ for all $t \in [0,T]$. Consequently, 
\begin{equation*}
t^{\lambda} \|X_{2,R}(t) - X_{1,R}(t)\|_{\dot{\mathbb{H}}^\nu} = 0 \quad \mathbb{P}\text{-a.s.}
\end{equation*}
Letting $R \to \infty$ concludes the proof.
\end{proof}

Our next result concerns the modified problem obtained by localizing the initial data.

\begin{lemma}\label{lem:localized_prob_with_localized_initial_data}
For each $N \in \mathbb{N}$ and $p \geq 2$, the localized stochastic initial value problem
\begin{equation*}
\begin{cases}
dX^N(t)+ AX^N(t)\, dt = F(t, X^N(t))  \, dt + G(t, X^N(t)) \, dW(t), & t \in (0,T], \\
X^N(0) = X_0^N \coloneqq \mathbbm{1}_{\Omega_N} X_0, \quad \Omega_N \coloneqq \{\omega \in \Omega : \|X_0(\omega)\|_{\dot{\mathbb{H}}^\mu} \leq N\},
\end{cases}
\end{equation*}
admits a unique mild solution $X^N \in \mathbb{V}_p$ satisfying the following regularity properties $\mathbb{P}$-almost surely:
\begin{enumerate}[label={\upshape(\roman*)}]
    \item For any $a > 0$ and $0 < \delta < \min\big\{1 - \frac{1}{2}(\nu+\alpha)^+, \frac{1}{2} - \frac{1}{2}(\nu+\beta)^+\big\}$,
    \begin{equation*}
    X^N \in \mathcal{C}([0,T]; \dot{\mathbb{H}}^\mu) \cap \mathring{\mathcal{C}}^{0, \delta}([a,T]; \dot{\mathbb{H}}^\nu).
    \end{equation*}
    \item For $\lambda \coloneqq (\nu-\mu)/2 > 0$,
    \begin{equation*}
    \lim_{t \searrow 0} t^\lambda \|X^N(t)\|_{\dot{\mathbb{H}}^\nu} = 0.
    \end{equation*}
\end{enumerate}
\end{lemma}

\begin{proof}
Since $\mathbbm{1}_{\Omega_N} X_0 \in L^{\infty}(\Omega; \dot{\mathbb{H}}^\mu)$ for each $p \geq 2$, Theorem~\ref{thm:main_global_existence} guarantees the existence of a unique mild solution $X^N \in \mathbb{V}_p$ satisfying, for all $t \in [0,T]$,
\begin{equation}\label{eq:mild_solution_truncated_gen_ini_value}
X^N(t) = S(t)X_0^N + \int_0^t S(t-\sigma) F(\sigma, X^N(\sigma)) \, d\sigma + \int_0^t S(t-\sigma) G(\sigma, X^N(\sigma)) \, dW(\sigma) \quad \text{in } \dot{\mathbb{H}}^\mu \text{ } \mathbb{P}\text{-a.s.}
\end{equation}
Furthermore, since $X^N \in \mathbb{V}_p$ for all $p \geq 2$, Theorem~\ref{thm:pathwise_regularity} ensures that $X^N$ exhibits the stated pathwise regularity in (i).

To establish (ii), we multiply both sides of \eqref{eq:mild_solution_truncated_gen_ini_value} by $t^\lambda$ and demonstrate that all three resulting terms on the right-hand side vanish as $t \searrow 0$ $\mathbb{P}$-almost surely. Since $X_0^N(\omega) \in \dot{\mathbb{H}}^\mu$ and $\dot{\mathbb{H}}^\nu$ is dense in $\dot{\mathbb{H}}^\mu$, for any $\epsilon > 0$, there exists $Y \in \dot{\mathbb{H}}^\nu$ such that $\|X_0^N(\omega) - Y\|_{\dot{\mathbb{H}}^\mu} < \epsilon$. By Lemma~\ref{lem:semigroup_estimates}(i),
\begin{equation*}
t^\lambda \|S(t)X_0^N(\omega)\|_{\dot{\mathbb{H}}^\nu} \leq t^\lambda \|S(t)(X_0^N(\omega) - Y)\|_{\dot{\mathbb{H}}^\nu} + t^\lambda \|S(t)Y\|_{\dot{\mathbb{H}}^\nu} \leq C\|X_0^N(\omega) - Y\|_{\dot{\mathbb{H}}^\mu} + t^\lambda \|Y\|_{\dot{\mathbb{H}}^\nu}.
\end{equation*}
Consequently, $\limsup_{t \searrow 0} t^\lambda \|S(t)X_0^N(\omega)\|_{\dot{\mathbb{H}}^\nu} \leq C\epsilon$. Since $\epsilon > 0$ is arbitrary, we obtain
\begin{equation}\label{eq:first_term_limit_X_0^N}
\lim_{t \searrow 0} t^\lambda \|S(t)X_0^N(\omega)\|_{\dot{\mathbb{H}}^\nu} = 0.
\end{equation}
Next, utilizing the factorization formulae \eqref{eq:factorization_formula_F}--\eqref{eq:factorization_formula_G}, we write $\mathbb{P}$-almost surely:
\begin{align}
(\mathcal{J}_F X^N)(t) &= \frac{\sin{\pi\theta}}{\pi} \int_0^t (t-\sigma)^{\theta-1} \sigma^{-\lambda} S(t-\sigma) \widetilde{F}_\theta(\sigma) \, d\sigma \quad\text{in }\;\dot{\mathbb{H}}^{-\alpha}\; \text{ for } 0 < \theta < 1, \label{eq:factorization_formula_F_X^N} \\
(\mathcal{J}_G X^N)(t) &= \frac{\sin{\pi\theta}}{\pi} \int_0^t (t-\sigma)^{\theta-1} \sigma^{-\lambda} S(t-\sigma) \widetilde{G}_\theta(\sigma) \, d\sigma \quad \text{in }\;\dot{\mathbb{H}}^{-\beta}\; \text{ for } 0 < \theta < 1/2, \label{eq:factorization_formula_G_X^N}
\end{align}
where 
\begin{equation*}
\widetilde{F}_\theta(\sigma) \coloneqq \sigma^\lambda \int_{0}^\sigma (\sigma-r)^{-\theta} S(\sigma-r) F(r, X^N(r)) \, dr, \quad \text{and} \quad \widetilde{G}_\theta(\sigma) \coloneqq \sigma^\lambda \int_{0}^\sigma (\sigma-r)^{-\theta} S(\sigma-r) G(r, X^N(r)) \, dW(r).
\end{equation*}
Following the same derivation as in \eqref{eq:ftheta_smoothing_2}--\eqref{eq:gtheta_smoothing_2}, we deduce
\begin{align*}
\|\widetilde{F}_\theta\|_{L^p(\Omega; L^p(0,T; \dot{\mathbb{H}}^{-\alpha}))} &\leq C (1 + \|X_0^N\|_{L^p(\Omega; \dot{\mathbb{H}}^{\mu})}) \left[\int_{0}^T \left(t^\lambda \int_{0}^t (t-\sigma)^{-\theta} \sigma^{-\hat{\alpha}-\lambda} \, d\sigma\right)^{p} \, dt\right]^{1/p} \\
&\leq C (1 + \|X_0^N\|_{L^p(\Omega; \dot{\mathbb{H}}^{\mu})}) B(1-\theta, 1-\hat{\alpha}-\lambda) T^{1-\theta-\hat{\alpha}+\frac{1}{p}} < \infty,
\end{align*}
for $0 < \theta < \min\{1, 1-\hat{\alpha}+1/p\}$, and similarly,
\begin{align*}
\|\widetilde{G}_\theta\|_{L^p(\Omega; L^p(0,T; \dot{\mathbb{H}}^{-\beta}))} &\leq C (1 + \|X_0^N\|_{L^p(\Omega; \dot{\mathbb{H}}^{\mu})}) \left[\int_{0}^T \left(t^{2\lambda} \int_{0}^t (t-\sigma)^{-2\theta} \sigma^{-2\hat{\beta}-2\lambda} \, d\sigma\right)^{p/2} \, dt\right]^{1/p} \\
&\leq C (1 + \|X_0^N\|_{L^p(\Omega; \dot{\mathbb{H}}^{\mu})}) \big(B(1-2\theta, 1-2\hat{\beta}-2\lambda)\big)^{1/2} T^{1/2-\theta-\hat{\beta}+\frac{1}{p}} < \infty,
\end{align*}
for $0 < \theta < \min\{1/2, 1/2-\hat{\beta}+1/p\}$. These estimates imply that for almost all $\omega \in \Omega$, 
\begin{align}
\|\widetilde{F}_\theta(\cdot, \omega)\|_{L^p(0,T; \dot{\mathbb{H}}^{-\alpha})} &< \infty \quad \text{for } 0 < \theta < \min\{1, 1-\hat{\alpha}+1/p\}, \label{eq:ftheta_smoothing_2_X_0^N_w} \\
\|\widetilde{G}_\theta(\cdot, \omega)\|_{L^p(0,T; \dot{\mathbb{H}}^{-\beta})} &< \infty \quad \text{for } 0 < \theta < \min\{1/2, 1/2-\hat{\beta}+1/p\}. \label{eq:gtheta_smoothing_2_X_0^N_w}
\end{align}
Given $0 < \lambda < 1/2$, $(\nu+\alpha)^+ < 2$, $(\nu+\beta)^+ < 1$, $0 \leq \hat{\alpha} < 1 - \frac{1}{2}(\nu+\alpha)^+$, and $0 \leq 2\hat{\beta} < 1 - (\nu+\beta)^{+}$, there exist pairs $(p_F, \theta_F) \in [2, \infty) \times (0, 1)$ and $(p_G, \theta_G) \in [2, \infty) \times (0, 1/2)$ such that
\begin{align*}
1-\lambda-\frac{1}{p_F} &> 0 \quad \text{and} \quad \frac{1}{p_F} + \frac{1}{2}(\nu+\alpha)^+ < \theta_F < \min\left\{1, 1-\hat{\alpha}+\frac{1}{p_F}\right\}, \\
1-\lambda-\frac{1}{p_G} &> 0 \quad \text{and} \quad \frac{1}{p_G} + \frac{1}{2}(\nu+\beta)^+ < \theta_G < \min\left\{\frac{1}{2}, \frac{1}{2}+\frac{1}{p_G}-\hat{\beta}\right\}.
\end{align*}
Applying Lemma~\ref{lem:semigroup_estimates}(i) and H\"older's inequality, together with estimates \eqref{eq:ftheta_smoothing_2_X_0^N_w}-\eqref{eq:gtheta_smoothing_2_X_0^N_w} in \eqref{eq:factorization_formula_F_X^N}-\eqref{eq:factorization_formula_G_X^N}, we deduce that for almost all $\omega \in \Omega$,
\begin{align*}
t^\lambda \|(\mathcal{J}_F X^N)(t,\omega)\|_{\dot{\mathbb{H}}^\nu} &\leq C t^{\theta_F-\frac{1}{p_F}-\frac{1}{2}(\nu+\alpha)^+} \|\widetilde{F}_\theta(\cdot,\omega)\|_{L^{p_F}(0,T; \dot{\mathbb{H}}^{-\alpha})} \left(B\big(1+q_F(\theta_F - 1 - \tfrac{1}{2}(\nu+\alpha)^+), 1-\lambda q_F\big)\right)^{1/q_F}, \\
t^\lambda \|(\mathcal{J}_G X^N)(t,\omega)\|_{\dot{\mathbb{H}}^\nu} &\leq C t^{\theta_G-\frac{1}{p_G}-\frac{1}{2}(\nu+\beta)^+} \|\widetilde{G}_\theta(\cdot,\omega)\|_{L^{p_G}(0,T; \dot{\mathbb{H}}^{-\beta})} \left(B\big(1+q_G(\theta_G - 1 - \tfrac{1}{2}(\nu+\beta)^+), 1-\lambda q_G\big)\right)^{1/q_G},
\end{align*}
where $1/p_F + 1/q_F = 1$ and $1/p_G + 1/q_G = 1$. These bounds, combined with \eqref{eq:first_term_limit_X_0^N}, show that the terms vanish as $t \searrow 0$, completing the proof of (ii).
\end{proof}

\begin{proof}[\textbf{Proof of Theorem~\ref{thm:general_initial_value}}]
Since $\Omega_N \subseteq \Omega_{N+1}$, invoking Lemmas~\ref{lem:local_uniqueness_gen_ini_data} and~\ref{lem:localized_prob_with_localized_initial_data} reveals that for almost all $\omega \in \bigcup_{N \in \mathbb{N}} \Omega_N$, there exists an index $N_\omega \in \mathbb{N}$ such that
\begin{equation*}
\|X^M(\cdot, \omega) - X^N(\cdot, \omega)\|_{\mathcal{C}([0,T]; \dot{\mathbb{H}}^\mu)} = 0 \quad \text{for all } M, N \geq N_\omega.
\end{equation*}
Thus, $\big\{X^N(\cdot, \omega)\big\}_{N \in \mathbb{N}}$ is a Cauchy sequence in $\mathcal{C}([0,T]; \dot{\mathbb{H}}^\mu)$, and its limit $X^{N_\omega}(\cdot, \omega) = \lim_{N \to \infty} X^N(\cdot, \omega)$ exists. Moreover, since $\bigcup_{N \in \mathbb{N}} \Omega_N = \{\omega \in \Omega : \|X_0(\omega)\|_{\dot{\mathbb{H}}^\mu} < \infty\}$ and $\mathbb{P}(\bigcup_{N \in \mathbb{N}} \Omega_N) = 1$, the process
\begin{equation*}
X \coloneqq \lim_{N \to \infty} X^N \quad \text{in } \mathcal{C}([0,T]; \dot{\mathbb{H}}^\mu) \quad \mathbb{P}\text{-a.s.}
\end{equation*}
is well-defined. Since each $X^N$ is predictable, $X$ admits a predictable version. Noting that $\mathbb{P}(\bigcup_{N \in \mathbb{N}} \Omega_N) = 1$, $\|X - X^N\|_{\mathcal{C}([0,T]; \dot{\mathbb{H}}^\mu)} = 0$ $\mathbb{P}$-a.s. on $\Omega_N$, and each $X^N$ satisfies both \eqref{eq:mild_solution_truncated_gen_ini_value} and the pathwise regularities of Lemma~\ref{lem:localized_prob_with_localized_initial_data}, it follows that $X$ is a mild solution of \eqref{eq:SEE} fulfilling regularity conditions \eqref{eq:general_initial_value_1} and \eqref{eq:general_initial_value_2}. Finally, the uniqueness of the mild solution $X$ is guaranteed by Lemma~\ref{lem:local_uniqueness_gen_ini_data}.
\end{proof}

\subsection{Proof of Theorem~\ref{thm:local_existence}}
To establish Theorem~\ref{thm:local_existence}, we first prove an auxiliary result that facilitates the construction of a local mild solution to \eqref{eq:SEE}. 

\begin{lemma}\label{lem:local_mild_solution_existence}
Suppose that the parameter constraints \eqref{eq:parameter_constraints} hold. For $j \in \{1,2\}$, consider the stochastic evolution equations
\begin{equation}\label{eq:SEE_j}
\begin{aligned}
    dX_j(t) + A X_j(t) \, dt &= F_j(t, X_j(t)) \, dt + G_j(t, X_j(t)) \, dW(t), \quad t \in (0, T], \\
    X_j(0) &= X_0 \in L^0(\Omega, \mathcal{F}_0; \dot{\mathbb{H}}^\mu),
\end{aligned}
\end{equation}
where $F_j$ and $G_j$ satisfy Assumptions~\ref{ass:measurability_drift_diff} and~\ref{ass:global_lips_nonlin}, respectively. Let $R > 0$ be such that, for all $t \in [0,T]$, the following conditions hold $\mathbb{P}$-almost surely,
\begin{equation}\label{eq:F1=F2_and_G1=G2}
F_1(t,u) = F_2(t,u) \quad \text{and} \quad G_1(t,u) = G_2(t,u),
\end{equation}
for all $u \in \dot{\mathbb{H}}^\nu$ satisfying $\|u\|_{\lambda,t} \leq R$. Then, the random times $\tau_j : \Omega \to [0,T]$ defined by
\begin{equation*}
\tau_j \coloneqq T \wedge \inf \{ t \in [0,T] : \|X_j(t)\|_{\lambda,t} > R \}\quad\Big(\inf \emptyset = \infty\Big)
\end{equation*}
are stopping times such that, for $t \in [0,T]$,
\begin{equation*}
\|X_2(t) - X_1(t)\|_{\lambda,t} = 0 \quad \mathbb{P}\text{-a.s. on } \{t \leq \tau_1\},
\end{equation*}
and $\tau_1 = \tau_2$ $\mathbb{P}$-almost surely.
\end{lemma}

\begin{proof}
By Theorem~\ref{thm:general_initial_value}, for each $j \in \{1,2\}$, problem \eqref{eq:SEE_j} admits a unique mild solution $X_j$ satisfying \eqref{eq:mild_solution_formula} and the following pathwise regularity $\mathbb{P}$-almost surely: for any $a \in (0,T)$,
\begin{equation*}
X_j \in \mathcal{C}([0,T]; \dot{\mathbb{H}}^\mu) \cap \mathcal{C}([a,T]; \dot{\mathbb{H}}^\nu) \quad \text{and} \quad \|X_j(\cdot)\|_{\lambda,\cdot} \in \mathcal{C}([0,T]) \quad \text{with} \quad \|X_j(0)\|_{\lambda,0} = \|X_0\|_{\dot{\mathbb{H}}^\mu}.
\end{equation*}
Consequently, the processes $\big(\|X_j(t)\|_{\lambda,t}\big)_{t \in [0,T]}$ are adapted and pathwise continuous. Thus, 
\begin{equation*}
\eta_n \coloneqq T \wedge \inf \{ t \in [0,T] : \|X_1(t)\|_{\lambda,t} + \|X_2(t)\|_{\lambda,t} > n \}\quad(\inf \emptyset = \infty) 
\end{equation*}
defines a sequence of stopping times $(\eta_n)_{n \in \mathbb{N}}$. Let $\xi_n(t) \coloneqq \mathbbm{1}_{\llbracket 0, \eta_n \wedge \tau_1 \rrbracket}(t)$ and set $X_j^n(t) \coloneqq \xi_n(t) X_j(t)$ for $j \in \{1,2\}$. Since the processes $X_j$ satisfy \eqref{eq:mild_solution_formula}, we have $\mathbb{P}$-almost surely:
\begin{equation}\label{eq:difference_Fj}
\begin{aligned}
    X_2^n(t) - X_1^n(t) &= \xi_n(t) \int_0^t \xi_n(\sigma) S(t-\sigma) \big(F_2(\sigma, X_2^n(\sigma)) - F_2(\sigma, X_1^n(\sigma))\big) \, d\sigma \\
    &\quad + \xi_n(t) \int_0^t \xi_n(\sigma) S(t-\sigma) \big(G_2(\sigma, X_2^n(\sigma)) - G_2(\sigma, X_1^n(\sigma))\big) \, dW(\sigma) \\
    &\quad + \xi_n(t) \int_0^t \xi_n(\sigma) S(t-\sigma) \big(F_2(\sigma, X_1^n(\sigma)) - F_1(\sigma, X_1^n(\sigma))\big) \, d\sigma \\
    &\quad + \xi_n(t) \int_0^t \xi_n(\sigma) S(t-\sigma) \big(G_2(\sigma, X_1^n(\sigma)) - G_1(\sigma, X_1^n(\sigma))\big) \, dW(\sigma).
\end{aligned}
\end{equation}
By the definition of $\tau_1$ and the pathwise continuity of $\|X_j(t)\|_{\lambda,t}$, it follows that $\|X_j(t)\|_{\lambda,t} \leq R$ on $\llbracket 0, \eta_n \wedge \tau_1 \rrbracket$. Consequently, hypothesis \eqref{eq:F1=F2_and_G1=G2} implies that the last two integrals in the equation above vanish. Taking the norm in $\dot{\mathbb{H}}^\nu$ and applying the Burkholder-Davis-Gundy (BDG) inequality (Lemma~\ref{lem:BDG}), we obtain for $t \in [0,T]$:
\begin{align*}
    \mathbb{E}\left[t^{2\lambda} \|X_2^n(t) - X_1^n(t)\|_{\dot{\mathbb{H}}^\nu}^2\right] &\leq C t^{2\lambda} \mathbb{E}\left[\left(\int_0^t \xi_n(\sigma) \|S(t-\sigma) \big(F_2(\sigma, X_2^n(\sigma)) - F_2(\sigma, X_1^n(\sigma))\big)\|_{\dot{\mathbb{H}}^\nu} \, d\sigma \right)^2\right] \\
    &\quad + C t^{2\lambda} \mathbb{E}\left[\int_0^t \xi_n(\sigma) \|S(t-\sigma) \big(G_2(\sigma, X_2^n(\sigma)) - G_2(\sigma, X_1^n(\sigma))\big)\|_{\mathrm{HS}(U_0,\dot{\mathbb{H}}^\nu)}^2 \, d\sigma\right].
\end{align*}
Invoking Assumption~\ref{ass:global_lips_nonlin} alongside Lemma~\ref{lem:semigroup_estimates} yields
\begin{align*}
    \mathbb{E}\left[t^{2\lambda} \|X_2^n(t) - X_1^n(t)\|_{\dot{\mathbb{H}}^\nu}^2\right] &\leq C t^{2\lambda} \mathbb{E}\left[\left(\int_0^t (t-\sigma)^{-\frac{1}{2}(\alpha+\nu)^+} \sigma^{-\hat{\alpha}} \|X_2^n(\sigma) - X_1^n(\sigma)\|_{\dot{\mathbb{H}}^\nu} \, d\sigma\right)^2\right] \\
    &\quad + C t^{2\lambda} \int_0^t (t-\sigma)^{-(\beta+\nu)^+} \sigma^{-2\hat{\beta}} \mathbb{E}\left[\|X_2^n(\sigma) - X_1^n(\sigma)\|_{\dot{\mathbb{H}}^\nu}^2\right] \, d\sigma.
\end{align*}
Applying H\"older's inequality followed by Fubini's theorem under the parameter constraints \eqref{eq:parameter_constraints}, we deduce for all $t \in [0,T]$:
\begin{align*}
    \mathbb{E}\left[t^{2\lambda} \|X_2^n(t) - X_1^n(t)\|_{\dot{\mathbb{H}}^\nu}^2\right] &\leq C t^{\lambda+1-\hat{\alpha}-\frac{1}{2}(\nu+\alpha)^+} \int_0^t (t-\sigma)^{-\frac{1}{2}(\alpha+\nu)^+} \sigma^{-\hat{\alpha}-\lambda} \mathbb{E}\left[\sigma^{2\lambda} \|X_2^n(\sigma) - X_1^n(\sigma)\|_{\dot{\mathbb{H}}^\nu}^2\right] \, d\sigma \\
    &\quad + C t^{2\lambda} \int_0^t (t-\sigma)^{-(\beta+\nu)^+} \sigma^{-2\hat{\beta}-2\lambda} \mathbb{E}\left[\sigma^{2\lambda} \|X_2^n(\sigma) - X_1^n(\sigma)\|_{\dot{\mathbb{H}}^\nu}^2\right] \, d\sigma.
\end{align*}
An application of Lemma~\ref{lem:gtype}, in conjunction with the assumed pathwise regularity of $X_j$, ensures $\mathbb{E}\big[t^{2\lambda} \|X_2^n(t) - X_1^n(t)\|_{\dot{\mathbb{H}}^\nu}^2\big] = 0$. Consequently, $t^\lambda \|X_2^n(t) - X_1^n(t)\|_{\dot{\mathbb{H}}^\nu} = 0$ $\mathbb{P}$-almost surely. Letting $n \to \infty$ gives
\begin{equation*}
t^\lambda \mathbbm{1}_{\llbracket 0, \tau_1 \rrbracket}(t) \|X_2(t) - X_1(t)\|_{\dot{\mathbb{H}}^\nu} = 0 \quad \mathbb{P}\text{-a.s.}
\end{equation*}
Since $\|X_2(0) - X_1(0)\|_{\dot{\mathbb{H}}^\mu} = 0$ $\mathbb{P}$-almost surely, we conclude that
\begin{equation*}
\mathbbm{1}_{\llbracket 0, \tau_1 \rrbracket}(t) \|X_2(t) - X_1(t)\|_{\lambda,t} = 0 \quad \mathbb{P}\text{-a.s.}
\end{equation*}
The pathwise continuity of the processes $\big(\|X_j(t)\|_{\lambda,t}\big)_{t \in [0,T]}$ then implies $\tau_1 \leq \tau_2$ $\mathbb{P}$-almost surely. Interchanging the roles of $\tau_1$ and $\tau_2$ yields $\tau_2 \leq \tau_1$ $\mathbb{P}$-almost surely, establishing $\tau_1 = \tau_2$.
\end{proof}

Our next result establishes the uniqueness of local mild solutions.

\begin{lemma}\label{lem:local_mild_solution_uniqueness}
Suppose that Assumption~\ref{ass:local_lip_stoch} and the parameter constraints \eqref{eq:parameter_constraints} are satisfied. Let $X_0 \in L^0(\Omega, \mathcal{F}_0; \dot{\mathbb{H}}^\mu)$. For $j \in \{1,2\}$, let $(X_j, \tau_j)$ be a local mild solution of \eqref{eq:SEE} possessing a localizing sequence of stopping times $\tau_j^N \uparrow \tau_j$ $\mathbb{P}$-almost surely such that, for each $N \in \mathbb{N}$, the stopped process $\{X_j(t \wedge \tau_j^N)\}_{t \in [0,T]}$ is an $\dot{\mathbb{H}}^\mu$-valued predictable stochastic process satisfying the following pathwise regularity $\mathbb{P}$-almost surely: for any $a \in (0,T)$,
\begin{align}
X_j(\cdot \wedge \tau_j^N) &\in \mathcal{C}([0,T]; \dot{\mathbb{H}}^\mu) \cap \mathcal{C}([a,T]; \dot{\mathbb{H}}^\nu) \quad \text{with } X_j(0) = X_0, \label{eq:local_existence_1N} \\
\lim_{t \searrow 0} \|X_j(t)\|_{\lambda,t} &= \|X_0\|_{\dot{\mathbb{H}}^\mu} \quad \text{on } \{\tau_j^N > 0\}. \label{eq:local_existence_2N}
\end{align}
Then, for any $t \in [0,T]$,
\begin{equation*}
\|X_2(t) - X_1(t)\|_{\lambda,t} = 0 \quad \mathbb{P}\text{-a.s. on } \{t \leq \tau_1^N \wedge \tau_2^N\} \quad \text{for all } N \in \mathbb{N},
\end{equation*}
and consequently,
\begin{equation*}
\|X_2(t) - X_1(t)\|_{\lambda,t} = 0 \quad \mathbb{P}\text{-a.s. on } \{t < \tau_1 \wedge \tau_2\}.
\end{equation*}
\end{lemma}

\begin{proof}
For a fixed $N \in \mathbb{N}$ and $j \in \{1,2\}$, the adaptedness and pathwise regularity \eqref{eq:local_existence_1N}--\eqref{eq:local_existence_2N} of $\{X_j(t \wedge \tau_j^N)\}_{t \in [0,T]}$ imply that the process $Y^N(t) \coloneqq \|X_1(t \wedge \tau_1^N)\|_{\lambda, t \wedge \tau_1^N} + \|X_2(t \wedge \tau_2^N)\|_{\lambda, t \wedge \tau_2^N}$ is adapted and pathwise continuous $\mathbb{P}$-almost surely. Thus, 
\begin{equation}\label{eq:R_localized_stopping_time}
\theta_n \coloneqq \tau_1^N \wedge \tau_2^N \wedge \inf \{ t \in [0,T] : Y^N(t) > n \} \quad(\inf \emptyset = \infty)
\end{equation}
defines a sequence of stopping times. Let $\zeta_n(t) \coloneqq \mathbbm{1}_{\llbracket 0, \theta_n \rrbracket}(t)$ and $X_j^n(t) \coloneqq \zeta_n(t) X_j(t)$. Because $X_j(t \wedge \tau_j^N)$ satisfies \eqref{eq:local_mild_solution_formula}, we have $\mathbb{P}$-almost surely for $t \in [0,T]$:
\begin{align*}
    X_2^n(t) - X_1^n(t) &= \zeta_n(t) \int_0^t \zeta_n(s) S(t-s) \big(F(s, X_2^n(s)) - F(s, X_1^n(s))\big) \, ds \\
    &\quad + \zeta_n(t) \int_0^t \zeta_n(s) S(t-s) \big(G(s, X_2^n(s)) - G(s, X_1^n(s))\big) \, dW(s).
\end{align*}
Invoking Assumption~\ref{ass:local_lip_stoch} and applying Lemmas~\ref{lem:BDG} and~\ref{lem:semigroup_estimates}, we deduce for $t \in [0,T]$:
\begin{align*}
    \mathbb{E}\left[t^{2\lambda} \|X_2^n(t) - X_1^n(t)\|_{\dot{\mathbb{H}}^\nu}^2\right] &\leq C(n) t^{2\lambda} \mathbb{E}\left[\left(\int_0^t (t-s)^{-\frac{1}{2}(\alpha+\nu)^+} s^{-\hat{\alpha}} \|X_2^n(s) - X_1^n(s)\|_{\lambda,s} \, ds\right)^2\right] \\
    &\quad + C(n) t^{2\lambda} \int_0^t (t-s)^{-(\beta+\nu)^+} s^{-2\hat{\beta}} \mathbb{E}\left[\|X_2^n(s) - X_1^n(s)\|_{\lambda,s}^2\right] \, ds,
\end{align*}
where the positive constants \(C(n)\) may depend on \(n\). Employing H\"older's inequality, Fubini's theorem, and the bound $\|X_2^n(s) - X_1^n(s)\|_{\lambda,s} \leq (C+T^\lambda) \|X_2^n(s) - X_1^n(s)\|_{\dot{\mathbb{H}}^\nu}$, the parameter constraints \eqref{eq:parameter_constraints} yield
\begin{equation*}
\varphi_n(t) \leq \widetilde{C}(n) t^{\lambda+1-\hat{\alpha}-\frac{1}{2}(\nu+\alpha)^+} \int_0^t (t-s)^{-\frac{1}{2}(\alpha+\nu)^+} s^{-\hat{\alpha}-\lambda} \varphi_n(s) \, ds + \widetilde{C}(n) t^{2\lambda} \int_0^t (t-s)^{-(\beta+\nu)^+} s^{-2\hat{\beta}-2\lambda} \varphi_n(s) \, ds,
\end{equation*}
where $\varphi_n(t) \coloneqq \mathbb{E}\big[t^{2\lambda} \|X_2^n(t) - X_1^n(t)\|_{\dot{\mathbb{H}}^\nu}^2\big]$, and the positive constant \(\widetilde{C}(n)\) may depend on \(n\). Lemma~\ref{lem:gtype} then guarantees $\varphi_n(t) = 0$ for all $t \in [0,T]$, which immediately implies $t^\lambda \|X_2^n(t) - X_1^n(t)\|_{\dot{\mathbb{H}}^\nu} = 0$ $\mathbb{P}$-almost surely. Letting $n \to \infty$, we obtain
\begin{equation*}
t^\lambda \mathbbm{1}_{\llbracket 0, \tau_1^N \wedge \tau_2^N \rrbracket}(t) \|X_2(t) - X_1(t)\|_{\dot{\mathbb{H}}^\nu} = 0 \quad \mathbb{P}\text{-a.s.}
\end{equation*}
Since $N \in \mathbb{N}$ is arbitrary and $\|X_2(0) - X_1(0)\|_{\dot{\mathbb{H}}^\mu} = 0$, we conclude
\begin{equation*}
\mathbbm{1}_{\llbracket 0, \tau_1^N \wedge \tau_2^N \rrbracket}(t) \|X_2(t) - X_1(t)\|_{\lambda,t} = 0 \quad \mathbb{P}\text{-a.s. for all } N \in \mathbb{N}.
\end{equation*}
Taking the limit as $N \to \infty$ establishes the final claim.
\end{proof}

\begin{proof}[\textbf{Proof of Theorem~\ref{thm:local_existence}}]
We proceed via a truncation argument. For each $N \in \mathbb{N}$, we define the truncated drift $F_N : [0,T] \times \Omega \times \dot{\mathbb{H}}^\mu \to \dot{\mathbb{H}}^{-\alpha}$ and diffusion $G_N : [0,T] \times \Omega \times \dot{\mathbb{H}}^\mu \to \mathrm{HS}(U_0, \dot{\mathbb{H}}^{-\beta})$ via the time-weighted norm $\|\cdot\|_{\lambda,t}$ given in \eqref{eq:time_weighted_norm}:
\begin{equation*}
F_N(t, u) \coloneqq 
\begin{cases} 
    F\left(t, \frac{N u}{\|u\|_{\lambda,t}}\right) & \text{if } u\in\dot{\mathbb{H}}^\nu \text{ and } \|u\|_{\lambda,t} > N, \\
    F(t, u)  & \text{otherwise},
\end{cases} \quad \mathbb{P}\text{-a.s.}
\end{equation*}
and
\begin{equation*}
G_N(t, u) \coloneqq 
\begin{cases} 
    G\left(t, \frac{N u}{\|u\|_{\lambda,t}}\right) & \text{if } u\in\dot{\mathbb{H}}^\nu \text{ and } \|u\|_{\lambda,t} > N, \\
    G(t, u)  & \text{otherwise},
\end{cases} \quad \mathbb{P}\text{-a.s.}
\end{equation*} 
Under Assumption~\ref{ass:local_lip_stoch}, the truncated coefficients $F_N$ and $G_N$ satisfy Assumption~\ref{ass:global_lips_nonlin} for each fixed $N$. Thus, by Theorem~\ref{thm:general_initial_value}, the truncated stochastic evolution equation 
\begin{equation}\label{eq:truncated_sivp}
\begin{cases}
    dY_N(t) + A Y_N(t) \, dt = F_N(t, Y_N(t)) \, dt + G_N(t, Y_N(t)) \, dW(t), & t \in (0, T], \\
    Y_N(0) = X_0,
\end{cases}
\end{equation}
admits a unique global mild solution $Y_N$. That is, for all \(t \in [0,T]\), \(Y_N\) satisfies the following equation $\mathbb{P}$-almost surely in $\dot{\mathbb{H}}^\mu$,
\begin{equation*}
Y_N(t) = S(t)X_0 + (\mathcal{J}_{F_N}Y_N)(t) + (\mathcal{J}_{G_N}Y_N)(t),
\end{equation*}
and exhibits the following pathwise regularity $\mathbb{P}$-almost surely: for any $a \in (0,T)$ and $0 < \delta < \min\big\{1 - \frac{1}{2}(\nu+\alpha)^+, \frac{1}{2} - \frac{1}{2}(\nu+\beta)^+\big\}$,
\begin{align}
&Y_N \in \mathcal{C}([0,T]; \dot{\mathbb{H}}^\mu) \cap \mathring{\mathcal{C}}^{0,\delta}([a,T]; \dot{\mathbb{H}}^\nu) \quad \text{with } Y_N(0) = X_0 \text{ in } \dot{\mathbb{H}}^\mu , \label{eq:truncated_general_initial_value_1} \\
&\lim_{t \searrow 0} \|Y_N(t)\|_{\lambda,t} = \|X_0\|_{\dot{\mathbb{H}}^\mu} \quad\text{ for } \lambda \coloneqq (\nu-\mu)/2 > 0. \label{eq:truncated_general_initial_value_2}
\end{align}
Define a sequence of random times $(\tau_N)_{N \in \mathbb{N}}$ by
\begin{equation}\label{eq:stopping_time_def}
\tau_N \coloneqq T \wedge \inf \{ t \in [0,T] : \|Y_N(t)\|_{\lambda,t} > N \} \quad (\inf \emptyset = \infty).
\end{equation}
Since $\|Y_N(t)\|_{\lambda,t}$ is adapted and pathwise continuous, $\tau_N$ is a well-defined stopping time. If we define the intermediate stopping time 
\begin{equation*}
\widetilde{\tau}_N \coloneqq T \wedge \inf \{ t \in [0,T] : \|Y_{N+1}(t)\|_{\lambda,t} > N \},
\end{equation*}
it follows immediately that $\widetilde{\tau}_N \leq \tau_{N+1}$ $\mathbb{P}$-almost surely. Lemma~\ref{lem:local_mild_solution_existence} further ensures $\widetilde{\tau}_N = \tau_N$ $\mathbb{P}$-almost surely, yielding the monotonicity property:
\begin{equation}\label{eq:stoping_time_monotonicity}
\tau_N \leq \tau_{N+1} \quad \mathbb{P}\text{-a.s.}
\end{equation}
This monotonicity established in \eqref{eq:stoping_time_monotonicity}, combined with Lemma~\ref{lem:local_mild_solution_uniqueness}, establishes the local uniqueness statement:
\begin{equation}\label{eq:local_uniqueness_result_N}
\|Y_{N+1}(t) - Y_N(t)\|_{\lambda,t} = 0 \quad \mathbb{P}\text{-a.s. on } \{t \leq \tau_N\}.
\end{equation}
Define the stopping time $\tau_{\max} \coloneqq \lim_{N\to\infty}\tau_N$, which is well-defined due to the monotonicity in \eqref{eq:stoping_time_monotonicity}. Since $X_0 \in L^0(\Omega, \mathcal{F}_0; \dot{\mathbb{H}}^\mu)$, the random variable $\|X_0\|_{\dot{\mathbb{H}}^\mu}$ is finite $\mathbb{P}$-almost surely. By the limit condition \eqref{eq:truncated_general_initial_value_2}, for $\mathbb{P}$-almost every $\omega$, we have $\lim_{t \searrow 0} \|Y_N(t, \omega)\|_{\lambda,t} = \|X_0(\omega)\|_{\dot{\mathbb{H}}^\mu}$. Therefore, for any integer $N > \|X_0(\omega)\|_{\dot{\mathbb{H}}^\mu}$, there exists a deterministic time $t_0(\omega) > 0$ such that $\|Y_N(t, \omega)\|_{\lambda,t} < N$ for all $t \in (0, t_0(\omega))$. Consequently, the infimum in \eqref{eq:stopping_time_def} must be bounded below by $t_0(\omega)$, yielding $\tau_N(\omega) \ge t_0(\omega) \wedge T > 0$. Since this holds for all sufficiently large $N$, we conclude that $\tau_{\max} > 0$ $\mathbb{P}$-almost surely.

Next, we construct the limit process. Fix $t \in [0,T]$. By the local uniqueness statement \eqref{eq:local_uniqueness_result_N}, for almost every $\omega \in \bigcup_{N=1}^\infty \{t \leq \tau_N\}$, there exists an integer $N_0(\omega)$ such that $t \leq \tau_{N_0}(\omega)$. Consequently, for all $M > N \geq N_0(\omega)$, we have
\begin{equation*}
    \|Y_M(t,\omega) - Y_N(t,\omega)\|_{\lambda,t} = 0.
\end{equation*}
This trivially implies that $(Y_N(t,\omega))_{N\in\mathbb{N}}$ is a Cauchy sequence with respect to the time-weighted norm $\|\cdot\|_{\lambda,t}$. Since the underlying space equipped with this norm is complete (coinciding with $\dot{\mathbb{H}}^\mu$ at $t=0$, and equivalent to $\dot{\mathbb{H}}^\nu$ for $t > 0$), the sequence is convergent. Therefore, for all $t \in [0,T]$, we may define
\begin{equation}\label{eq:maximal_local_solution}
X(t) \coloneqq \begin{cases}
    \lim_{N \to \infty} Y_N(t) \quad & \text{in the } \|\cdot\|_{\lambda,t} \text{ norm, } \mathbb{P}\text{-a.s. on } \bigcup_{N=1}^\infty \{t \leq \tau_N\}, \\
    0 & \text{otherwise.}
\end{cases} 
\end{equation}
Hence, $\{X(t)\}_{t\in[0,T]}$ is a well-defined process. By construction, for each $n \in \mathbb{N}$,
\begin{equation}\label{eq:maximal_local_solution_X}
X(t) = Y_n(t) \quad \mathbb{P}\text{-a.s. on } \{t \leq \tau_n\},
\end{equation}
\begin{equation}\label{eq:maximal_local_solution_FG}
F_n(t, Y_n(t)) = F(t, X(t)), \quad G_n(t, Y_n(t)) = G(t, X(t)) \quad \mathbb{P}\text{-a.s. on } \{t \leq \tau_n\}.
\end{equation}
Since each $Y_n$ is a mild solution of \eqref{eq:truncated_sivp} satisfying \eqref{eq:truncated_general_initial_value_1}--\eqref{eq:truncated_general_initial_value_2}, relations \eqref{eq:maximal_local_solution_X}--\eqref{eq:maximal_local_solution_FG} confirm that $(X, \tau_{\max})$ is a local mild solution to \eqref{eq:SEE} fulfilling the necessary pathwise regularity conditions \eqref{eq:local_existence_1}--\eqref{eq:local_existence_2}.

To establish the blow-up alternative, observe that for almost every $\omega \in \{\tau_{\max} < T\}$, $\tau_N(\omega) \uparrow \tau_{\max}(\omega)$. Thus, the pathwise continuity of the truncated solutions requires $\|X(\tau_N(\omega), \omega)\|_{\lambda, \tau_N(\omega)} = \|Y_N(\tau_N(\omega), \omega)\|_{\lambda, \tau_N(\omega)} = N$. Consequently,
\begin{equation}\label{eq:local_mild_sol_blow-up}
\limsup_{t \uparrow \tau_{\max}} \|X(t)\|_{\lambda,t} = +\infty \quad \mathbb{P}\text{-a.s. on } \{ \tau_{\max} < T \}.
\end{equation}

Finally, we prove the maximality and uniqueness of $(X, \tau_{\max})$. Suppose $(Y, \tau)$ is another local mild solution satisfying the pathwise regularity results \eqref{eq:local_existence_1}--\eqref{eq:local_existence_2}. We must show that $\tau \leq \tau_{\max}$ $\mathbb{P}$-almost surely. Suppose, for the sake of contradiction, that $\mathbb{P}(\tau > \tau_{\max}) > 0$. Let $\Omega_0 \coloneqq \{\tau > \tau_{\max}\}$. Because both stopping times are bounded by $T$, it trivially follows that $\Omega_0 \subseteq \{\tau_{\max} < T\}$. For almost every \(\omega\in\Omega_0\), the pathwise regularity results \eqref{eq:local_existence_1}--\eqref{eq:local_existence_2} imply that the path $t \mapsto \|Y(t, \omega)\|_{\lambda,t}$ is continuous on the compact interval $[0, \tau_{\max}(\omega)]$, thereby guaranteeing the existence of a finite constant $C_\omega < \infty$ such that
\begin{equation}\label{eq:maximality_bddness}\sup_{t \in [0, \tau_{\max}(\omega)]} \|Y(t, \omega)\|_{\lambda,t} \leq C_\omega.
\end{equation}
By Lemma~\ref{lem:local_mild_solution_uniqueness}, we have \(
\|X(t) - Y(t)\|_{\lambda,t} = 0 \quad \mathbb{P}\text{-a.s. on } \{t < \tau_{\max} \wedge \tau\}\),
which implies that for almost every $\omega \in \Omega_0$,  $\|X(t, \omega)\|_{\lambda,t} = \|Y(t, \omega)\|_{\lambda,t}$ for all $t < \tau_{\max}(\omega)$. Combining this with \eqref{eq:maximality_bddness} yields that for almost every $\omega \in \Omega_0$,
\begin{equation*}
\limsup_{t \uparrow \tau_{\max}(\omega)} \|X(t, \omega)\|_{\lambda,t} = \limsup_{t \uparrow \tau_{\max}(\omega)} \|Y(t, \omega)\|_{\lambda,t} \le \sup_{t\in[0,\tau_{\max}(\omega)]} \|Y(t,\omega)\|_{\lambda,t} \leq C_\omega < \infty.
\end{equation*}
However, because $\Omega_0 \subseteq \{\tau_{\max} < T\}$, this directly contradicts the blow-up alternative \eqref{eq:local_mild_sol_blow-up}. Because assuming $\mathbb{P}(\Omega_0) > 0$ leads to a contradiction, we conclude $\mathbb{P}(\tau > \tau_{\max}) = 0$, or equivalently, $\tau \leq \tau_{\max}$ $\mathbb{P}$-almost surely. Uniqueness up to $\tau \wedge \tau_{\max}$ is directly guaranteed by Lemma~\ref{lem:local_mild_solution_uniqueness}.
\end{proof}

\section{Applications to SPDEs}\label{sec:examples}
In this section, we verify that the abstract assumptions introduced above, in particular the local Lipschitz conditions in Assumption~\ref{ass:local_lip_stoch}, are satisfied by several classical nonlinear SPDEs. We first recall the following standard embedding results, which will be used frequently throughout this section.

\begin{lemma}[Fractional Sobolev and duality embeddings]
\label{lem:sobolev_embedding}
Let $\mathcal{O} \subseteq \mathbb{R}^d$ be either the torus $\mathbb{T}^d$ or a bounded, open, connected domain possessing the $C^2$-extension property. Let the operator $A \colon \mathcal{D}(A) \subset \mathbb{H} \to \mathbb{H}$ (where $\mathbb{H} = L^2(\mathcal{O}; \mathbb{R}^N)$ with $N \in \mathbb{N}$) satisfy Assumption~\ref{ass:operator_A}, and assume that $\mathcal{D}(A) \subset H^2(\mathcal{O}; \mathbb{R}^N)$. Furthermore, assume that $\mathcal{D}(A)$ is dense in $L^p(\mathcal{O};\mathbb{R}^N)$ for the values of $p$ considered below. Then, for any $s \in [0, 2]$, the following continuous embeddings hold:
\begin{enumerate}[label={\upshape(\roman*)}]
\item $\dot{\mathbb{H}}^{s} \hookrightarrow L^{p}(\mathcal{O}; \mathbb{R}^N)$ densely, for any $p \in [2, \infty)$ satisfying $\frac{1}{p} \geq \frac{1}{2} - \frac{s}{d}$.
\item $L^{q}(\mathcal{O}; \mathbb{R}^N) \hookrightarrow \dot{\mathbb{H}}^{-s}$, for any $q \in (1, 2]$ satisfying $\frac{1}{q} \leq \frac{1}{2} + \frac{s}{d}$.
\end{enumerate}
\end{lemma}

For simplicity of exposition, in this section, we restrict our attention to open, connected, bounded domains $\mathcal{O} \subset \mathbb{R}^d, \,d=1,2,3,$ possessing the $C^2$-extension property. The arguments and results presented below may be carried out, with minor modifications, for other domains \(\mathcal{O}\) considered in Lemma~\ref{lem:sobolev_embedding}. Also, we set $U=\mathbb{H}=L^2(\mathcal{O};\mathbb{R}^N)$, for \(N\in\mathbb{N}\), and assume that $W$ is a $Q$-Wiener process defined on $U=\mathbb{H}$. Since \(Q:U\to U\) is self-adjoint and positive semidefinite, there exists a sequence of non-negative real numbers \(\{q_n\}_{n\in\mathbb{N}}\) and an orthonormal basis \(\{h_n\}_{n\in\mathbb{N}}\) in \(U=\mathbb{H}\) such that \(Q h_n = q_n h_n\). We define the operator \(A:\mathcal{D}(A)\subset \mathbb{H}\rightarrow\mathbb{H}\)  in all the examples below such that it satisfies Assumption~\ref{ass:operator_A} and $\mathcal{D}(A) \subset H^2(\mathcal{O}; \mathbb{R}^N)$, and therefore there exists a non-decreasing sequence of positive real numbers $\{\eta_n\}_{n\in\mathbb{N}}$ and an orthonormal basis of eigenvectors $\{e_n\}_{n\in\mathbb{N}}$ in $\mathbb{H}$ such that $A e_n = \eta_n e_n$ for all $n \in \mathbb{N}$, and 
    \[
        0 < \eta_1 \le \eta_2 \le \cdots \le \eta_n \le \cdots, \quad \lim_{n \to \infty} \eta_n = \infty.
    \]
For notational simplicity, we suppress the explicit dependence on the sample space variable $\omega \in \Omega$ when writing the random field $X(t,x)$ and the local coefficients, with the standard understanding that these equations hold $\mathbb{P}$-almost surely.

We now start illustrating our theoretical findings by first applying them to the following linear stochastic convection-diffusion equation with rough initial data. 

\subsection{The stochastic convection-diffusion equation}\label{subsec:convection_diffusion}

We consider the following stochastic convection-diffusion equation with additive noise:
\begin{equation}\label{eq:convection_diffusion}
\left\{
\begin{aligned}
    dX(t,\boldsymbol{x}) &= \left[\nabla \cdot \big(a(\boldsymbol{x})\nabla X(t,\boldsymbol{x})\big) - \boldsymbol{b}(\boldsymbol{x}) \cdot \nabla X(t,\boldsymbol{x}) - c(\boldsymbol{x})X(t,\boldsymbol{x}) \right] dt \\
                         &\quad + g(t,\boldsymbol{x}) \, dW(t,\boldsymbol{x}), && (t,\boldsymbol{x}) \in (0,T] \times \mathcal{O}, \\
    X(t,\boldsymbol{x}) &= 0, && (t,\boldsymbol{x}) \in (0,T] \times \partial\mathcal{O}, \\
    X(0,\boldsymbol{x}) &= X_0(\boldsymbol{x}), && \boldsymbol{x} \in \mathcal{O},
\end{aligned}
\right.
\end{equation}
where $a \in C^1(\bar{\mathcal{O}}; \mathbb{R}^{d \times d})$ is symmetric ($a_{ij} = a_{ji}$) and satisfies the uniform ellipticity condition, i.e., there exists a constant $a_0 > 0$ such that 
\begin{equation}\label{uniformellipticity}
    \sum_{i,j=1}^d a_{ij}(\boldsymbol{x}) \xi_i \xi_j \ge a_0 |\boldsymbol{\xi}|^2 \quad \forall \boldsymbol{x} \in \bar{\mathcal{O}}, \; \forall \boldsymbol{\xi} \in \mathbb{R}^d,
\end{equation}
 $\boldsymbol{b} \in C^1(\bar{\mathcal{O}}; \mathbb{R}^d)$, $c \in L^\infty(\mathcal{O}; \mathbb{R})$, \(X_0\in L^p(\Omega, \mathcal{F}_0; \dot{\mathbb{H}}^{\mu}),\;p\geq 2.\) The diffusion coefficient $g \colon [0,T]\times\Omega \times \mathcal{O} \to \mathbb{R}$ is measurable from the measurable space $(\Omega_T \times \mathcal{O}, \mathcal{P}_T \otimes \mathcal{B}(\mathcal{O}))$ to $(\mathbb{R}, \mathcal{B}(\mathbb{R}))$, \(\Omega_T\coloneqq [0,T]\times\Omega,\) and satisfies 
 \[
 \sup_{(t,\omega,\boldsymbol{x})\in [0,T]\times\Omega \times \mathcal{O}}|g(t,\omega,\boldsymbol{x})| \le C_g 
 \]
 for some constant $C_g > 0$.

To recast the SPDE \eqref{eq:convection_diffusion} as an abstract stochastic evolution equation of the form \eqref{eq:SEE}, we define the linear operator $A u = -\nabla \cdot (a(\boldsymbol{x})\nabla u) + u$ with domain $\mathcal{D}(A) = H^2(\mathcal{O}) \cap H_0^1(\mathcal{O})$. The operator $A$ satisfies Assumption~\ref{ass:operator_A} and generates a family of fractional order spaces $\dot{\mathbb{H}}^\gamma \coloneqq \mathcal{D}(A^{\gamma/2})$ for $\gamma \in \mathbb{R}$. We then define the drift mapping $F \colon [0,T] \times \Omega \times \dot{\mathbb{H}}^\mu \to \dot{\mathbb{H}}^{-1}$ such that, for $(t,\omega,u) \in [0,T] \times \Omega \times \dot{\mathbb{H}}^\mu$,
\begin{align} \label{eq:drift_convection}
    F(t,\omega,u) \coloneqq \begin{cases}
        -\boldsymbol{b}\cdot\nabla u - cu + u & \text{if } u \in \mathbb{H},\\
        0 & \text{otherwise.}
    \end{cases}
\end{align}
Clearly, the mapping $F$ is well-defined and satisfies the measurability condition of Assumption~\ref{ass:measurability_drift_diff}, as well as the boundedness and global Lipschitz conditions of Assumption~\ref{ass:global_lips_nonlin} with parameters $\nu=0$, $\alpha=1$, and $\hat{\alpha}=0$. 

Next, setting $U_0 = Q^{1/2}(\mathbb{H})$, we define the diffusion mapping $G \colon [0,T] \times \Omega \times \dot{\mathbb{H}}^\mu \to \mathrm{HS}(U_0, \dot{\mathbb{H}}^{-\beta})$ such that, for $(t,\omega,u) \in [0,T] \times \Omega \times \dot{\mathbb{H}}^\mu$,
\begin{equation}\label{eq:operator_G_additive}
G(t,\omega,u) \coloneqq \mathcal{M}_{g(t,\omega)},
\end{equation}
where $\mathcal{M}_h \colon U_0 \subseteq \mathbb{H} \to \dot{\mathbb{H}}^{-\beta}$ denotes the multiplication operator defined by $\mathcal{M}_h(\phi) = h\phi$. Because $g$ is measurable and uniformly bounded, $G$ is well-defined and satisfies the measurability condition in Assumption~\ref{ass:measurability_drift_diff}. For trace-class noise (i.e., a $Q$-Wiener process with $\mathrm{Tr}(Q) < \infty$), we obtain the following bound for any $u \in \mathbb{H}$:
\begin{align}\label{eq:G_well-definedness_con_diff_colored_noise}
    \nonumber \|G(t,\omega, u)\|_{\mathrm{HS}(U_0, \dot{\mathbb{H}}^{-\beta})}^2 &= \sum_{n=1}^\infty \| G(t,\omega, u)Q^{1/2}h_n\|_{\dot{\mathbb{H}}^{-\beta}}^2 = \sum_{n=1}^\infty\|g(t,\omega)Q^{1/2}h_n\|_{\dot{\mathbb{H}}^{-\beta}}^2\\
    &\le \sum_{n=1}^\infty\|g(t,\omega)Q^{1/2}h_n\|_{\mathbb{H}}^2 \le \|g(t,\omega)\|_{L^\infty(\mathcal{O})}^2 \mathrm{Tr}(Q) < \infty.
\end{align}
However, for space-time white noise ($Q = \mathrm{Id}_{\mathbb{H}}$), we have:
\begin{align}\label{eq:G_well-definedness_con_diff_space_time_white_noise}
    \nonumber \|G(t,\omega, u)\|_{\mathrm{HS}(U_0, \dot{\mathbb{H}}^{-\beta})}^2 &= \sum_{n=1}^\infty \| G(t,\omega, u)e_n\|_{\dot{\mathbb{H}}^{-\beta}}^2 = \sum_{n=1}^\infty\|g(t,\omega)e_n\|_{\dot{\mathbb{H}}^{-\beta}}^2\\
    \nonumber &= \sum_{n=1}^\infty \sum_{m=1}^\infty \eta_m^{-\beta} \langle g(t,\omega,\cdot) e_n, e_m \rangle_{\dot{\mathbb{H}}^{-\beta},\dot{\mathbb{H}}^\beta}^2 \\
    \nonumber &= \sum_{m=1}^\infty \eta_m^{-\beta} \sum_{n=1}^\infty \langle g(t,\omega,\cdot)  e_n, e_m \rangle_{\dot{\mathbb{H}}^{-\beta},\dot{\mathbb{H}}^\beta}^2 \\
    \nonumber &= \sum_{m=1}^\infty \eta_m^{-\beta} \sum_{n=1}^\infty ( g(t,\omega,\cdot) e_m, e_n )_{\mathbb{H}}^2 \\
    \nonumber &= \sum_{m=1}^\infty \eta_m^{-\beta} \| g(t,\omega,\cdot)  e_m\|_{\mathbb{H}}^2 \\
    \nonumber &\le \| g(t,\omega,\cdot) \|_{L^\infty(\mathcal{O})}^2 \sum_{m=1}^\infty \eta_m^{-\beta}  \\
    &\le C \| g(t,\omega,\cdot) \|_{L^\infty(\mathcal{O})}^2 \sum_{m=1}^\infty m^{-2\beta/d} < \infty, \quad \text{provided } \beta > \frac{d}{2},
\end{align}
where the final inequality in \eqref{eq:G_well-definedness_con_diff_space_time_white_noise} relies on Weyl's law, which asserts that $C_1 m^{2/d} \le \eta_m \le C_2 m^{2/d}$ for positive constants $C_1$ and $C_2$. Thus, the mapping $G$ satisfies the boundedness and global Lipschitz conditions of Assumption~\ref{ass:global_lips_nonlin} with parameters $\nu=0$, $\hat{\beta}=0$, and $\beta \ge 0$ whenever $\mathrm{Tr}(Q) < \infty$. In the case $Q = \mathrm{Id}_{\mathbb{H}}$, the same assumptions hold provided $\beta > d/2$. 

Applying Theorem~\ref{thm:main_global_existence}, we establish the existence of a unique global mild solution to the stochastic convection-diffusion equation \eqref{eq:convection_diffusion}. This solution possesses all the (pathwise) regularity properties derived in Theorem~\ref{thm:main_global_existence} and Theorem~\ref{thm:pathwise_regularity} under the following parameter constraints:
$$
\begin{aligned}
    &\text{(i)}~ {-1} < \mu \le 0, \quad &&\text{(ii)}~ \nu = 0, \quad &&\text{(iii)}~ \alpha = 1, \\
    &\text{(iv)}~ \beta < 1, \quad &&\text{(v)}~ \beta > \frac{d}{2} \text{ when } Q = \mathrm{Id}_{\mathbb{H}} \text{ and } \beta \ge 0 \text{ when } \mathrm{Tr}(Q) < \infty.
\end{aligned}
$$
Crucially, because $G$ is independent of $u$, equation \eqref{eq:convection_diffusion} admits a unique global mild solution under the following relaxed parameter constraints (see Remark~\ref{rem:global_case_comparison}(i)):
$$
\begin{aligned}
    &\text{(i)}~ {-2} < \mu \le 0, \quad &&\text{(ii)}~ \nu = 0, \quad &&\text{(iii)}~ \alpha = 1, \\
    &\text{(iv)}~ \beta < 1, \quad &&\text{(v)}~ \beta > \frac{d}{2} \text{ when } Q = \mathrm{Id}_{\mathbb{H}} \text{ and } \beta \ge 0 \text{ when } \mathrm{Tr}(Q) < \infty.
\end{aligned}
$$
The extended range $-2 < \mu \le 0$ demonstrates that problem \eqref{eq:convection_diffusion} is well-posed even for distribution-valued initial data. In fact, in one spatial dimension, the problem is well-posed even when the initial condition is the derivative of the Dirac measure, $\delta'_{\boldsymbol{x}_0}$ for $\boldsymbol{x}_0 \in \mathcal{O}$.

\subsection*{SPDEs with multiplicative noise}
Henceforth, we consider semilinear SPDE models wherein the nonlinear diffusion coefficient $g \colon [0,T] \times \Omega \times \mathcal{O} \times \mathbb{R}^N \to \mathbb{R}^N$ is measurable from the measurable space $(\Omega_T \times \mathcal{O} \times \mathbb{R}^N, \mathcal{P}_T \otimes \mathcal{B}(\mathcal{O}) \otimes \mathcal{B}(\mathbb{R}^N))$ to $(\mathbb{R}^N, \mathcal{B}(\mathbb{R}^N))$, and satisfies, for some $\sigma \ge 0$ and $\theta \ge 0$, the boundedness condition
\begin{align}
    \label{eq:g_boundedness} \sup_{(t,\omega) \in \Omega_T} t^{\sigma} \|g(t,\omega,\cdot,0)\|_{\mathbb{H}} \le C_2,
\end{align}
along with the local Lipschitz condition
\begin{align}
    \label{eq:g_locally_Lipschitz} \sup_{(t,\omega,\boldsymbol{x}) \in \Omega_T \times \mathcal{O}} t^{\sigma} |g(t,\omega,\boldsymbol{x},y_2) - g(t,\omega,\boldsymbol{x},y_1)| \le C_1(1+|y_1|+|y_2|)^\theta |y_1-y_2| \quad \forall y_1, y_2 \in \mathbb{R}^N,
\end{align}
where $\mathbb{H} = L^2(\mathcal{O}; \mathbb{R}^N)$, $|\cdot|$ denotes the Euclidean norm on $\mathbb{R}^N$, and $C_1, C_2 > 0$ are constants.

The following lemma plays a crucial role in casting these SPDE models as abstract stochastic evolution equations of the form \eqref{eq:SEE}.

\begin{lemma}\label{lem:cond_on_diffusion}
Let the assumptions of Lemma~\ref{lem:sobolev_embedding} be satisfied with the linear operator $A  = -\Delta  + I$ defined on the domain $\mathcal{D}(A) = H^2(\mathcal{O}; \mathbb{R}^N) \cap H_0^1(\mathcal{O}; \mathbb{R}^N)$. Assume that the nonlinear diffusion coefficient $g \colon [0,T] \times \Omega \times \mathcal{O} \times \mathbb{R}^N \to \mathbb{R}^N$ is measurable and satisfies conditions \eqref{eq:g_boundedness} and \eqref{eq:g_locally_Lipschitz}. Let $\mu \le \nu$, $0\leq \beta <1-\nu$, and $\nu \ge \frac{\theta d}{2(\theta+1)}$ for $\theta \ge 0$. Define the mapping $G \colon [0,T] \times \Omega \times \dot{\mathbb{H}}^\mu \to \mathrm{HS}(U_0, \dot{\mathbb{H}}^{-\beta})$ with $U_0 = Q^{1/2}(\mathbb{H})$ such that, for $(t,\omega,u) \in [0,T] \times \Omega \times \dot{\mathbb{H}}^\mu$,
\begin{equation}\label{eq:nemytskii_G}
G(t,\omega,u) = 
\begin{cases}
    \mathcal{M}_{\mathcal{N}_g(t,\omega,u)}  & \text{if } u \in \dot{\mathbb{H}}^\nu, \\
    0 & \text{otherwise,}
\end{cases}
\end{equation}
where $\mathcal{M}_h \colon U_0 \subseteq \mathbb{H} \to \dot{\mathbb{H}}^{-\beta}$ is the multiplication operator defined by $\mathcal{M}_h(\phi) = h\phi$, and the mapping $\mathcal{N}_g \colon [0,T] \times \Omega \times \dot{\mathbb{H}}^\nu \to \mathbb{H}$ is the Nemytskii operator given by 
\[
(\mathcal{N}_g(t,\omega,u))(\boldsymbol{x}) = \mathbbm{1}_{(0,T]}(t)g(t,\omega,\boldsymbol{x},u(\boldsymbol{x})), \quad \boldsymbol{x} \in \mathcal{O}.
\]
If $d=1$ and $\beta > 1/2$, or if $Q$ is a weighted-trace-class operator such that its eigensystem $\{q_n, h_n\}_{n \in \mathbb{N}}$ satisfies \textnormal{(see \cite[Eq. (25)]{MR3912731} and \cite[Section~1.2]{agresti2026optimallocaltheoryreactiondiffusion}):}
\begin{align}
    \label{eq:smooth_noise} \sum_{n=1}^\infty q_n\|h_n\|_{L^\infty(\mathcal{O};\mathbb{R}^N)}^2 < \infty,
\end{align}
then the mapping $G$ is well-defined and satisfies the measurability condition in Assumption~\ref{ass:measurability_drift_diff}. Furthermore, $G$ satisfies:
\begin{enumerate}[label={\upshape(\roman*)}]
    \item the boundedness and global Lipschitz condition of Assumption~\ref{ass:global_lips_nonlin} with parameters $\theta=0$ and $\hat{\beta}=\sigma$, and
    \item the boundedness and local Lipschitz condition of Assumption~\ref{ass:local_lip_stoch} with parameters $\theta \ge 0$ and $\hat{\beta} = \sigma + (\theta+1)\lambda$.
\end{enumerate}
\end{lemma}
\begin{proof}
Since $\nu \ge \frac{\theta d}{2(\theta +1)}$, Sobolev embedding (Lemma~\ref{lem:sobolev_embedding}) yields $\dot{\mathbb{H}}^{\nu} \hookrightarrow L^{2(\theta+1)}(\mathcal{O}; \mathbb{R}^N)$. Therefore, utilizing conditions \eqref{eq:g_boundedness} and \eqref{eq:g_locally_Lipschitz}, we deduce that $\mathbbm{1}_{(0,T]}(t)g(t,\omega,\cdot,u(\cdot)) \in \mathbb{H}$ for all $(t,\omega,u) \in [0,T] \times \Omega \times \dot{\mathbb{H}}^\nu$. This ensures the Nemytskii operator $\mathcal{N}_g \colon [0,T] \times \Omega \times \dot{\mathbb{H}}^\nu \to \mathbb{H}$ is well-defined.

Let $\{h_n\}_{n \in \mathbb{N}}$ be an orthonormal basis of $\mathbb{H}$ consisting of the eigenvectors of the covariance operator $Q$. The sequence $\{Q^{1/2}h_n\}_{n \in \mathbb{N}}$ thus forms an orthonormal basis for $U_0 = Q^{1/2}(\mathbb{H})$. Let $\{e_m\}_{m \in \mathbb{N}}$ denote the orthonormal basis of $\mathbb{H}$ consisting of the eigenvectors of the linear operator $A$, with corresponding eigenvalues $\eta_m$. We proceed by considering two cases:

\noindent\textbf{Case 1:} $d=1$ and $\beta > 1/2$.

\noindent 
In this case, since $\dot{\mathbb{H}}^\beta \hookrightarrow L^{\infty}(\mathcal{O}; \mathbb{R}^N)$ and $h\phi \in L^{1}(\mathcal{O}; \mathbb{R}^N)$ for all $h \in \mathbb{H}$ and $\phi \in U_0 = Q^{1/2}(\mathbb{H})$, the multiplication operator $\mathcal{M}_h \colon U_0 \subseteq \mathbb{H} \to \dot{\mathbb{H}}^{-\beta}$ is well-defined. The measurability of $G$ follows naturally from the measurability of $g$. 

Since $Q^{1/2}h_n \in \mathbb{H}$ and $g(t,\omega,\cdot,u(\cdot))e_m \in \mathbb{H}$ (as $e_m \in L^\infty(\mathcal{O}; \mathbb{R}^N)$), the duality pairing $\langle\cdot,\cdot\rangle_{\dot{\mathbb{H}}^{-\beta}, \dot{\mathbb{H}}^\beta}$ coincides with the standard inner product in $\mathbb{H}$. For any $u \in \dot{\mathbb{H}}^\nu$, utilizing the self-adjoint and positive semidefinite properties of the covariance operator $Q$, we have: 
\begin{align}\label{eq:G_well-definedness}
    \nonumber \|G(t,\omega, u)\|_{\mathrm{HS}(U_0, \dot{\mathbb{H}}^{-\beta})}^2 &= \sum_{n=1}^\infty \| G(t,\omega, u)Q^{1/2}h_n\|_{\dot{\mathbb{H}}^{-\beta}}^2 \\  \nonumber &=\sum_{n=1}^\infty\|\mathcal{M}_{\mathcal{N}_g(t,\omega,u)}Q^{1/2}h_n\|_{\dot{\mathbb{H}}^{-\beta}}^2\\
    \nonumber &= \sum_{n=1}^\infty\|\mathcal{N}_g(t,\omega,u)Q^{1/2}h_n\|_{\dot{\mathbb{H}}^{-\beta}}^2\\
    \nonumber &= \mathbbm{1}_{(0,T]}(t) \sum_{m=1}^\infty \eta_m^{-\beta} \sum_{n=1}^\infty \langle g(t,\omega,\cdot,u(\cdot)) Q^{1/2}h_n, e_m \rangle_{\dot{\mathbb{H}}^{-\beta},\dot{\mathbb{H}}^\beta}^2 \\
    &= \mathbbm{1}_{(0,T]}(t) \sum_{m=1}^\infty \eta_m^{-\beta} \sum_{n=1}^\infty \big( Q^{1/2}h_n, g(t,\omega,\cdot,u(\cdot)) e_m \big)_{\mathbb{H}}^2.
\end{align}
To evaluate the inner sum over $n$, we utilize the self-adjointness of $Q^{1/2}$ and apply Parseval's identity with respect to the orthonormal basis $\{h_n\}_{n \in \mathbb{N}}$ of $\mathbb{H}$:
\begin{align*}
    \sum_{n=1}^\infty \big( Q^{1/2}h_n, g(t,\omega,\cdot,u(\cdot)) e_m \big)_{\mathbb{H}}^2 &= \sum_{n=1}^\infty \big( h_n, Q^{1/2} \big( g(t,\omega,\cdot,u(\cdot)) e_m \big) \big)_{\mathbb{H}}^2 \\
    &= \big\| Q^{1/2} \big( g(t,\omega,\cdot,u(\cdot)) e_m \big) \big\|_{\mathbb{H}}^2 \\
    &= \big( Q \big( g(t,\omega,\cdot,u(\cdot)) e_m \big), g(t,\omega,\cdot,u(\cdot)) e_m \big)_{\mathbb{H}} \\
    &\le \|Q\|_{\mathcal{L}(\mathbb{H})} \|g(t,\omega,\cdot,u(\cdot)) e_m\|_{\mathbb{H}}^2 \\
    &\le \|Q\|_{\mathcal{L}(\mathbb{H})} \|g(t,\omega,\cdot,u(\cdot))\|_{\mathbb{H}}^2 \|e_m\|_{L^\infty(\mathcal{O};\mathbb{R}^N)}^2.
\end{align*}
Substituting this bound back into \eqref{eq:G_well-definedness} yields:
\begin{align}\label{eq:G_HS_bound_final}
    \|G(t,\omega, u)\|_{\mathrm{HS}(U_0, \dot{\mathbb{H}}^{-\beta})}^2 \le \mathbbm{1}_{(0,T]}(t) \|Q\|_{\mathcal{L}(\mathbb{H})} \|g(t,\omega,\cdot,u(\cdot))\|_{\mathbb{H}}^2 \sum_{m=1}^\infty \eta_m^{-\beta} \|e_m\|_{L^\infty(\mathcal{O};\mathbb{R}^N)}^2.
\end{align}
Because $d=1$, the domain $\mathcal{O}$ is a bounded interval, and the eigenfunctions of the Sturm-Liouville operator $A$ are uniformly bounded in amplitude. Thus, there exists a constant $C > 0$ such that $\|e_m\|_{L^\infty(\mathcal{O};\mathbb{R}^N)} \le C$ for all $m \in \mathbb{N}$. Furthermore, the one-dimensional Weyl's law dictates that $\eta_m \sim m^2$. Combining these facts, the series converges because $\beta > 1/2$:
\begin{equation}\label{eq:series_eig_sys_Dirichlet_Lap}
\sum_{m=1}^\infty \eta_m^{-\beta} \|e_m\|_{L^\infty(\mathcal{O};\mathbb{R}^N)}^2 \le C^2 \sum_{m=1}^\infty \eta_m^{-\beta} \le C' \sum_{m=1}^\infty m^{-2\beta} < \infty.
\end{equation}
Since $\nu \ge \frac{\theta}{2(\theta+1)}$ (setting $d=1$), the Nemytskii operator maps into $\mathbb{H}$, yielding $\|g(t,\omega,\cdot,u(\cdot))\|_{\mathbb{H}}^2 < \infty$. Applying \eqref{eq:series_eig_sys_Dirichlet_Lap} to \eqref{eq:G_HS_bound_final} deduces that $\|G(t,\omega, u)\|_{\mathrm{HS}(U_0, \dot{\mathbb{H}}^{-\beta})} < \infty$ for all $u \in \dot{\mathbb{H}}^\nu$.

Consequently, the mapping $G$ is well-defined and measurable. Following an identical procedure, for any $u,v \in \dot{\mathbb{H}}^\nu$ and $t>0$, we obtain:
\begin{align}
\label{eq:G_lip_main} \nonumber\|G(t,\omega, u) - G(t,\omega, v)\|_{\mathrm{HS}(U_0, \dot{\mathbb{H}}^{-\beta})}^2  &\le C \|g(t,\omega,\cdot,u(\cdot)) - g(t,\omega,\cdot,v(\cdot))\|_{\mathbb{H}}^2\\
\nonumber &\le C t^{-2\sigma}\big(1+\|u\|_{L^{2(\theta+1)}(\mathcal{O};\mathbb{R}^N)}^{2\theta} + \|v\|_{L^{2(\theta+1)}(\mathcal{O};\mathbb{R}^N)}^{2\theta}\big)\|u-v\|_{L^{2(\theta+1)}(\mathcal{O};\mathbb{R}^N)}^2\\
&\le C t^{-2\sigma}\big(1+\|u\|_{\dot{\mathbb{H}}^\nu}^{2\theta} + \|v\|_{\dot{\mathbb{H}}^\nu}^{2\theta}\big)\|u-v\|_{\dot{\mathbb{H}}^\nu}^2.
\end{align}
Estimate \eqref{eq:G_lip_main} confirms that $G$ satisfies the global Lipschitz condition of Assumption~\ref{ass:global_lips_nonlin} for $\theta=0$ and $\hat{\beta}=\sigma$, and the local Lipschitz condition of Assumption~\ref{ass:local_lip_stoch} for $\theta \ge 0$ and $\hat{\beta}=\sigma+(\theta+1)\lambda$.

\noindent\textbf{Case 2:} The eigensystem $\{q_n,h_n\}_{n\in\mathbb{N}}$ of $Q$ satisfies \eqref{eq:smooth_noise}.

\noindent In this case, since each $\phi \in U_0$ belongs to $L^{\infty}(\mathcal{O}; \mathbb{R}^N)$, the multiplication operator $\mathcal{M}_h \colon U_0 \subseteq \mathbb{H} \to \dot{\mathbb{H}}^{-\beta}$ is well-defined for all $\beta \ge 0$. Measurability follows as before. For any $u \in \dot{\mathbb{H}}^\nu$, we have:
\begin{align}\label{eq:G_well-definedness_case_2}
    \nonumber \|G(t,\omega, u)\|_{\mathrm{HS}(U_0, \dot{\mathbb{H}}^{-\beta})}^2 &= \sum_{n=1}^\infty \| G(t,\omega, u)Q^{1/2}h_n\|_{\dot{\mathbb{H}}^{-\beta}}^2 = \sum_{n=1}^\infty\|\mathcal{M}_{\mathcal{N}_g(t,\omega,u)}Q^{1/2}h_n\|_{\dot{\mathbb{H}}^{-\beta}}^2\\
    \nonumber &= \sum_{n=1}^\infty\|\mathcal{N}_g(t,\omega,u)Q^{1/2}h_n\|_{\dot{\mathbb{H}}^{-\beta}}^2\\
    \nonumber &\le C \mathbbm{1}_{(0,T]}(t)\sum_{n=1}^\infty\|g(t,\omega,\cdot,u(\cdot))Q^{1/2}h_n\|_{\mathbb{H}}^2 \\
    &\le C \mathbbm{1}_{(0,T]}(t)\|g(t,\omega,\cdot,u(\cdot))\|_{\mathbb{H}}^2\sum_{n=1}^\infty q_n\|h_n\|_{L^\infty(\mathcal{O};\mathbb{R}^N)}^2 < \infty. 
\end{align}
Thus, $G$ is well-defined and measurable. Proceeding similarly to the previous case, we find for $u,v \in \dot{\mathbb{H}}^\nu$ and $t>0$:
\begin{align}
    \label{eq:G_lip_main_case_2} \nonumber\|G(t,\omega, u) - G(t,\omega, v)\|_{\mathrm{HS}(U_0, \dot{\mathbb{H}}^{-\beta})}^2  &\le C \|g(t,\omega,\cdot,u(\cdot)) - g(t,\omega,\cdot,v(\cdot))\|_{\mathbb{H}}^2\\
    \nonumber &\le C t^{-2\sigma}\big(1+\|u\|_{L^{2(\theta+1)}(\mathcal{O};\mathbb{R}^N)}^{2\theta} + \|v\|_{L^{2(\theta+1)}(\mathcal{O};\mathbb{R}^N)}^{2\theta}\big)\|u-v\|_{L^{2(\theta+1)}(\mathcal{O};\mathbb{R}^N)}^2\\
     &\le C t^{-2\sigma}\big(1+\|u\|_{\dot{\mathbb{H}}^\nu}^{2\theta} + \|v\|_{\dot{\mathbb{H}}^\nu}^{2\theta}\big)\|u-v\|_{\dot{\mathbb{H}}^\nu}^2.
\end{align}
Estimate \eqref{eq:G_lip_main_case_2} confirms that $G$ fulfills both the global and local Lipschitz conditions under the exact parameters stated in Case~1. This completes the proof of the lemma.
\end{proof}

\begin{remark}
Observe that the weighted-trace-class condition \eqref{eq:smooth_noise} is not required, and the standard trace-class requirement $\mathrm{Tr}(Q) \coloneqq \sum_{n=1}^{\infty}q_n < \infty$ is sufficient, if the diffusion coefficient $g$ is globally uniformly bounded (i.e., $\sup_{(t,\omega,\boldsymbol{x},y) \in \Omega_T \times \mathcal{O}\times \mathbb{R}^N} t^{\sigma} |g(t,\omega,\boldsymbol{x},y)| < \infty$). In this scenario, the $L^\infty$ norm of $g$ can be extracted directly, negating the need for $L^\infty$ bounds on $h_n$. Furthermore, condition \eqref{eq:smooth_noise} reduces to the standard trace-class requirement if $\mathcal{O} \subset \mathbb{R}^d$ is a bounded cube and $\{h_n\}_{n \in \mathbb{N}}$ are the eigenfunctions of the Dirichlet Laplacian, as their uniform $L^\infty$ bounds allow the norm to be factored out of the series.
\end{remark}

\subsection{The parabolic Anderson model}\label{subsec:PAM}
The parabolic Anderson model (PAM) classically characterizes the transport of a substance through a random medium or the evolution of a population within a random environment (see \cite{MR1185878}). Utilizing the framework established by Walsh \cite{MR876085}, Chen and Dalang \cite{MR3255231, MR3433576} investigated the existence, uniqueness, and regularity results for this model on the entire real line subject to rough initial data. More recently, employing a semigroup-theoretic approach, Andersson et~al.~\cite[Eq.~(9)]{MR4172830} investigated the well-posedness and regularity results for a continuous version of the PAM on a bounded interval with singular initial conditions, such as the derivative of the Dirac measure at the origin. We consider the following continuous version of the PAM on a bounded domain:
\begin{equation}\label{eq:PAM}
\left\{
\begin{aligned}
    dX(t,\boldsymbol{x}) &= \Delta X(t,\boldsymbol{x}) \, dt + X(t,\boldsymbol{x}) \, dW(t,\boldsymbol{x}), && (t,\boldsymbol{x}) \in (0,T] \times \mathcal{O}, \\[1ex]
    X(t,\boldsymbol{x}) &= 0, && (t,\boldsymbol{x}) \in (0,T] \times \partial\mathcal{O}, \\[1ex]
    X(0,\boldsymbol{x}) &= X_0(\boldsymbol{x}), && \boldsymbol{x} \in \mathcal{O},
\end{aligned}
\right.
\end{equation}
 where $W$ is a $Q$-Wiener process defined on the spatial Hilbert space $\mathbb{H} = L^2(\mathcal{O})$, and the initial condition satisfies $X_0 \in L^p(\Omega, \mathcal{F}_0; \dot{\mathbb{H}}^{\mu})$ for $p\geq 2$. 
 
 We define the linear operator $ A \coloneqq -\Delta +I $ with the domain $\mathcal{D}(A) = H^2(\mathcal{O}) \cap H_0^1(\mathcal{O})$, which satisfies Assumption~\ref{ass:operator_A} and generates the fractional Sobolev spaces $\dot{\mathbb{H}}^\gamma$. The associated drift operator $F \colon [0,T] \times \Omega \times \dot{\mathbb{H}}^\mu \to \dot{\mathbb{H}}^{-\alpha}$ is given by $F(t, \omega, u) = u$, which satisfies the measurability condition in Assumption~\ref{ass:measurability_drift_diff}, and the boundedness and global Lipschitz conditions in Assumption~\ref{ass:global_lips_nonlin} with parameters \(\nu=0\), \(-\mu\leq \alpha <2\), and $\hat{\alpha} = 0$,  and constants $L_1=1$ and $L_2 = 0$. Furthermore,  since the diffusion term $g(y) = y$ satisfies the hypothesis of Lemma~\ref{lem:cond_on_diffusion} with \(\theta=0\) and \(\sigma=0\), the mapping $G \colon [0,T] \times \Omega \times \dot{\mathbb{H}}^\mu \to \mathrm{HS}(U_0, \dot{\mathbb{H}}^{-\beta})$, with $U_0 = Q^{1/2}(\mathbb{H})$ defined in \eqref{eq:nemytskii_G}, satisfies the measurability condition in Assumption~\ref{ass:measurability_drift_diff}, and global Lipschitz and boundedness conditions in Assumption~\ref{ass:global_lips_nonlin} with parameters \(\nu=0\), \(0\leq \beta <1\), and $\hat{\beta} = 0$,  and constants $L_3=1$ and $L_4 = 0$. Hence, by Theorems~\ref{thm:main_global_existence} and \ref{thm:pathwise_regularity}, well-posedness and pathwise regularity are guaranteed under the following parameter constraints:
\[
\begin{aligned}
\text{(i)}\;& -1 < \mu \le 0, \qquad
\text{(ii)}\; \nu = 0, \qquad
\text{(iii)}\; -\mu \le \alpha <2, \qquad
\text{(iv)}\; 0\leq \beta <1 \text{ when the eigensystem } \{q_n, h_n\}_{n\in\mathbb{N}} \\[4pt] &\text{ of } Q \text{ satisfies \eqref{eq:smooth_noise}, otherwise }  \max\left\{\frac{d}{2}, \frac{2d-1}{2}\right\} <\beta<1.
\end{aligned}
\]

It is worth noting that our framework accommodates the parabolic Anderson model driven by space-time white noise ($Q = \mathrm{Id}_{\mathbb{H}}$) subject to rough initial data, such as the Dirac measure, in dimension $d=1$.

\subsection{Stochastic solid fuel ignition model in 1D}\label{subsec:SPDEs_with_exponential_growth}
In this subsection, we demonstrate that our framework accommodates nonlinearities with exponential growth. In the deterministic setting, equations featuring exponential nonlinearities, such as the initial value problem $\partial_t u = \Delta u + e^u$ and its steady-state counterpart $\Delta u + e^u = 0$, originate from combustion theory, where they model the thermal self-ignition of chemically active gases \cite{fujita1969nonlinear, MR1012946}. Today, these equations serve as foundational models in the comprehensive theory of superlinear parabolic PDEs \cite{MR1012946, MR2346798}. To this end, we consider the following stochastic extension of the classical solid fuel ignition model, often referred to as the parabolic Gelfand problem:
\begin{equation}\label{eq:exponential_growth}
\left\{
\begin{aligned}
    dX(t,x) &= \left[ \Delta X(t,x) - X(t,x) + e^{t^\kappa X(t,x)} \right] dt \\
            &\quad + g(t, x, X(t,x)) \, dW(t,x), && (t,x) \in (0,T] \times \mathcal{O}, \\[1ex]
    X(t,x) &= 0, && (t,x) \in (0,T] \times \partial\mathcal{O}, \\[1ex]
    X(0,x) &= X_0(x), && x \in \mathcal{O},
\end{aligned}
\right.
\end{equation}
where $\kappa \ge 0$, $\mathcal{O} \subset \mathbb{R}$ is a bounded interval, $W$ is a $Q$-Wiener process defined on $\mathbb{H} = L^2(\mathcal{O})$, and $X_0 \in L^0(\Omega, \mathcal{F}_0; \dot{\mathbb{H}}^\mu)$ is an $\mathcal{F}_0$-measurable random variable. The diffusion coefficient $g \colon [0,T] \times \Omega \times \mathcal{O} \times \mathbb{R} \to \mathbb{R}$ satisfies the hypothesis of Lemma~\ref{lem:cond_on_diffusion} with $\theta=0$ and $\sigma=0$. We define the linear operator $A \coloneqq -\Delta + I$ with the domain $\mathcal{D}(A) = H^2(\mathcal{O}) \cap H^1_0(\mathcal{O})$. For $f(t, y) \coloneqq e^{t^\kappa y}$, the drift mapping $F \colon [0,T] \times \Omega \times \dot{\mathbb{H}}^\mu \to \mathbb{H}$ is defined via the associated Nemytskii operator:
\begin{equation}\label{eq:nemytskii_exponential_growth}
F(t,\omega,u) \coloneqq \begin{cases}
        \mathcal{N}_f(t, u) &: u \in \dot{\mathbb{H}}^{\nu},\\
        0 &: \text{otherwise},
    \end{cases} \quad \forall (t,\omega, u )\in [0,T]\times\Omega\times \dot{\mathbb{H}}^{\mu},
\end{equation}
where $(\mathcal{N}_f(t, u))(x) = e^{t^\kappa u(x)}$. To establish the local Lipschitz estimate, we assume $\kappa \ge \lambda \coloneqq (\nu-\mu)/2$. Furthermore, we require $\nu > 1/2$ to guarantee the continuous Sobolev embedding $\dot{\mathbb{H}}^\nu \hookrightarrow L^\infty(\mathcal{O})$. Applying the Mean Value Theorem, we obtain:
\begin{align}
\|F(t,u) - F(t,v)\|_{\mathbb{H}}^2 &= \int_{\mathcal{O}} \left|e^{t^\kappa u(x)} - e^{t^\kappa v(x)}\right|^2 \, dx \nonumber \\
&\le  t^{2\kappa} \int_{\mathcal{O}} e^{2t^\kappa \max\{|u(x)|, |v(x)|\}} |u(x) - v(x)|^2 \, dx \nonumber \\
&\le t^{2\kappa} e^{2C T^{\kappa-\lambda} R} \int_{\mathcal{O}} |u(x) - v(x)|^2 \, dx \nonumber \\
&= t^{2\kappa} e^{2C T^{\kappa-\lambda} R} \|u - v\|_{\mathbb{H}}^2, \label{eq:drift_lipschitz_base}
\end{align}
where we have utilized the fact that for any $R > 0$, $(t,\omega) \in (0,T] \times \Omega$, and $u, v \in \{\phi \in \dot{\mathbb{H}}^\nu : \|\phi\|_{\lambda, t} \le R\}$, 
\begin{equation*}
t^\kappa \|u\|_{L^\infty(\mathcal{O})} \le C t^{\kappa-\lambda} t^\lambda \|u\|_{\dot{\mathbb{H}}^\nu} \le C T^{\kappa-\lambda} \|u\|_{\lambda, t} \le C T^{\kappa-\lambda} R,
\end{equation*}
which implies $t^\kappa \max\{|u(x)|, |v(x)|\} \le C T^{\kappa-\lambda} R$ for almost every $x \in \mathcal{O}$. 

\noindent \textbf{Case 1: $\mu \ge 0$.} In this regime, taking the square root of \eqref{eq:drift_lipschitz_base} and applying the continuous embedding $\dot{\mathbb{H}}^\mu \hookrightarrow \mathbb{H}$ yields:
\begin{equation}\label{eq:drift_case1}
\|F(t,u) - F(t,v)\|_{\mathbb{H}} \le C t^\kappa e^{C T^{\kappa-\lambda} R} \|u - v\|_{\lambda, t}.
\end{equation}

\noindent \textbf{Case 2: $\mu < 0$.} While we will shortly see that multiplicative noise ultimately restricts the parameter $\mu > 0$, this analysis is essential for the additive noise scenario discussed in Remark~\ref{rem:exponential_growth_additive_noise}, where this restriction is lifted. 

Taking the square root of \eqref{eq:drift_lipschitz_base}, bounding the difference $z = u-v$ via standard fractional interpolation $\|z\|_{\mathbb{H}} \le M \|z\|_{\dot{\mathbb{H}}^\mu}^{1-\vartheta} \|z\|_{\dot{\mathbb{H}}^\nu}^\vartheta$ with $\vartheta = \frac{-\mu}{\nu-\mu} \in (0,1)$, and applying Young's inequality ($x^{1-\vartheta} y^\vartheta \le x + y$ for $x,y \ge 0$), we obtain:
\begin{align}
\|F(t,u) - F(t,v)\|_{\mathbb{H}} &\le M e^{C T^{\kappa-\lambda} R} t^{\kappa} \|z\|_{\dot{\mathbb{H}}^\mu}^{1-\vartheta} \|z\|_{\dot{\mathbb{H}}^\nu}^\vartheta \nonumber \\
&= M e^{C T^{\kappa-\lambda} R} t^{\kappa - \lambda\vartheta} \|z\|_{\dot{\mathbb{H}}^\mu}^{1-\vartheta} (t^\lambda \|z\|_{\dot{\mathbb{H}}^\nu})^\vartheta \nonumber \\
&\le M e^{C T^{\kappa-\lambda} R} t^{\kappa - \lambda\vartheta} \left(\|z\|_{\dot{\mathbb{H}}^\mu} + t^\lambda \|z\|_{\dot{\mathbb{H}}^\nu}\right) \nonumber \\
&= M e^{C T^{\kappa-\lambda} R} t^{\kappa - \lambda\vartheta} \|u-v\|_{\lambda, t}. \label{eq:drift_case2}
\end{align}
Combining \eqref{eq:drift_case1} and \eqref{eq:drift_case2}, we conclude that for all $\mu \le \nu$, the drift satisfies:
\begin{equation}\label{eq:drift_final_lipschitz}
\|F(t,u) - F(t,v)\|_{\mathbb{H}} \le L_1(R) t^{-\hat{\alpha}} \|u-v\|_{\lambda, t},
\end{equation}
where $\hat{\alpha} = 0$ and $L_1(R) \coloneqq \max\{C T^\kappa, M T^{\kappa + \mu/2}\} e^{C T^{\kappa-\lambda} R}$. Note that for all $(t,\omega) \in (0,T] \times \Omega$, we have $\|F(t,\omega,0)\|_{\mathbb{H}} = \sqrt{|\mathcal{O}|}<\infty$. Thus, the drift operator satisfies Assumption~\ref{ass:local_lip_stoch} with $\alpha = 0$ and $\hat{\alpha} = 0$. 

Furthermore, since the diffusion coefficient $g \colon [0,T] \times \Omega \times \mathcal{O} \times \mathbb{R} \to \mathbb{R}$ satisfies the hypothesis of Lemma~\ref{lem:cond_on_diffusion} with $\theta=0$ and $\sigma=0$, the corresponding mapping $G \colon [0,T] \times \Omega \times \dot{\mathbb{H}}^\mu \to \mathrm{HS}(U_0, \dot{\mathbb{H}}^{-\beta})$ defined in \eqref{eq:nemytskii_G} satisfies the measurability condition of Assumption~\ref{ass:measurability_drift_diff} and the local Lipschitz condition of Assumption~\ref{ass:local_lip_stoch} with $\hat{\beta} = \lambda = \frac{\nu-\mu}{2}$. Thus, under the following parameter constraints:
\[
\begin{aligned}
    &\text{(i)}~ \nu \ge \mu, \quad &&\text{(ii)}~ \nu > \frac{1}{2}, \quad &&\text{(iii)}~ \alpha = 0, \\
    &\text{(iv)}~ \nu - \mu < \frac{1}{2}, \quad &&\text{(v)}~ \frac{1}{2} < \beta < 1 - 2\nu + \mu,
\end{aligned}
\]
an appeal to Theorem~\ref{thm:local_existence} guarantees the existence of a unique maximal local mild solution of the SPDE \eqref{eq:exponential_growth}. Moreover, if the eigensystem $\{q_n, h_n\}_{n\in\mathbb{N}}$ of $Q$ satisfies \eqref{eq:smooth_noise}, then constraint (v) $1/2 < \beta < 1-2\nu+\mu$ can be relaxed to $0 \le \beta < 1-2\nu+\mu$. (Note that while the drift bound permits $\mu < 0$ for $\kappa > 1/4$, constraints (ii) and (iv) derived from the diffusion term parameter $\mu$ to be positive ($\mu > 0$)).

\begin{remark}[Additive noise and rough initial data]\label{rem:exponential_growth_additive_noise}
If the diffusion coefficient $g \colon [0,T] \times \Omega \times \mathcal{O} \times \mathbb{R} \to \mathbb{R}$ is independent of the state variable (i.e., $g(t,\omega,x,u) = g(t,\omega,x)$) and satisfies the hypothesis of Lemma~\ref{lem:cond_on_diffusion}, the stochastic perturbation is purely additive. Consequently, the induced mapping $G(t,\omega,u) \equiv G(t,\omega)$ is independent of $u$, trivially satisfies the required Lipschitz condition. If $g$ additionally satisfies \eqref{eq:g_boundedness} with $\sigma = 0$, Lemma~\ref{lem:cond_on_diffusion} yields $\hat{\beta} = 0$.

As noted in Remark~\ref{rem:global_case_comparison}(i), the parameter restrictions for additive noise are relaxed to $0\leq \nu-\mu<2$ and $0\leq \hat{\beta} < \frac12-\frac12\max\left\{\frac{\nu-\mu}{2},\beta+\nu\right\}$. Consequently, under the weighted-trace-class noise assumption \eqref{eq:smooth_noise}, we have the following parameter constraints:
\[
\begin{aligned}
    &\text{(i)}~ \nu \ge \mu, \quad &&\text{(ii)}~ \nu > \frac{1}{2}, \quad &&\text{(iii)}~ \alpha = 0, \\
    &\text{(iv)}~ \nu - \mu < 2, \quad &&\text{(v)}~ 0 \le \beta < 1 - \nu.
\end{aligned}
\]
By selecting $\nu \in (1/2, 1)$ and $\beta \in [0, 1-\nu)$, the parameter $\mu$ is permitted to be negative ($\mu > -1.5$). Thus, when subject to weighted-trace-class additive noise, the SPDE \eqref{eq:exponential_growth} can accommodate rough initial data.
\end{remark}

\subsection{The stochastic Allen--Cahn equation}\label{subsec:SACE}
The stochastic Allen--Cahn equation is a classical reaction-diffusion model describing phase separation under random perturbations. Due to its bistable polynomial nonlinearity, it serves as a standard benchmark for the theoretical and numerical analysis of SPDEs with locally Lipschitz coefficients, especially since standard explicit schemes are known to diverge in this setting (see Hutzenthaler et al.~\cite{MR2985171}). For comprehensive treatments of the underlying theory and specialized numerical methods that accommodate this polynomial growth, we refer the reader to Lord et al.~\cite{MR3308418}, Wang~\cite{MR4140034}, and Kov\'{a}cs et al.~\cite{MR3372078}.

We consider the following stochastic Allen--Cahn equation:
\begin{equation} \label{eq:SACE} 
\left\{ 
\begin{aligned}     
    dX(t,\boldsymbol{x}) &= \left[\Delta X(t,\boldsymbol{x}) + X(t,\boldsymbol{x}) - X^3(t,\boldsymbol{x})\right] dt \\     
    &\quad + g(t,\boldsymbol{x},X(t,\boldsymbol{x})) \, dW(t,\boldsymbol{x}), && (t,\boldsymbol{x}) \in (0,T] \times \mathcal{O}, \\     
    X(t,\boldsymbol{x}) &= 0, && (t,\boldsymbol{x}) \in (0,T] \times \partial\mathcal{O}, \\     
    X(0,\boldsymbol{x}) &= X_0(\boldsymbol{x}), && \boldsymbol{x} \in \mathcal{O}, 
\end{aligned} 
\right.
\end{equation}
where $\mathcal{O} \subset \mathbb{R}^d$, $W$ is a $Q$-Wiener process defined on $\mathbb{H}=L^2(\mathcal{O})$, $X_0 \colon \Omega \to \dot{\mathbb{H}}^\mu$ is an $\mathcal{F}_0$-measurable random variable, and the diffusion coefficient $g \colon [0,T] \times \Omega \times \mathcal{O} \times \mathbb{R} \to \mathbb{R}$ satisfies the hypothesis of Lemma~\ref{lem:cond_on_diffusion} with $\theta=0$ and $\sigma=0$. 

The operator $A \coloneqq -\Delta + I$ with the domain $\mathcal{D}(A) = H^2(\mathcal{O}) \cap H_0^1(\mathcal{O})$ satisfies Assumption~\ref{ass:operator_A} and generates a family of fractional-order spaces $\dot{\mathbb{H}}^\gamma \coloneqq \mathcal{D}(A^{\gamma/2})$ for $\gamma \in \mathbb{R}$. Assume that
\begin{equation}
    \frac{d}{6} < \nu, \quad 0 \leq \alpha, \quad \text{and} \quad \alpha + 3\nu \geq d.
\end{equation}
Then, there exists an \(r\in(1,2]\) such that
\begin{equation}
     \frac{1}{3r} \geq \frac{1}{2} - \frac{\nu}{d} \quad \text{and} \quad \frac{1}{r} \leq \frac{1}{2} + \frac{\alpha}{d}.
\end{equation}
Consequently, by Lemma~\ref{lem:sobolev_embedding}, 
\begin{align}
    \nonumber & \dot{\mathbb{H}}^{\nu} \hookrightarrow L^{3r}(\mathcal{O})\;\text{ and }\;L^{r}(\mathcal{O})\hookrightarrow \dot{\mathbb{H}}^{-\alpha}. 
\end{align}
Let \(f(y)\coloneqq 2y - y^3\). Then, for \(\mu\leq \nu\), we define the map \(F:[0,T]\times\Omega\times \dot{\mathbb{H}}^{\mu}\to \dot{\mathbb{H}}^{-\alpha} \) by 
\begin{equation}\label{eq:nemytskii_allen_cahn}
     F(t,\omega,u) \coloneqq \begin{cases}
        j_r\mathcal{N}_f(u) &: u\in \dot{\mathbb{H}}^{\nu},\\
        0 &: u\in \dot{\mathbb{H}}^{\mu}\backslash \dot{\mathbb{H}}^{\nu},
    \end{cases} \quad \forall (t,\omega, u )\in [0,T]\times\Omega\times \dot{\mathbb{H}}^{\mu}.
\end{equation}
This mapping $F$ is well-defined and measurable from the measurable space $(\Omega_T \times \dot{\mathbb{H}}^{\mu}, \mathcal{P}_T \otimes \mathcal{B}(\dot{\mathbb{H}}^{\mu}))$ to $(\dot{\mathbb{H}}^{-\alpha}, \mathcal{B}(\dot{\mathbb{H}}^{-\alpha}))$, where \(j_r:L^r(\mathcal{O})\hookrightarrow \dot{\mathbb{H}}^{-\alpha}\) is the canonical continuous embedding and \(\mathcal{N}_f:\dot{\mathbb{H}}^{\nu} \to L^r(\mathcal{O})\) such that \((\mathcal{N}_f(u))(\boldsymbol{x}) = f(u(\boldsymbol{x}))\) denotes the Nemytskii operator. Note that, for all \((t,\omega)\in(0,T]\times\Omega\), \(\|F(t,\omega,0) \|_{\dot{\mathbb{H}}^{-\alpha}} = 0\). Moreover, 
for any given \(R>0\), \((t,\omega)\in(0,T]\times\Omega\), and \(u,v\in\{\phi\in  \dot{\mathbb{H}}^{\nu} :\,\|\phi\|_{\lambda,t}\leq R,\;\lambda\coloneqq (\nu-\mu)/2\}\),
\begin{align}
    \nonumber \|F(t,\omega,u) - F(t,\omega,v)\|_{\dot{\mathbb{H}}^{-\alpha}}& \leq \|j_r\|_{\mathcal{L}(L^r(\mathcal{O}); \dot{\mathbb{H}}^{-\alpha})}\|\mathcal{N}_f(u) - \mathcal{N}_f(v)\|_{L^r(\mathcal{O})}\\
    \nonumber& = \|j_r\|_{\mathcal{L}(L^r(\mathcal{O}); \dot{\mathbb{H}}^{-\alpha})}\|2(u-v) - (u^3 - v^3)\|_{L^r(\mathcal{O})}\\
    \nonumber& \leq \|j_r\|_{\mathcal{L}(L^r(\mathcal{O}); \dot{\mathbb{H}}^{-\alpha})}\Big(2\|u-v\|_{L^r(\mathcal{O})} +\| u^2 + uv +v^2\|_{L^{3r/2}(\mathcal{O})} \|u-v\|_{L^{3r}(\mathcal{O})}\Big)\\
\nonumber& \leq C\Big(\|u-v\|_{L^{3r}(\mathcal{O})} +\big(\| u\|_{L^{3r}(\mathcal{O})}^2 +\|v\|_{L^{3r}(\mathcal{O})}^2\big) \|u-v\|_{L^{3r}(\mathcal{O})}\Big)\\
\nonumber& \leq C \Big(\|u-v\|_{\dot{\mathbb{H}}^\nu} + \big(\| u\|_{\dot{\mathbb{H}}^\nu}^2 +\|v\|_{\dot{\mathbb{H}}^\nu}^2\big) \|u-v\|_{\dot{\mathbb{H}}^\nu}\Big)\\
&\leq L_1(R) \,t^{-\hat{\alpha}} \|u-v\|_{\lambda,t},
\end{align}
where \(L_1(R) \coloneqq C(T^{2\lambda} + 2R^2)\) and \(\hat{\alpha}\coloneqq 3\lambda = \frac{3(\nu-\mu)}{2}\). 

Furthermore, since the diffusion coefficient $g \colon [0,T] \times \Omega \times \mathcal{O} \times \mathbb{R} \to \mathbb{R}$ satisfies the hypothesis of Lemma~\ref{lem:cond_on_diffusion} with $\theta=0$ and $\sigma=0$, the corresponding mapping $G \colon [0,T] \times \Omega \times \dot{\mathbb{H}}^\mu \to \mathrm{HS}(U_0, \dot{\mathbb{H}}^{-\beta})$ defined in \eqref{eq:nemytskii_G} satisfies the measurability condition of Assumption~\ref{ass:measurability_drift_diff} and the local Lipschitz condition of Assumption~\ref{ass:local_lip_stoch} with $\hat{\beta} = \lambda = \frac{\nu-\mu}{2}$. Thus, under the following parameter constraints
\[
\begin{aligned}
    &\text{(i)}~ \nu \ge \mu, \quad &&\text{(ii)}~ \nu > \frac{d}{6}, \quad &&\text{(iii)}~ \alpha \geq 0, \quad &&\text{(iv)}~ 3\nu + \alpha \geq d, \\
    &\text{(v)}~ \nu - \mu < \frac{1}{2}, \quad &&\text{(vi)}~ 4\nu - 3\mu + \alpha < 2, \quad &&\text{(vii)}~ \max\left\{\frac{d}{2},\frac{2d-1}{2} \right\}<\beta < 1 +\mu- 2\nu,
\end{aligned}
\]
an appeal to Theorem~\ref{thm:local_existence} guarantees the existence of a unique maximal local mild solution of the stochastic Allen--Cahn equation \eqref{eq:SACE}. Moreover, if the eigensystem $\{q_n, h_n\}_{n\in\mathbb{N}}$ of $Q$ satisfies \eqref{eq:smooth_noise}, then constraint (vii) $\max\{d/2,(2d-1)/2\} < \beta < 1+\mu-2\nu$ can be relaxed to $0 \leq \beta < 1+\mu-2\nu$.

In particular, in dimension $d=1$, by fixing $\nu = 1/6 + \epsilon$ for $0 < \epsilon < 1/12$, one can choose the following set of parameter values:

\begin{itemize}
    \item For (weighted) trace-class noise (i.e., the eigensystem $\{q_n, h_n\}_{n \in \mathbb{N}}$ of $Q$ satisfies \eqref{eq:smooth_noise}):
\[
\begin{aligned}
    &\text{(a)}~ -\frac{5}{18} + \frac{\epsilon}{3} < \mu \le \frac{1}{6}+\epsilon, \quad &&\text{(b)}~ \nu = \frac{1}{6}+\epsilon, \\
    &\text{(c)}~ \frac{1}{2}-3\epsilon \leq \alpha <\frac{4}{3}-4\epsilon + 3\mu, \quad &&\text{(d)}~ 0 \le \beta < \frac{2}{3}-2\epsilon + \mu.
\end{aligned}
\]

\item For space-time white noise (i.e., $Q = \mathrm{Id}_{\mathbb{H}}$):
\[
\begin{aligned}
    &\text{(a)}~ -\frac{1}{6} + 2\epsilon < \mu \le \frac{1}{6}+\epsilon, \quad &&\text{(b)}~ \nu = \frac{1}{6}+\epsilon, \\
    &\text{(c)}~ \frac{1}{2}-3\epsilon \leq \alpha <\frac{4}{3}-4\epsilon + 3\mu, \quad &&\text{(d)}~ \frac{1}{2} < \beta < \frac{2}{3}-2\epsilon + \mu.
\end{aligned}
\]
\end{itemize}

\begin{remark}
    We note that the parameter constraints permit $\mu < 0$ in dimension $d=1$, meaning the abstract framework accommodates rough initial data. Taking the limit as $\epsilon \to 0$ in condition (a) yields $\mu > -5/18$ under (weighted) trace-class noise (where the eigensystem of $Q$ satisfies \eqref{eq:smooth_noise}), and $\mu > -1/6$ when the equation is driven by space-time white noise ($Q = \mathrm{Id}_{\mathbb{H}}$).
\end{remark}

\begin{remark}[The stochastic Chafee--Infante equation]
    The analysis presented above extends immediately to the stochastic Chafee--Infante equation, another widely studied reaction-diffusion model given by
    \begin{equation*}
    dX(t,\boldsymbol{x}) = \left[ \Delta X(t,\boldsymbol{x}) + \kappa X(t,\boldsymbol{x}) - \gamma X^3(t,\boldsymbol{x}) \right] dt + g(t,\boldsymbol{x},X(t,\boldsymbol{x})) \, dW(t,\boldsymbol{x})
\end{equation*}
    where $\kappa, \gamma > 0$. The linear operator is defined as $A = -\Delta + I$, the associated drift operator $F$ is defined via the Nemytskii operator corresponding to the polynomial $f(y) \coloneqq (\kappa+1)y - \gamma y^3$. Due to the identical cubic polynomial growth, the fractional embedding requirements, the resulting Nemytskii operator bounds, and the time singularity exponent $\hat{\alpha} = 3(\nu-\mu)/2$ remain unchanged. Furthermore, assuming the diffusion coefficient $g$ satisfies the hypothesis of Lemma~\ref{lem:cond_on_diffusion} with $\theta=0$ and $\sigma=0$, the corresponding mapping $G$ satisfies the necessary measurability and local Lipschitz conditions with the identical time singularity exponent $\hat{\beta} = (\nu-\mu)/2$. Consequently, the stochastic Chafee--Infante equation admits a unique maximal local mild solution under the exact same parameter constraints (i)--(vii) established for the stochastic Allen--Cahn equation.
\end{remark}

\subsection{The stochastic Fisher--KPP equation}\label{subsec:SFKPPE}
We consider the stochastic Fisher--KPP (Kolmogorov--Petrovsky--Piskunov) equation:
\begin{equation} \label{eq:SFKPPE} 
\left\{ 
\begin{aligned}     
    dX(t,\boldsymbol{x}) &= \left[\Delta X(t,\boldsymbol{x}) + X(t,\boldsymbol{x}) - X^2(t,\boldsymbol{x})\right] dt \\     
    &\quad + g(t,\boldsymbol{x},X(t,\boldsymbol{x})) \, dW(t,\boldsymbol{x}), && (t,\boldsymbol{x}) \in (0,T] \times \mathcal{O}, \\     
    X(t,\boldsymbol{x}) &= 0, && (t,\boldsymbol{x}) \in (0,T] \times \partial\mathcal{O}, \\     
    X(0,\boldsymbol{x}) &= X_0(\boldsymbol{x}), && \boldsymbol{x} \in \mathcal{O}, 
\end{aligned} 
\right. 
\end{equation}
 where $\mathcal{O} \subset \mathbb{R}^d$, $W$ is a $Q$-Wiener process defined on $\mathbb{H}=L^2(\mathcal{O})$, $X_0 \colon \Omega \to \dot{\mathbb{H}}^\mu$ is an $\mathcal{F}_0$-measurable random variable, and the diffusion coefficient $g \colon [0,T] \times \Omega \times \mathcal{O} \times \mathbb{R} \to \mathbb{R}$ satisfies the hypothesis of Lemma~\ref{lem:cond_on_diffusion} with $\theta=0$ and $\sigma=0$. 
 
 The linear operator $A \coloneqq -\Delta + I$ satisfies identical structural and measurability assumptions detailed in previous Subsection~\ref{subsec:SACE}. For $\nu > 0$, $\alpha \geq 0$, and $2\nu + \alpha \geq d/2$, Lemma~\ref{lem:sobolev_embedding} guarantees the existence of an $r\in(1,2]$ such that $\dot{\mathbb{H}}^{\nu} \hookrightarrow L^{2r}(\mathcal{O}; \mathbb{R})$ and $L^{r}(\mathcal{O}; \mathbb{R})\hookrightarrow \dot{\mathbb{H}}^{-\alpha}$. Defining the drift operator $F \colon [0,T] \times \Omega \times \dot{\mathbb{H}}^\mu \to \dot{\mathbb{H}}^{-\alpha}$ via \eqref{eq:nemytskii_allen_cahn} with $f(y)\coloneqq 2y - y^2$, we verify local Lipschitz continuity by utilizing the factorization $u^2 - v^2 = (u-v)(u+v)$. Applying Hölder's inequality in $L^r(\mathcal{O})$, we obtain:
\begin{align}\label{eq:fisher_estimate_F}
    \nonumber \|F(t,\omega,u) - F(t,\omega,v)\|_{\dot{\mathbb{H}}^{-\alpha}} 
    &\leq \|j_r\|_{\mathcal{L}(L^r(\mathcal{O}); \dot{\mathbb{H}}^{-\alpha})} \Big(2\|u-v\|_{L^r(\mathcal{O})} + \|u+v\|_{L^{2r}(\mathcal{O})} \|u-v\|_{L^{2r}(\mathcal{O})}\Big) \\
    \nonumber &\leq C \Big(\|u-v\|_{\dot{\mathbb{H}}^\nu} + \big(\|u\|_{\dot{\mathbb{H}}^\nu} + \|v\|_{\dot{\mathbb{H}}^\nu}\big)\|u-v\|_{\dot{\mathbb{H}}^\nu}\Big) \\
    &\leq L_1(R) \,t^{-\hat{\alpha}} \|u-v\|_{\lambda,t},
\end{align}
for all $u,v\in\{\phi\in \dot{\mathbb{H}}^{\nu} :\,\|\phi\|_{\lambda,t}\leq R\}$, where $\lambda \coloneqq (\nu-\mu)/2$, the singularity exponent is $\hat{\alpha}\coloneqq 2\lambda = \nu-\mu$, and $L_1(R) \coloneqq C(T^{\lambda} + 2R)$.

Furthermore, since the diffusion coefficient $g \colon [0,T] \times \Omega \times \mathcal{O} \times \mathbb{R} \to \mathbb{R}$ satisfies the hypothesis of Lemma~\ref{lem:cond_on_diffusion} with $\theta=0$ and $\sigma=0$, the corresponding mapping $G \colon [0,T] \times \Omega \times \dot{\mathbb{H}}^\mu \to \mathrm{HS}(U_0, \dot{\mathbb{H}}^{-\beta})$ defined in \eqref{eq:nemytskii_G} satisfies the measurability condition in Assumption~\ref{ass:measurability_drift_diff} and the local Lipschitz condition of Assumption~\ref{ass:local_lip_stoch} with $\hat{\beta} = \lambda = \frac{\nu-\mu}{2}$.

Consequently, applying Theorem~\ref{thm:local_existence}, the stochastic Fisher--KPP equation admits a unique maximal local mild solution under the following parameter constraints:
$$
\begin{aligned}
    &\text{(i)}~ \nu \ge \mu, \quad &&\text{(ii)}~ \nu > 0, \quad &&\text{(iii)}~ \alpha \geq 0, \quad &&\text{(iv)}~ 2\nu + \alpha \geq \frac{d}{2}, \\
    &\text{(v)}~ \nu - \mu < \frac{1}{2}, \quad &&\text{(vi)}~ 3\nu - 2\mu + \alpha < 2, \quad &&\text{(vii)}~ \max\left\{\frac{d}{2},\frac{2d-1}{2} \right\}<\beta < 1 +\mu- 2\nu.
\end{aligned}
$$
If the eigensystem of $Q$ satisfies \eqref{eq:smooth_noise}, constraint (vii) relaxes to $0 \leq \beta < 1+\mu-2\nu$.

In particular, in dimension $d=1$, by fixing $\nu = \epsilon$ for $0 < \epsilon < 1/4$, one can choose the following set of parameter values:

\begin{itemize}
    \item For (weighted) trace-class noise (i.e., the eigensystem $\{q_n, h_n\}_{n \in \mathbb{N}}$ of $Q$ satisfies \eqref{eq:smooth_noise}):
\[
\begin{aligned}
    &\text{(a)}~ -\frac{1}{2} + \epsilon < \mu \le \epsilon, \quad &&\text{(b)}~ \nu = \epsilon, \\
    &\text{(c)}~ \frac{1}{2}-2\epsilon \leq \alpha < 2-3\epsilon + 2\mu, \quad &&\text{(d)}~ 0 \le \beta < 1-2\epsilon + \mu.
\end{aligned}
\]

\item For space-time white noise (i.e., $Q = \mathrm{Id}_{\mathbb{H}}$):
\[
\begin{aligned}
    &\text{(a)}~ -\frac{1}{2} + 2\epsilon < \mu \le \epsilon, \quad &&\text{(b)}~ \nu = \epsilon, \\
    &\text{(c)}~ \frac{1}{2}-2\epsilon \leq \alpha < 2-3\epsilon + 2\mu, \quad &&\text{(d)}~ \frac{1}{2} < \beta < 1-2\epsilon + \mu.
\end{aligned}
\]
\end{itemize}

\begin{remark}
    We note that due to the quadratic nonlinearity of the drift term, our framework accommodates rough initial data ($\mu < 0$) in dimension $d=1$. Taking the limit as $\epsilon \to 0$ in condition (a) yields $\mu > -1/2$ under both (weighted) trace-class noise (where the eigensystem of $Q$ satisfies \eqref{eq:smooth_noise}) and space-time white noise ($Q = \mathrm{Id}_{\mathbb{H}}$).
\end{remark}

\subsection{The stochastic Burgers equation}\label{subsec:SBE}
The deterministic Burgers equation was introduced by Burgers~\cite{MR27195} as a simplified model for fluid turbulence and shock wave formation. Its stochastic counterpart has since been extensively studied in one spatial dimension. Existence and uniqueness for $d=1$ have been established through various approaches, including semigroup methods on bounded domains by Gy\"{o}ngy~\cite{MR1608641}, Da Prato et al.~\cite{MR1300149}, Da Prato and Gatarek~\cite{MR1380259}, and Zaidi and Nualart~\cite{MR1692868}; analysis on the whole real line by Gy\"{o}ngy and Nualart~\cite{MR1698967}; and the Cole--Hopf transformation by Bertini et al.~\cite{MR1301846}. For higher spatial dimensions ($d \ge 2$), Brze\'{z}niak et al.~\cite{MR3166959} established the existence of global strong solutions in $L^p$-spaces ($p>d$). 

We consider the following divergence-form stochastic Burgers equation:
\begin{equation} \label{eq:SBE} 
\left\{ 
\begin{aligned}     
    d\boldsymbol{X}(t,\boldsymbol{x}) &= \big[ \Delta \boldsymbol{X}(t,\boldsymbol{x}) - \nabla \cdot \big(\boldsymbol{X}(t,\boldsymbol{x}) \otimes \boldsymbol{X}(t,\boldsymbol{x})\big) \big] dt \\     
    &\quad + g(t,\boldsymbol{x},\boldsymbol{X}(t,\boldsymbol{x})) \, dW(t,\boldsymbol{x}), && (t,\boldsymbol{x}) \in (0,T] \times \mathcal{O}, \\     
    \boldsymbol{X}(t,\boldsymbol{x}) &= \boldsymbol{0}, && (t,\boldsymbol{x}) \in (0,T] \times \partial\mathcal{O}, \\     
    \boldsymbol{X}(0,\boldsymbol{x}) &= \boldsymbol{X}_0(\boldsymbol{x}), && \boldsymbol{x} \in \mathcal{O}. 
\end{aligned} 
\right. 
\end{equation}
To ensure the tensor product $\boldsymbol{X} \otimes \boldsymbol{X}$ is dimensionally well-defined under the divergence operator, the state space must be vector-valued with a codomain matching the spatial dimension $d$. Thus, the diffusion coefficient $g$ acts on vector fields. More precisely, the diffusion coefficient $g \colon [0,T] \times \Omega \times \mathcal{O} \times \mathbb{R}^d \to \mathbb{R}^d$ is assumed to satisfy the hypothesis of Lemma~\ref{lem:cond_on_diffusion} with $\theta=0$ and $\sigma=0$, $W$ is a $Q$-Wiener process defined on the Hilbert space $\mathbb{H} = L^2(\mathcal{O}; \mathbb{R}^d)$, and $X_0 \colon \Omega \to \dot{\mathbb{H}}^\mu$ is an $\mathcal{F}_0$-measurable random variable. 

The operator $A \coloneqq -\Delta+I$ with the domain $\mathcal{D}(A) = H^2(\mathcal{O}; \mathbb{R}^d) \cap H_0^1(\mathcal{O}; \mathbb{R}^d)$ satisfies Assumption~\ref{ass:operator_A}, and generates a family of fractional Sobolev spaces $\dot{\mathbb{H}}^\gamma \coloneqq \mathcal{D}(A^{\gamma/2})$ for $\gamma \in \mathbb{R}$. Let $\nu > 0$, $\alpha \geq 1$, and $2\nu + \alpha \geq 1 + d/2$. Defining the shifted regularity index $\rho \coloneqq \alpha - 1 \ge 0$, there exists an $r \in (1, 2]$ such that 
$$ \frac{1}{2r} \ge \frac{1}{2} - \frac{\nu}{d} \quad \text{and} \quad \frac{1}{r} \le \frac{1}{2} + \frac{\rho}{d}. $$
Consequently, by applying Lemma~\ref{lem:sobolev_embedding} component-wise, alongside the classical fractional Sobolev embedding theorem, the following continuous embeddings hold:
\begin{equation*}
    \dot{\mathbb{H}}^{\nu} \hookrightarrow L^{2r}(\mathcal{O}; \mathbb{R}^d) \quad \text{and} \quad H^{\rho}(\mathcal{O}; \mathbb{R}^{d \times d}) \hookrightarrow L^{r'}(\mathcal{O}; \mathbb{R}^{d\times d}),
\end{equation*}
where $r'$ is the conjugate exponent satisfying $\frac{1}{r} + \frac{1}{r'} = 1$. Let $f(\boldsymbol{y}) \coloneqq \boldsymbol{y} \otimes \boldsymbol{y}$. Then, for $\mu \le \nu$, the map $F : [0, T] \times \Omega \times \dot{\mathbb{H}}^\mu \rightarrow \dot{\mathbb{H}}^{-\alpha}$ defined by
\begin{equation} \label{eq:burgers_drift}
F(t, \omega, \boldsymbol{u}) \coloneqq 
\begin{cases} 
\boldsymbol{u} - \nabla \cdot \mathcal{N}_f(\boldsymbol{u}) & : \boldsymbol{u} \in \dot{\mathbb{H}}^\nu, \\ 
0 & : \text{otherwise}, 
\end{cases} \quad \forall (t,\omega, \boldsymbol{u} )\in [0,T]\times\Omega\times \dot{\mathbb{H}}^{\mu},
\end{equation}
is a well-defined measurable map from the measurable space $(\Omega_T \times \dot{\mathbb{H}}^\mu, \mathcal{P}_T \otimes \mathcal{B}(\dot{\mathbb{H}}^\mu))$ to $(\dot{\mathbb{H}}^{-\alpha}, \mathcal{B}(\dot{\mathbb{H}}^{-\alpha}))$. Here, $\mathcal{N}_f \colon \dot{\mathbb{H}}^{\nu} \to L^r(\mathcal{O}; \mathbb{R}^{d \times d})$ such that $(\mathcal{N}_f(\boldsymbol{u}))(\boldsymbol{x}) = f(\boldsymbol{u}(\boldsymbol{x}))$ denotes the associated Nemytskii operator. The divergence operator $\nabla \cdot \colon L^r(\mathcal{O}; \mathbb{R}^{d \times d}) \rightarrow \dot{\mathbb{H}}^{-\alpha}$ is bounded. To see this, we define its action via duality for any $\Phi \in L^r(\mathcal{O}; \mathbb{R}^{d \times d})$ and any smooth, compactly supported test function $\boldsymbol{v} \in C_c^\infty(\mathcal{O}; \mathbb{R}^d)$ as follows:
\begin{equation}\label{eq:divergence_duality}
    \langle \nabla \cdot \Phi, \boldsymbol{v} \rangle_{(C_c^\infty)^\ast, C_c^\infty} \coloneqq - \int_{\mathcal{O}} \Phi : \nabla \boldsymbol{v} \, d\boldsymbol{x}. 
\end{equation}
Assuming $1 \leq \alpha < 2-\nu$, and recalling $\rho = \alpha -1$, we apply H\"{o}lder's inequality and the classical fractional Sobolev embedding $H^\rho(\mathcal{O}; \mathbb{R}^{d \times d}) \hookrightarrow L^{r'}(\mathcal{O}; \mathbb{R}^{d \times d})$. Noting that the gradient operator $\nabla$ maps $H^\alpha(\mathcal{O}; \mathbb{R}^d)$ continuously into $H^\rho(\mathcal{O}; \mathbb{R}^{d \times d})$, we obtain:
\begin{align}\label{eq:divergence_bound}
    \nonumber | \langle \nabla \cdot \Phi, \boldsymbol{v} \rangle_{(C_c^\infty)^\ast, C_c^\infty}| 
    &\le \|\Phi\|_{L^r(\mathcal{O}; \mathbb{R}^{d \times d})} \|\nabla \boldsymbol{v}\|_{L^{r'}(\mathcal{O}; \mathbb{R}^{d \times d})} \\
    \nonumber &\le C \|\Phi\|_{L^r(\mathcal{O}; \mathbb{R}^{d \times d})} \|\nabla \boldsymbol{v}\|_{H^{\rho}(\mathcal{O}; \mathbb{R}^{d \times d})} \\
    \nonumber &\le C \|\Phi\|_{L^r(\mathcal{O}; \mathbb{R}^{d \times d})} \|\boldsymbol{v}\|_{H^{\alpha}(\mathcal{O}; \mathbb{R}^d)} \\
    &\le C \|\Phi\|_{L^r(\mathcal{O}; \mathbb{R}^{d \times d})} \|\boldsymbol{v}\|_{\dot{\mathbb{H}}^\alpha}.
\end{align}
Inequality \eqref{eq:divergence_bound} establishes the bound
\begin{equation}\label{eq:divergence_operator_bound}
    \|\nabla \cdot \Phi\|_{\dot{\mathbb{H}}^{-\alpha}} \le C \|\Phi\|_{L^r(\mathcal{O}; \mathbb{R}^{d \times d})}.
\end{equation}
To estimate the locally Lipschitz difference, we apply \eqref{eq:divergence_operator_bound} to $\Phi = \boldsymbol{w} \otimes \boldsymbol{z}$ for any $\boldsymbol{w}, \boldsymbol{z} \in \dot{\mathbb{H}}^\nu$. Utilizing the Cauchy-Schwarz inequality alongside the continuous embedding $\dot{\mathbb{H}}^\nu \hookrightarrow L^{2r}(\mathcal{O}; \mathbb{R}^d)$, we obtain 
\begin{align}\label{eq:bilinear_estimate}
    \nonumber \|\nabla \cdot (\boldsymbol{w} \otimes \boldsymbol{z})\|_{\dot{\mathbb{H}}^{-\alpha}} &\le C \|\boldsymbol{w} \otimes \boldsymbol{z}\|_{L^r(\mathcal{O}; \mathbb{R}^{d \times d})} \\
    \nonumber &\le C \|\boldsymbol{w}\|_{L^{2r}(\mathcal{O}; \mathbb{R}^d)} \|\boldsymbol{z}\|_{L^{2r}(\mathcal{O}; \mathbb{R}^d)} \\
    &\le M \|\boldsymbol{w}\|_{\dot{\mathbb{H}}^\nu} \|\boldsymbol{z}\|_{\dot{\mathbb{H}}^\nu}.
\end{align}
Hence, for any given $R>0$, $(t,\omega)\in(0,T]\times\Omega$, and $\boldsymbol{u},\boldsymbol{v}\in\{\boldsymbol{\phi}\in \dot{\mathbb{H}}^{\nu} :\,\|\boldsymbol{\phi}\|_{\lambda,t}\leq R,\;\lambda\coloneqq (\nu-\mu)/2\}$, we apply \eqref{eq:bilinear_estimate} to deduce:
\begin{align}\label{eq:burgers_estimate_F}
    \nonumber &\|F(t,\omega,\boldsymbol{u}) - F(t,\omega,\boldsymbol{v})\|_{\dot{\mathbb{H}}^{-\alpha}} \\
    \nonumber &\quad \leq \|\boldsymbol{u}-\boldsymbol{v}\|_{\dot{\mathbb{H}}^{-\alpha}} + \|\nabla \cdot (\boldsymbol{u} \otimes \boldsymbol{u}) - \nabla \cdot (\boldsymbol{v} \otimes \boldsymbol{v})\|_{\dot{\mathbb{H}}^{-\alpha}} \\
    \nonumber &\quad \leq \|\boldsymbol{u}-\boldsymbol{v}\|_{\dot{\mathbb{H}}^{-\alpha}} +  \|\nabla \cdot \big(\boldsymbol{u} \otimes (\boldsymbol{u}-\boldsymbol{v})\big)\|_{\dot{\mathbb{H}}^{-\alpha}} + \|\nabla \cdot \big((\boldsymbol{u}-\boldsymbol{v}) \otimes \boldsymbol{v}\big)\|_{\dot{\mathbb{H}}^{-\alpha}} \\
    \nonumber &\quad \leq \|\boldsymbol{u}-\boldsymbol{v}\|_{\dot{\mathbb{H}}^{-\alpha}} +  M \|\boldsymbol{u}\|_{\dot{\mathbb{H}}^\nu} \|\boldsymbol{u}-\boldsymbol{v}\|_{\dot{\mathbb{H}}^\nu} + M \|\boldsymbol{u}-\boldsymbol{v}\|_{\dot{\mathbb{H}}^\nu} \|\boldsymbol{v}\|_{\dot{\mathbb{H}}^\nu} \\
    &\quad \leq L_1(R) \,t^{-\hat{\alpha}} \|\boldsymbol{u}-\boldsymbol{v}\|_{\lambda,t},
\end{align}
where $L_1(R) \coloneqq C T^{\lambda} + 2MR$ and $\hat{\alpha}\coloneqq 2\lambda = \nu-\mu$. Note that $\|F(t,\omega,\boldsymbol{0}) \|_{\dot{\mathbb{H}}^{-\alpha}} = 0$ for all $(t,\omega)\in(0,T]\times\Omega$. Therefore, with $L_2(R) \coloneqq 0$, the drift operator satisfies Assumption~\ref{ass:local_lip_stoch}.

Furthermore, since the diffusion coefficient $g \colon [0,T] \times \Omega \times \mathcal{O} \times \mathbb{R}^d \to \mathbb{R}^d$ satisfies the hypothesis of Lemma~\ref{lem:cond_on_diffusion} with $\theta=0$ and $\sigma=0$, the corresponding mapping $G \colon [0,T] \times \Omega \times \dot{\mathbb{H}}^\mu \to \mathrm{HS}(U_0, \dot{\mathbb{H}}^{-\beta})$ defined in \eqref{eq:nemytskii_G} satisfies the measurability condition in Assumption~\ref{ass:measurability_drift_diff} and the local Lipschitz condition of Assumption~\ref{ass:local_lip_stoch} with $\hat{\beta} = \lambda = \frac{\nu-\mu}{2}$.

An application of Theorem~\ref{thm:local_existence} guarantees the existence of a unique maximal local mild solution to the divergence-form stochastic Burgers equation \eqref{eq:SBE} under the following parameter constraints:
$$
\begin{aligned}
    &\text{(i)}~ \nu \ge \mu, \quad &&\text{(ii)}~ \nu > 0, \quad &&\text{(iii)}~ \alpha \ge 1, \quad &&\text{(iv)}~ 2\nu + \alpha \geq 1 + \frac{d}{2}, \\
    &\text{(v)}~ \nu - \mu < \frac{1}{2}, \quad &&\text{(vi)}~ 3\nu - 2\mu + \alpha < 2, \quad &&\text{(vii)}~ \max\left\{\frac{d}{2},\frac{2d-1}{2} \right\} <\beta < 1 +\mu- 2\nu.
\end{aligned}
$$
If the eigensystem of $Q$ satisfies \eqref{eq:smooth_noise}, constraint (vii) relaxes to $0 \leq \beta < 1+\mu-2\nu$. 

In particular, in dimension $d=1$, by fixing $\nu = \epsilon$ for $0 < \epsilon < 1/6$, one can choose the following set of parameter values:

\begin{itemize}
    \item For (weighted) trace-class noise (i.e., the eigensystem $\{q_n, h_n\}_{n \in \mathbb{N}}$ of $Q$ satisfies \eqref{eq:smooth_noise}):
\[
\begin{aligned}
    &\text{(a)}~ -\frac{1}{4} + \frac{\epsilon}{2} < \mu \le \epsilon, \quad &&\text{(b)}~ \nu = \epsilon, \\
    &\text{(c)}~ \frac{3}{2}-2\epsilon \leq \alpha < 2-3\epsilon + 2\mu, \quad &&\text{(d)}~ 0 \le \beta < 1-2\epsilon + \mu.
\end{aligned}
\]

\item For space-time white noise (i.e., $Q = \mathrm{Id}_{\mathbb{H}}$):
\[
\begin{aligned}
    &\text{(a)}~ -\frac{1}{4} + \frac{\epsilon}{2} < \mu \le \epsilon, \quad &&\text{(b)}~ \nu = \epsilon, \\
    &\text{(c)}~ \frac{3}{2}-2\epsilon \leq \alpha < 2-3\epsilon + 2\mu, \quad &&\text{(d)}~ \frac{1}{2} < \beta < 1-2\epsilon + \mu.
\end{aligned}
\]
\end{itemize}

\begin{remark} 
    A notable feature of the derived parameter constraints is that in dimension $d=1$, the framework successfully accommodates rough initial data ($\mu < 0$). Taking the limit as $\epsilon \to 0$ in condition (a) yields $\mu > -1/4$ under both space-time white noise ($Q = \mathrm{Id}_{\mathbb{H}}$) and (weighted) trace-class noise (where the eigensystem of $Q$ satisfies \eqref{eq:smooth_noise}).
\end{remark}

\subsection{The stochastic Burgers--Fisher equation}\label{subsec:SBFE}
We consider the following divergence-form stochastic Burgers--Fisher equation:
\begin{equation} \label{eq:SBFE} 
\left\{ 
\begin{aligned}     
    d\boldsymbol{X}(t,\boldsymbol{x}) &= \big[ \Delta \boldsymbol{X}(t,\boldsymbol{x}) + \boldsymbol{X}(t,\boldsymbol{x}) - \boldsymbol{X}^2(t,\boldsymbol{x}) - \nabla \cdot \big(\boldsymbol{X}(t,\boldsymbol{x}) \otimes \boldsymbol{X}(t,\boldsymbol{x})\big) \big] dt \\     
    &\quad + g(t,\boldsymbol{x},\boldsymbol{X}(t,\boldsymbol{x})) \, dW(t,\boldsymbol{x}), && (t,\boldsymbol{x}) \in (0,T] \times \mathcal{O}, \\     
    \boldsymbol{X}(t,\boldsymbol{x}) &= \boldsymbol{0}, && (t,\boldsymbol{x}) \in (0,T] \times \partial\mathcal{O}, \\     
    \boldsymbol{X}(0,\boldsymbol{x}) &= \boldsymbol{X}_0(\boldsymbol{x}), && \boldsymbol{x} \in \mathcal{O}, 
\end{aligned} 
\right. 
\end{equation}
where $W$ is a $Q$-Wiener process defined on the Hilbert space $\mathbb{H} = L^2(\mathcal{O}; \mathbb{R}^d)$, $X_0 \colon \Omega \to \dot{\mathbb{H}}^\mu$ is an $\mathcal{F}_0$-measurable random variable, $X^2$ denotes the component-wise square vector field, and the diffusion coefficient $g \colon [0,T] \times \Omega \times \mathcal{O} \times \mathbb{R}^d \to \mathbb{R}^d$ is assumed to satisfy the hypothesis of Lemma~\ref{lem:cond_on_diffusion} with $\theta=0$ and $\sigma=0$. 

The operator $A \coloneqq -\Delta+I$ with the domain $\mathcal{D}(A) = H^2(\mathcal{O}; \mathbb{R}^d) \cap H_0^1(\mathcal{O}; \mathbb{R}^d)$ satisfies Assumption~\ref{ass:operator_A}. Let $f_1(\boldsymbol{y}) \coloneqq \boldsymbol{y} - \boldsymbol{y}^2$ and $f_2(\boldsymbol{y}) \coloneqq \boldsymbol{y} \otimes \boldsymbol{y}$. For $\mu \le \nu$, we construct the drift operator $F \colon [0,T] \times \Omega \times \dot{\mathbb{H}}^\mu \to \dot{\mathbb{H}}^{-\alpha}$ as the sum $F \coloneqq F_F + F_B$, where the reaction and convective operators are respectively defined by
\begin{align}
    F_F(t,\omega,\boldsymbol{u}) &\coloneqq \begin{cases}
        j_r\mathcal{N}_{f_1}(\boldsymbol{u}) &: \boldsymbol{u}\in \dot{\mathbb{H}}^{\nu},\\
        0 &: \text{otherwise},
    \end{cases} \label{eq:drift_fisher_part} \\
    F_B(t,\omega,\boldsymbol{u}) &\coloneqq \begin{cases}
        \boldsymbol{u} - \nabla \cdot \mathcal{N}_{f_2}(\boldsymbol{u}) &: \boldsymbol{u}\in \dot{\mathbb{H}}^{\nu},\\
        0 &: \text{otherwise},
    \end{cases} \label{eq:drift_burgers_part}
\end{align}
for all $(t,\omega, \boldsymbol{u}) \in [0,T]\times\Omega\times \dot{\mathbb{H}}^{\mu}$. Applying the triangle inequality in $\dot{\mathbb{H}}^{-\alpha}$ and using \eqref{eq:fisher_estimate_F} and \eqref{eq:burgers_estimate_F}, we obtain
\begin{align}
    \nonumber \|F(t,\omega,\boldsymbol{u}) - F(t,\omega,\boldsymbol{v})\|_{\dot{\mathbb{H}}^{-\alpha}} 
    &\le \|F_F(t,\omega,\boldsymbol{u}) - F_F(t,\omega,\boldsymbol{v})\|_{\dot{\mathbb{H}}^{-\alpha}} + \|F_B(t,\omega,\boldsymbol{u}) - F_B(t,\omega,\boldsymbol{v})\|_{\dot{\mathbb{H}}^{-\alpha}} \\
    \nonumber &\le L_1^F(R) \,t^{-\hat{\alpha}} \|\boldsymbol{u}-\boldsymbol{v}\|_{\lambda,t} + L_1^B(R) \,t^{-\hat{\alpha}} \|\boldsymbol{u}-\boldsymbol{v}\|_{\lambda,t} \\
    &= L_1(R) \,t^{-\hat{\alpha}} \|\boldsymbol{u}-\boldsymbol{v}\|_{\lambda,t},
\end{align}
where $L_1(R) \coloneqq L_1^F(R) + L_1^B(R)$ and $\hat{\alpha}\coloneqq 2\lambda = \nu-\mu$. Furthermore, evaluating $F$ at zero yields $\|F(t,\omega,\boldsymbol{0})\|_{\dot{\mathbb{H}}^{-\alpha}} \le \|F_F(t,\omega,\boldsymbol{0})\|_{\dot{\mathbb{H}}^{-\alpha}} + \|F_B(t,\omega,\boldsymbol{0})\|_{\dot{\mathbb{H}}^{-\alpha}} = 0$. Thus, with $L_2(R) \coloneqq 0$, $F$ satisfies Assumption~\ref{ass:local_lip_stoch}.

Additionally, since the diffusion coefficient $g \colon [0,T] \times \Omega \times \mathcal{O} \times \mathbb{R}^d \to \mathbb{R}^d$ satisfies the hypothesis of Lemma~\ref{lem:cond_on_diffusion} with $\theta=0$ and $\sigma=0$, the corresponding mapping $G \colon [0,T] \times \Omega \times \dot{\mathbb{H}}^\mu \to \mathrm{HS}(U_0, \dot{\mathbb{H}}^{-\beta})$ defined in \eqref{eq:nemytskii_G} satisfies the measurability condition in Assumption~\ref{ass:measurability_drift_diff} and the local Lipschitz condition of Assumption~\ref{ass:local_lip_stoch} with $\hat{\beta} = \lambda = \frac{\nu-\mu}{2}$.

Note that the reaction term $F_F$ requires $\alpha \ge 0$ and $2\nu + \alpha \ge d/2$, whereas the divergence-form convective term $F_B$ imposes stronger constraints $\alpha \ge 1$ and $2\nu + \alpha \ge 1 + d/2$. Hence, the convective constraints dominate. By Theorem~\ref{thm:local_existence}, the divergence-form stochastic Burgers--Fisher equation \eqref{eq:SBFE} admits a unique maximal local mild solution under the exact same parameters established for the stochastic Burgers equation in Subsection~\ref{subsec:SBE}.

\subsection{The stochastic Ginzburg--Landau system}\label{subsec:Ginzburg_Landau_system}
The stochastic Ginzburg--Landau system, serving as the vector-valued extension of the scalar Allen--Cahn equation considered in Subsection~\ref{subsec:SACE}, is a fundamental framework for studying coupled phase transitions and infinite-dimensional dynamical systems. The existence and regularity of solutions for the deterministic Ginzburg--Landau system were studied by Ginibre and Velo~\cite{MR1406282} using compactness methods on unbounded domains, an approach subsequently extended to the stochastic setting by Rougemont~\cite{MR1889231}. Further advancing the theory on unbounded domains, Eckmann and Hairer~\cite{MR1808628} established the existence of invariant measures. More recently, Lin and Gao~\cite{MR3908922} proved existence and uniqueness for the generalized Ginzburg--Landau equation driven by jump noise on bounded domains.

We consider the following stochastic Ginzburg--Landau system with a locally Lipschitz diffusion coefficient on bounded domains:
\begin{equation} \label{eq:SGL_system} 
\left\{ 
\begin{aligned}     
    d\boldsymbol{X}(t,\boldsymbol{x}) &= \left[ \Delta \boldsymbol{X}(t,\boldsymbol{x}) + \boldsymbol{X}(t,\boldsymbol{x}) - |\boldsymbol{X}(t,\boldsymbol{x})|^2\boldsymbol{X}(t,\boldsymbol{x}) \right] dt \\     
    &\quad + g(t,\boldsymbol{x},\boldsymbol{X}(t,\boldsymbol{x})) \, dW(t,\boldsymbol{x}), && (t,\boldsymbol{x}) \in (0,T] \times \mathcal{O}, \\     
    \boldsymbol{X}(t,\boldsymbol{x}) &= \boldsymbol{0}, && (t,\boldsymbol{x}) \in (0,T] \times \partial\mathcal{O}, \\     
    \boldsymbol{X}(0,\boldsymbol{x}) &= \boldsymbol{X}_0(\boldsymbol{x}), && \boldsymbol{x} \in \mathcal{O}, 
\end{aligned} 
\right. 
\end{equation}
where $\mathcal{O} \subset \mathbb{R}^d$ possesses the $C^2$-extension property, $W$ is a $Q$-Wiener process defined on $\mathbb{H}=L^2(\mathcal{O}; \mathbb{R}^N)$, $\boldsymbol{X}_0 \colon \Omega \to \dot{\mathbb{H}}^\mu$ is an $\mathcal{F}_0$-measurable random variable, and the diffusion coefficient $g \colon [0,T] \times \Omega \times \mathcal{O} \times \mathbb{R}^d \to \mathbb{R}^d$ is assumed to satisfy the hypothesis of Lemma~\ref{lem:cond_on_diffusion} with $\theta=1$ and $\sigma=0$. 

The operator $A \coloneqq -\Delta + I$ with the domain $\mathcal{D}(A) = H^2(\mathcal{O}; \mathbb{R}^N) \cap H_0^1(\mathcal{O}; \mathbb{R}^N)$ satisfies Assumption~\ref{ass:operator_A}, and generates the fractional Sobolev spaces $\dot{\mathbb{H}}^\gamma \coloneqq \mathcal{D}(A^{\gamma/2})$. Let 
\begin{align}
    & \frac{d}{6}<\nu,\;0\leq \alpha,\;\text{ and }\;\alpha+3\nu \geq d.
\end{align}
Then there exists an $r\in(1,2]$ such that
\begin{align}
     & \frac{1}{3r}\geq \frac{1}{2} - \frac{\nu}{d}\;\text{ and }\;\frac{1}{r}\leq \frac{1}{2} + \frac{\alpha}{d}.
\end{align}
Consequently, by Lemma~\ref{lem:sobolev_embedding} applied component-wise, we have the continuous embeddings $\dot{\mathbb{H}}^{\nu} \hookrightarrow L^{3r}(\mathcal{O}; \mathbb{R}^N)$ and $L^{r}(\mathcal{O}; \mathbb{R}^N)\hookrightarrow \dot{\mathbb{H}}^{-\alpha}$. 

Let $f(\boldsymbol{y})\coloneqq \boldsymbol{y} - |\boldsymbol{y}|^2\boldsymbol{y}$. Then, for $\mu\leq \nu$, the drift mapping $F:[0,T]\times\Omega\times \dot{\mathbb{H}}^{\mu}\to \dot{\mathbb{H}}^{-\alpha}$ defined by
\begin{equation}\label{eq:nemytskii_allen_cahn_system}
     F(t,\omega,\boldsymbol{u}) \coloneqq \begin{cases}
        j_r\mathcal{N}_f(\boldsymbol{u}) &: \boldsymbol{u}\in \dot{\mathbb{H}}^{\nu},\\
        0 &: \text{otherwise},
    \end{cases} \quad \forall (t,\omega, \boldsymbol{u} )\in [0,T]\times\Omega\times \dot{\mathbb{H}}^{\mu},
\end{equation}
is a well-defined measurable map. Here, $j_r \colon L^r(\mathcal{O}; \mathbb{R}^N) \hookrightarrow \dot{\mathbb{H}}^{-\alpha}$ is the canonical embedding and $\mathcal{N}_f$ is the associated Nemytskii operator. Note that $\|F(t,\omega,\boldsymbol{0}) \|_{\dot{\mathbb{H}}^{-\alpha}} = 0$. By applying generalized Hölder's inequality and the embeddings, for any given $R>0$, $(t,\omega)\in(0,T]\times\Omega$, and $\boldsymbol{u},\boldsymbol{v}\in\{\boldsymbol{\phi}\in  \dot{\mathbb{H}}^{\nu} :\,\|\boldsymbol{\phi}\|_{\lambda,t}\leq R,\;\lambda\coloneqq (\nu-\mu)/2\}$, we bound the locally Lipschitz difference:
\begin{align}
    \nonumber \|F(t,\omega,\boldsymbol{u}) - F(t,\omega,\boldsymbol{v})\|_{\dot{\mathbb{H}}^{-\alpha}}& \leq \|j_r\| \|\mathcal{N}_f(\boldsymbol{u}) - \mathcal{N}_f(\boldsymbol{v})\|_{L^r(\mathcal{O}; \mathbb{R}^N)}\\
    \nonumber& \leq \|j_r\| \Big(\|\boldsymbol{u}-\boldsymbol{v}\|_{L^r} + C \big\| |\boldsymbol{u}|^2 + |\boldsymbol{v}|^2 \big\|_{L^{3r/2}} \|\boldsymbol{u}-\boldsymbol{v}\|_{L^{3r}}\Big)\\
\nonumber& \leq C \Big(\|\boldsymbol{u}-\boldsymbol{v}\|_{\dot{\mathbb{H}}^\nu} + \big(\| \boldsymbol{u}\|_{\dot{\mathbb{H}}^\nu}^2 +\|\boldsymbol{v}\|_{\dot{\mathbb{H}}^\nu}^2\big) \|\boldsymbol{u}-\boldsymbol{v}\|_{\dot{\mathbb{H}}^\nu}\Big)\\
&\leq L_1(R) \,t^{-\hat{\alpha}} \|\boldsymbol{u}-\boldsymbol{v}\|_{\lambda,t},
\end{align}
where $L_1(R) \coloneqq C(T^{2\lambda} + 2R^2)$ and the drift singularity index is $\hat{\alpha}\coloneqq 3\lambda = \frac{3(\nu-\mu)}{2}$.

Now, since the diffusion coefficient \(g\) satisfies the hypothesis of Lemma~\ref{lem:cond_on_diffusion} for $\theta=1$ and $\sigma=0$, the corresponding diffusion mapping $G$ defined in \eqref{eq:nemytskii_G} satisfies the measurability condition in Assumption~\ref{ass:measurability_drift_diff} and the local Lipschitz condition of Assumption~\ref{ass:local_lip_stoch} with $\hat{\beta} \coloneqq \sigma + (\theta+1)\lambda = 2\lambda = \nu - \mu$. 

According to Lemma~\ref{lem:cond_on_diffusion}, setting $\theta=1$ imposes the condition $\nu \ge \frac{d}{4}$. Because this dominates the drift term constraint ($\nu > \frac{d}{6}$), an appeal to Theorem~\ref{thm:local_existence} guarantees the existence of a unique maximal local mild solution to the system \eqref{eq:SGL_system} under the following constraints:
\[
\begin{aligned}
    &\text{(i)}~ \nu \ge \mu, \quad &&\text{(ii)}~ \nu \ge \frac{d}{4}, \quad &&\text{(iii)}~ \alpha \geq 0, \quad &&\text{(iv)}~ 3\nu + \alpha \geq d, \\
    &\text{(v)}~ \nu - \mu < \frac{1}{3}, \quad &&\text{(vi)}~ 4\nu - 3\mu + \alpha < 2, \quad &&\text{(vii)}~ \max\left\{\frac{d}{2},\frac{2d-1}{2} \right\} < \beta < 1 - 3\nu + 2\mu.
\end{aligned}
\]
In particular, in dimension $d=1$, by fixing $\nu = 1/4 + \epsilon$ for $0 < \epsilon < 1/12$, one can choose  the following set of parameter values:

\begin{itemize}
    \item For (weighted) trace-class noise (i.e., the eigensystem $\{q_n, h_n\}_{n \in \mathbb{N}}$ of $Q$ satisfies \eqref{eq:smooth_noise}):
\[
\begin{aligned}
    &\text{(a)}~ -\frac{1}{12} + \epsilon < \mu \le \frac{1}{4}+\epsilon, \quad &&\text{(b)}~ \nu = \frac{1}{4}+\epsilon, \\
    &\text{(c)}~ \frac{1}{4}-3\epsilon \leq \alpha < 1-4\epsilon + 3\mu, \quad &&\text{(d)}~ 0 \le \beta < \frac{1}{4}-3\epsilon + 2\mu.
\end{aligned}
\]

\item For space-time white noise (i.e., $Q = \mathrm{Id}_{\mathbb{H}}$):
\[
\begin{aligned}
    &\text{(a)}~ \frac{1}{8} + \frac{3\epsilon}{2} < \mu \le \frac{1}{4}+\epsilon, \quad &&\text{(b)}~ \nu = \frac{1}{4}+\epsilon, \\
    &\text{(c)}~ \frac{1}{4}-3\epsilon \leq \alpha < 1-4\epsilon + 3\mu, \quad &&\text{(d)}~ \frac{1}{2} < \beta < \frac{1}{4}-3\epsilon + 2\mu.
\end{aligned}
\]
\end{itemize}

\begin{remark}
    In dimension $d=1$, taking the limit as $\epsilon \to 0$ in condition (a) yields $\mu > -1/12$. However, when the system is driven by space-time white noise ($Q = \mathrm{Id}_{\mathbb{H}}$), we obtain $\mu > 1/8$. 
\end{remark}



\vspace{1em} 

\noindent \textbf{Acknowledgments} \,\,The authors acknowledge the support provided by the Indian Institute of Technology Goa, India.

\addcontentsline{toc}{section}{References}

\bibliographystyle{abbrv} 
\bibliography{references.bib}

\appendix

\end{document}